\documentclass[a4paper,11pt]{article}
\usepackage{graphicx}
\usepackage{amsmath,amsthm,amssymb,enumerate}
\usepackage{euscript,mathrsfs,amsfonts,dsfont}
\usepackage{xcolor}
\usepackage{empheq}
\usepackage{geometry}

\usepackage{bm}
\usepackage[linkcolor=blue,colorlinks=true]{hyperref}
\catcode`\@=11 \@addtoreset{equation}{section}

\catcode`\@=12

\usepackage{stmaryrd}
\allowdisplaybreaks

\usepackage[labelformat=simple]{subcaption}

\usepackage{enumitem}
\usepackage{url}

\usepackage{tikz}
\usepackage[framemethod=tikz]{mdframed}
\usetikzlibrary{patterns}
\usepackage[normalem]{ulem}
\usepackage{cancel}

\newtheorem{Theorem}{Theorem}[section]

\newtheorem{Lemma}[Theorem]{Lemma}

\newtheorem{Definition}[Theorem]{Definition}

\newtheorem{Remark}[Theorem]{Remark}

\newcommand{\vx}{x}
\newcommand{\vrh}{\vr_h}
\newcommand{\vmh}{\vm_h}
\newcommand{\vuh}{\bm{u}_h}
\newcommand{\vu}{\bm{u}}

\newcommand{\uih}{u_{i,h}}

\newcommand{\mh}{\vc{m}_h}

\newcommand{\br}{ \nonumber \\ }

\newcommand{\intSh}[1] {\int_{\sigma} #1 \ds }
\newcommand{\intDh}[1] {\int_{D_\sigma} #1 \dx }

\newcommand{\bbE}{ \mathbb{E}}
\newcommand{\RE}{ R_E}

\DeclareMathOperator{\dist}{dist}

\newcommand{\tor}{\mathbb{T}^d}

\newcommand{\bfv}{\bm{v}}

\newcommand{\tvm}{\widetilde{\vc{m}}}
\newcommand{\bfphi}{\boldsymbol{\phi}}

\newcommand{\bfvarphi}{\boldsymbol{\varphi}}

\newcommand{\Pim}{\Pi_\mathcal{T}}

\newcommand{\Pie}{\Pi_\faces}
\newcommand{\Piei}{\Pi_\faces^{(i)}}
\newcommand{\Pid}{\Pi_\ep}

\newcommand{\sumi}{\sum_{i=1}^d}

\newcommand{\sumKf}{{\sum_{K \in \grid} |K| \sum_{\sigma \in \facesK}}}

\newcommand{\ds}{\,{\rm d}S_x}

\newcommand{\grid}{{\cal T}_h}

\newcommand{\TS}{\Delta t}

\newcommand{\Div}{{\rm div}_x}
\newcommand{\Grad}{\nabla_x}
\newcommand{\Divh}{{\rm div}_h}
\newcommand{\Gradh}{\nabla_h}
\newcommand{\Lap}{\Delta_x}
\newcommand{\Laph}{\Delta_h}
\newcommand{\Laphi}{\Delta_h^{(i)}}
\newcommand{\Gradd}{\nabla_\faces}

\newcommand{\Divmesh}{{\rm div}_{\mathcal{T}}}

\newcommand{\pdedgei}{\eth_ \faces^{(i)}}

\newcommand{\pdmeshi}{\eth_{\cal T}^{(i)}}

\newcommand{\co}[2]{{\rm co}\{ #1 , #2 \}}
\newcommand{\Ov}[1]{\overline{ #1 } }
\newcommand{\avs}[1]{ \left\{\hspace{-3pt}\left\{ #1 \right\}\hspace{-3pt} \right\} }
\newcommand{\avsi}[1]{\avs{#1}^{(i)}}

\newcommand{\aleq}{\stackrel{<}{\sim}}
\newcommand{\ageq}{\stackrel{>}{\sim}}
\newcommand{\Un}[1]{\underline{#1}}
\newcommand{\vr}{\varrho}

\newcommand{\tvr}{\widetilde \vr}
\newcommand{\tvu}{{\widetilde \vu}}

\newcommand{\vm}{\bm{m}}
\newcommand{\ve}{\bm{e}}

\newcommand{\vn}{\bm{n}}

\newcommand{\vc}[1]{{\bm #1}}

\newcommand{\dx}{\,{\rm d} {x}}

\newcommand{\dt}{{\rm d} t }

\newcommand{\jump}[1]{\left\llbracket#1\right\rrbracket}
\newcommand{\jumpi}[1]{\jump{#1}^{(i)}}
\newcommand{\Bigabs}[1]{\Big| #1\Big|}
\newcommand{\abs}[1]{{\left| #1 \right|}}

\newcommand{\norm}[1]{\left\lVert#1\right\rVert}

\newcommand{\vU}{\vc{U}}

\newcommand{\dxdt}{\dx \dt}
\newcommand{\intTd}[1]{\int_{\tor} #1 \dx}
\newcommand{\intTdB}[1]{\int_{\tor}\left( #1 \right)\dx}
\newcommand{\intOs}[1]{\int_{\Os} #1 \dx}
\newcommand{\intOf}[1]{\int_{\Of} #1 \dx}
\newcommand{\intOfB}[1]{\int_{\Of} \left( #1 \right) \dx}
\newcommand{\intOsh}[1]{\int_{\Osh} #1 \dx}
\newcommand{\intOfh}[1]{\int_{\Ofh} #1 \dx}

\newcommand{\intfacesC}[1]{\int_{\facesC}{ #1 \ds}}
\newcommand{\intfacesIO}[1]{\int_{\faces\setminus\facesC}{ #1 \ds}}
\newcommand{\intfaces}[1]{\int_{\faces}{ #1 \ds}}

\newcommand{\Of}{\Omega^f}
\newcommand{\Os}{\Omega^s}
\newcommand{\Oc}{\Omega^C}
\newcommand{\Och}{\Omega^C_h}
\newcommand{\Ochm}{\Omega^{C-}_h}
\newcommand{\Ochp}{\Omega^{C+}_h}
\newcommand{\OI}{\Omega^I}
\newcommand{\OO}{\Omega^O}
\newcommand{\Osh}{\Os_h}
\newcommand{\Ofh}{\Of_h}

\newcommand{\intOch}[1]{\int_{\Oc_h}  #1 \dx}

\newcommand{\intOOh}[1]{\int_{\OO_h}  #1 \dx}
\newcommand{\intOIh}[1]{\int_{\OI_h}  #1 \dx}

\newcommand{\intn}{\int_{0}^{t_{n+1}}}
\newcommand{\intTau}{\int_0^\tau}
\newcommand{\intTauTd}[1]{\int_0^\tau \int_{\tor}  #1  \dxdt}
\newcommand{\intTauTdB}[1]{\int_0^\tau \int_{\tor} \left( #1 \right) \dxdt}

\newcommand{\intTauOch}[1]{\int_0^\tau \int_{\Och}  #1  \dxdt}
\newcommand{\intTauOchp}[1]{\int_0^\tau \int_{\Ochp}  #1  \dxdt}
\newcommand{\intTauOchm}[1]{\int_0^\tau \int_{\Ochm} #1  \dxdt}
\newcommand{\intTauOOh}[1]{\int_0^\tau \int_{\OO_h}  #1  \dxdt}
\newcommand{\intTauOIh}[1]{\int_0^\tau \int_{\OI_h}  #1  \dxdt}

\newcommand{\intTauOf}[1]{\int_0^\tau \int_{\Of}   #1   \dxdt}
\newcommand{\intTauOfh}[1]{\int_0^\tau \int_{\Ofh}   #1   \dxdt}
\newcommand{\intTauOs}[1]{\int_0^\tau \int_{\Os}   #1   \dxdt}
\newcommand{\intTauOsh}[1]{\int_0^\tau \int_{\Osh}  #1   \dxdt}

\newcommand{\RO}{\tor}

\newcommand{\intTauOB}[1]{ \int_0^\tau \int_{\RO} \left( #1 \right) \ \dxdt}

\newcommand{\vv}{\bm{v}}

\newcommand{\ep}{\varepsilon}

\newcommand{\R}{\mathbb{R}}
\newcommand{\I}{\mathbb{I}}

\def\softd{{\leavevmode\setbox1=\hbox{d}%
          \hbox to 1.05\wd1{d\kern-0.4ex{\char039}\hss}}}%custos

\definecolor{Cgrey}{rgb}{0.85,0.85,0.85}
\definecolor{Cblue}{rgb}{0.50,0.85,0.85}
\definecolor{Cred}{rgb}{1,0,0}
\definecolor{fancy}{rgb}{0.10,0.85,0.10}
\definecolor{forestgreen}{rgb}{0.13, 0.55, 0.13}

\newcommand{\cblue}{\color{blue}}

\newcommand{\betacf}{\beta_R}
\newcommand{\betane}{\widetilde{\beta_R}}

\newcommand{\pd}{\partial}
\newcommand{\pdt}{\pd_t}
\newcommand{\Hc}{{\mathcal{P}}}
\newcommand{\penl}{\epsilon}
\newcommand{\viso}{\alpha}
\newcommand{\faces}{\mathcal{E}}
\newcommand{\facesi}{\faces _i}
\newcommand{\facesK}{\faces(K)}

\newcommand{\facesC}{\faces_C}

\newcommand{\bS}{\mathbb S}

\allowdisplaybreaks

\usepackage{CJKutf8} \newcommand{\chinese}[1]{\begin{CJK*}{UTF8}{gbsn}\! #1 \end{CJK*}}%% for chinese

\begin{document}

%%%%%%%%%%%%%%%%%%%%%%%%%%%%%%%%
\title{Convergence and error estimates of a penalization finite volume method for the compressible Navier--Stokes system} 
\author{
M\'aria Luk\'a\v{c}ov\'a-Medvi\softd ov\'a\thanks{
The work of M.L. and Y.Y. was supported by the Deutsche Forschungsgemeinschaft (DFG, German Research Foundation) - project number 233630050 - TRR 146 as well as by TRR 165 Waves to Weather, and by Chinesisch-Deutschen Zentrum f\"ur Wissenschaftsf\"orderung\chinese{(中德科学中心)} - Sino-German project number GZ1465.
M.L. is grateful to the Gutenberg Research College and Mainz Institute of Multiscale Modelling for supporting her research. 
\newline \hspace*{1em} $^{\dag}$B.S.  was supported by National Natural Science Foundation of China under grant No. 12201437. 
%and Czech Sciences Foundation (GA\v CR), Grant Agreement 20--01074S. 
}
\and Bangwei She$^{\dag}$
\and Yuhuan Yuan$^{\clubsuit}$
}

\date{\today}

\maketitle

%\medskip
\vspace{-0.5cm}
\centerline{$^*$Institute of Mathematics, Johannes Gutenberg-University Mainz}
\centerline{Staudingerweg 9, 55 128 Mainz, Germany}
\centerline{lukacova@uni-mainz.de}

\medskip
\centerline{$^\dag$Academy for Multidisciplinary studies, Capital Normal University}
\centerline{ West 3rd Ring North Road 105, 100048 Beijing, P. R. China}
\centerline{bangweishe@cnu.edu.cn}

% \medskip
% \centerline{$^\dag$ Institute of Mathematics of the Academy of Sciences of the Czech Republic}
% \centerline{\v Zitn\' a 25, CZ-115 67 Praha 1, Czech Republic}
% \centerline{she@math.cas.cz}

\medskip
\centerline{$^\clubsuit$School of Mathematics, Nanjing University of Aeronautics and Astronautics}
\centerline{Jiangjun Avenue No. 29, 211106 Nanjing, P. R. China}
\centerline{yuhuanyuan@nuaa.edu.cn}

\begin{abstract}
In numerical simulations a smooth domain occupied by a fluid has to be approximated by a computational domain that typically does not coincide with a physical domain. Consequently, in order to study convergence and error estimates of a numerical method domain-related discretization errors, the so-called variational crimes, need to be taken into account.
In this paper we apply the penalty approach to control domain-related discretization errors. 
We embed the physical domain into a large enough cubed domain and study the convergence of a finite volume method for the corresponding domain-penalized problem.
We show that numerical solutions of the penalized problem converge to a generalized, the so-called dissipative weak, solution of the original problem. If a strong solution exists, the dissipative weak solution emanating from the same initial data coincides with the strong solution. In this case,  we apply a novel tool of the relative energy and derive the error estimates between the numerical solution and the strong solution. Extensive  numerical experiments that confirm theoretical results are presented.
\end{abstract}


%\tableofcontents

\section{Introduction}
We study compressible fluid flow in a smooth domain $\Of, \, \Of  \subset \R^d, \, d=2,3,$  modelled by the Navier--Stokes system
\begin{subequations}\label{pde}
\begin{align}
\partial_t \vr + \Div (\vr \vu) & = 0, \label{pde_d}	\\
\partial_t (\vr \vu) + \Div (\vr \vu \otimes \vu) + \Grad p(\vr) & = \Div \bS( \Grad \vu)  \label{pde_m}
\end{align}
subject to the following Dirichlet boundary condition and initial data, respectively,
\begin{equation} \label{pde_bc}
	\vu|_{\pd \Of} = 0,
\end{equation}
\begin{equation} \label{pde_ic}
	\vr(0, \cdot) = \vr_0, \ \vr \vu (0, \cdot) = \vm_0.
\end{equation}
\end{subequations}
Here, $\vr$ and $\vu$ are the fluid density and velocity, respectively. The viscous stress tensor $\bS$ reads
\begin{equation*}
\bS ( \Grad \vu) = \mu \left( \Grad \vu + \Grad^t \vu - \frac{2}{d} \Div \vu \mathbb{I} \right)   + \lambda \Div \vu \mathbb{I},
\end{equation*}
where $\mu > 0$ and $ \lambda \geq 0$ are the viscosity coefficients.
For the sake of simplicity, we consider the \emph{isentropic}  state equation:
\begin{equation} \label{EOS}
p(\vr) = a \vr^\gamma, \ a > 0,\ \gamma > 1 \ \mbox{ with the associated pressure potential }\
\Hc(\vr) = \frac{a}{\gamma - 1} \vr^\gamma.
\end{equation}
Throughout the paper we always assume that the initial data satisfy
\begin{equation} \label{ic}
 \vr_0 \geq \Un{\vr} > 0,\quad \vr_0 \in L^\infty(\Of), \quad \vm_0 \in L^\infty(\Of;\R^d).
\end{equation}

On the one hand, mathematical analysis of the Navier--Stokes system \eqref{pde} is currently available either for periodic boundary conditions or for no-slip boundary conditions applied in a smooth physical domain occupied by a fluid. On the other hand,  numerical methods, e.g., finite volume or finite element methods, typically require a polygonal computational domain.

Hence, a smooth physical domain has to be approximated by a polygonal computational domain and additional approximation errors arise. Let us note that in the case of complicated geometry the generation of a suitable polygonal approximation may be computationally very costly.

To overcome these difficulties we apply a penalty method originally used in the context of incompressible Navier--Stokes equations by Angot et al.~\cite{ABF_penalty}.
Thus, the physical domain $\Of$ is embedded into a large cube on which the periodic boundary conditions are imposed, see Figure~\ref{figD}. The original boundary conditions are enforced through a penalty term, that is a singular friction term in the momentum equation.
The resulting penalized problem is solved on a flat torus $\tor$ by an upwind-type finite volume (FV) method.
Specifically, the penalized Navier--Stokes system on $\tor$ reads
\begin{subequations}\label{ppde}
\begin{align}
\partial_t \vr + \Div (\vr \vu) &= 0, \label{ppde_d}	\\
\partial_t (\vr \vu) + \Div (\vr \vu \otimes \vu) + \Grad p(\vr) &= \Div \bS( \Grad \vu) -  \frac{ \mathds{1}_{\Os} }{\penl} \vu, \label{ppde_m}
\end{align}
where $\penl > 0$ is the penalty parameter and $ \mathds{1}_{\Os}$ is the characteristic function
\begin{equation*}
 \mathds{1}_{\Os} (\vx)=
\begin{cases}
1 & \mbox{if} \ \vx \in \Os :=\tor \setminus \Of,\\
0 & \mbox{if} \ \vx \in \Of.
\end{cases}
\end{equation*}
The initial data of the penalized problem \eqref{ppde_d}--\eqref{ppde_m} are set to be
\begin{equation} \label{ppde_ic}
\begin{aligned}
&\tvr_{0}  := \tvr(0, \cdot) =
	\begin{cases}
\mbox{ some } \vr_0^s  & \mbox{if} \ \vx \in \Os, \\
\vr_0 &\mbox{if} \  \vx \in \Of,
\end{cases} \quad \quad
 \widetilde{\vm}_{0}  :=\widetilde{\vm}(0, \cdot)=
\begin{cases}
0 & \mbox{if} \ \vx \in \Os,\\
\vr_0 \vu_0 &\mbox{if} \  \vx \in \Of,
\end{cases}
\\
&  \tvr_{0}  > 0 \ \  \  \mbox{satisfying the periodic boundary condition}.
\end{aligned}
\end{equation}
\end{subequations}

\begin{figure}[!h]
\centering
\begin{tikzpicture}[scale=0.7]
\draw[color=blue, very thick] ( -5,-3.2 ) rectangle ( 5, 3.2);
\path node at (-4,2) { \cblue \huge $\tor$};
\draw[very thick]  (0,0) ellipse (4 and 2.5 );
\path node at (0,0) { \huge $\Of$};
\end{tikzpicture}
\caption{A fluid domain $\Of$ embedded into a torus $\tor$.}
\label{figD}
\end{figure}
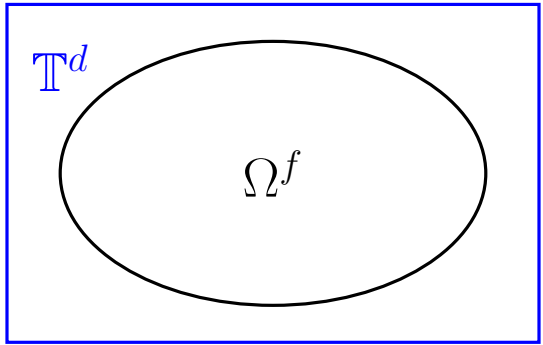

The main aim of the present paper is a rigorous convergence and error analysis of the Dirichlet boundary problem for the compressible Navier--Stokes equations. For a finite volume method using piecewise constant approximations a direct  approach of approximating a smooth boundary by a sequence of polygonal domains leads to problems in the control of consistency errors arising from the convective terms.
Alternatively, if a combined finite volume-finite element method is applied, the convergence and error analysis on a polygonal approximation of the fluid domain $\Of$ has be done, see our recent works \cite{FLM_FoCoM}, \cite{LMS:2023}. The main reason of this discrepancy is a finite element approximation of the velocity that allows more degree of freedom inside a mesh cell.

The idea to penalize a complicated physical domain and solve numerically the corresponding problem on a simple domain is quite often used in the literature. We refer a reader to \cite{Peskin:1972,Peskin:2002} for the immersed boundary method and to \cite{Glowinski-Pan-Periaux:1994,Glowinski-Pan-Periaux:1994a,Hyman:1952} for
the fictitious domain method developed in the context of
 incompressible Navier--Stokes equations. In  \cite{Bruneau:2018} a penalty method has been applied to approximate a moving domain in the
 fluid-structure interaction problem. Further, in  \cite{Hesthaven1, Hesthaven2} penalization of boundary conditions for the compressible Navier--Stokes-Fourier system was applied for the spectral method.
Error estimates between exact and penalized numerical solutions were presented in \cite{ABF_penalty} for the incompressible Navier--Stokes equations and in
 \cite{Maury:2009, Saito-Zhou:2015,Zhang:2006,Zhou-Saito:2014}
for elliptic boundary problems. We mention also  Basari\'c et al.~\cite{Basaric} and Feireisl et al.~\cite{FNS}, where the penalty method was used to prove the existence of weak solutions. In \cite{FNS} the penalization method has been used to show the existence of a weak solution to the compressible Navier--Stokes equations on a moving domain and in \cite{Basaric} the existence of a weak solution to the  Navier--Stokes--Fourier system with the Dirichlet conditions on a rough (Lipschitz) domain was proved.

The present paper is organized in the following way. In Section~\ref{sec-DW} we introduce the concept of generalized, the so-called \emph{dissipative weak} solution for the Dirichlet problem \eqref{pde} and the corresponding penalized problem \eqref{ppde}. A finite volume method for the approximation of the penalized problem \eqref{ppde} is introduced in Section~\ref{sec_Notations}. 
Section~\ref{sec_convergence} and Section~\ref{sec_EE} are devoted to the main results of the paper: convergence analysis of the finite volume method as well as error estimates between the finite volume solutions of the penalized problem and the exact strong solution of the Navier--Stokes system with the Dirichlet boundary conditions. Here we consider the errors with respect to the mesh (discretization) parameter as well as the penalty parameter.
The paper is closed with Section~\ref{sec_num}, where several numerical experiments illustrate our theoretical results.

\section{Dissipative weak solution}\label{sec-DW}
Following \cite{FeLMMiSh} we define  \emph{dissipative weak (DW) solutions} to the Navier--Stokes system. We consider both, the penalized problem on $\tor$ \eqref{ppde} and the original Dirichlet problem on $\Of$ \eqref{pde}. We will show in Section~\ref{sec_convergence} that a DW solution arises as a natural limit of numerical approximations and builds a suitable tool for the convergence analysis.

\

\begin{Definition}[{\bf DW solution of the penalized problem}] \label{PDW}
	We say that $(\vr, \vu)$ is a DW solution of the Navier--Stokes system \eqref{ppde} if the following hold:
	\begin{itemize}
		\item
		{\bf Integrability.}
		\begin{equation}\label{pdw_s}
		\begin{aligned}
			& \vr \geq 0, \ \vr \in L^{\infty}(0,T; L^\gamma(\tor)),  \quad
			\vu \in L^2(0,T; W^{1,2} (\tor; \R^d)), \\
			& \vr \vu  \in L^{\infty}(0,T; L^{\frac{2 \gamma}{\gamma + 1}}(\tor; \R^d)), \quad
		\bS   \in L^2((0,T)\times \tor;  \R^{d\times d}_{sym}).
		\end{aligned}
		\end{equation}

	\item {\bf Energy inequality.}	
	\begin{multline}\label{pdw_E}
		 \left[ \intTdB{ \frac{1}{2} \vr |\vu|^2 + \Hc(\vr)    }\right]  (\tau, \cdot)  +  \int_0^{\tau}  \intTd{ \bS( \Grad \vu): \Grad \vu } \dt  \\
		+ \frac{1}{\penl} \int_0^{\tau} \intOs{ |\vu|^2 } \dt  + \int_{\tor} d \mathfrak{E}(\tau)  + \int_0^{\tau}\int_{\tor} d \mathfrak{D}(\tau) \dt \leq \intTdB{   \frac{|\widetilde{\vm}_{0}|^2}{2\tvr_{0}} + \Hc(\tvr_{0})  }
	\end{multline}
for any $\tau \in [0,T]$, where the energy and dissipation defect measures satisfy
	\begin{equation*}
		\mathfrak{E} \in  L^{\infty}(0,T; \mathcal{M}^+(\tor)), \quad  \mathfrak{D} \in  \mathcal{M}^+([0,T]\times \tor).
	\end{equation*}
			
\item {\bf Equation of continuity.}
		\begin{equation} \label{pdw_d}
		-\intTd{ \tvr_{0}  \phi(0,\cdot) } = \int_0^{T} \intTdB{  \vr \partial_t \phi + \vr \vu \cdot \Grad \phi  } \dt
		\end{equation}
	for any test function $\phi \in  C_c^{1}([0,T) \times \tor) $. 	
		
\item {\bf Momentum equation.}
\begin{multline}\label{pdw_m}
	- \intTd{ \widetilde{\vm}_{0} \cdot  \bfvarphi(0,\cdot) }
	 = \int_0^{T} \intTdB{  \vr \vu \cdot \partial_t  \bfvarphi + \vr \vu \otimes \vu : \Grad  \bfvarphi + p(\vr) \Div  \bfvarphi  } \dt \\
	 - \frac{1}{\penl} \int_0^{T} \intOs{ \vu \cdot  \bfvarphi } \dt
	 - \int_0^{T} \intTd{ \bS( \Grad \vu) : \Grad  \bfvarphi } \dt  + \int_0^{T} \int_{\tor}{\Grad  \bfvarphi   : d\mathfrak{R}(t)   } \dt
\end{multline}	
for any test function $ \bfvarphi \in  C_c^1([0,T) \times \tor; \R^d)$ with the Reynolds defect
	 \begin{equation*}
	 \mathfrak{R} \in  L^{\infty}(0,T; \mathcal{M}^+(\tor; \mathbb{R}^{d\times d}_{sym}))
	 \end{equation*}
satisfying\footnote{ $\mathcal{M}^+(\tor; \mathbb{R}^{d\times d}_{sym})$ denotes the set of symmetric square matrices of order $d$, where each component is a positive Radon measure on $\tor$. %}\footnote{
$\mbox{tr}[\mathfrak{R}]$ denotes the trace of the matrix $\mathfrak{R}$.}
	 \begin{equation}
	 	\underline{d} \mathfrak{E}\leq \mbox{tr}[\mathfrak{R}]  \leq \overline{d} \mathfrak{E} \quad \mbox{for some constants} ~ 0 < \underline{d} \leq \overline{d}.
	 \end{equation}

\end{itemize}
	
\end{Definition}

%%%%%%%
\begin{Definition}[{\bf DW solution of the Dirichlet problem}] \label{DW}
	We say that $(\vr, \vu)$ is a DW solution of the Navier--Stokes system \eqref{pde} if the following hold:
	\begin{itemize}
		\item
		{\bf Integrability.}
		\begin{align}\label{dw_s}
            & \vr \geq 0, \ \vr \in L^{\infty}(0,T; L^\gamma(\Of)),  \quad
			\vu \in L^2(0,T; W^{1,2} (\Of; \R^d)), \br
			& \vr \vu  \in L^{\infty}(0,T; L^{\frac{2 \gamma}{\gamma + 1}}(\Of; \R^d)), \quad
		\bS   \in L^2((0,T)\times \Of;  \R^{d\times d}_{sym}).
		\end{align}
		
	\item {\bf Energy inequality.}	
	\begin{align}\label{dw_E}
		 \left[ \intOf{\left( \frac{1}{2} \vr |\vu|^2 + \Hc(\vr) \right)   }\right]  (\tau, \cdot) & +  \int_0^{\tau}  \intOf{    \bS( \Grad \vu) : \Grad \vu } \dt  \br
		 + \int_{\Ov \Of} d \mathfrak{E}(\tau)  +& \int_0^{\tau} \int_{\Ov \Of} d \mathfrak{D}(\tau)  \dt \leq \intOfB{   \frac{1}{2} \frac{|\vm_0|^2}{\vr_0} + \Hc(\vr_0)  }
	\end{align}
for any $\tau \in [0,T]$, where the energy and dissipation defect measures satisfy
	\begin{equation*}
		\mathfrak{E} \in  L^{\infty}(0,T; \mathcal{M}^+(\Ov \Of)),
		\quad  \mathfrak{D} \in  \mathcal{M}^+([0,T]\times \Ov \Of).
	\end{equation*}
			
		\item {\bf Equation of continuity.}
		\begin{equation} \label{dw_d}
		-\intOf{ \vr_{0} \cdot \phi(0,\cdot) } = \int_0^{T} \intOfB{  \vr \partial_t \phi + \vr \vu \cdot \Grad \phi  } \dt
		\end{equation}
	for any test function $\phi \in C_c^{1}([0,T) \times \Ov{\Of})$.

		\item {\bf Momentum equation.}
\begin{align}\label{dw_m}
	- \intOf{ \vm_0 \cdot  \bfvarphi(0,\cdot) }
	& = \int_0^{T} \intOfB{  \vr \vu \cdot \partial_t  \bfvarphi + \vr \vu \otimes \vu : \Grad  \bfvarphi + p(\vr) \Div  \bfvarphi  } \dt \br
	&- \int_0^{T} \intOf{  \bS( \Grad \vu) : \Grad  \bfvarphi } \dt  + \int_0^{T} \int_{\Ov\Of}{\Grad  \bfvarphi   : d\mathfrak{R}(t)  }\dt
\end{align}	
for any test function $\bfvarphi \in  C_c^1([0,T) \times \Of; \R^d) %{\cgrey C_c^1([0,T) \times \Ov{\Of}; \R^d)}
$ with the Reynolds defect
	 \begin{equation*}
	 \mathfrak{R} \in  L^{\infty}(0,T; \mathcal{M}^+(\Ov\Of; \mathbb{R}^{d\times d}_{sym}))
	 \end{equation*}
	 satisfies
	 \begin{equation}
	 	\underline{d} \mathfrak{E}\leq \mbox{tr}[\mathfrak{R}] \leq \overline{d} \mathfrak{E} \quad \mbox{for some constants} ~ 0 < \underline{d} \leq \overline{d}.
	 \end{equation}

\end{itemize}
	
\end{Definition}

\begin{Remark}
The existences of dissipative solutions to the penalized problem \eqref{ppde} as well as the original Dirichlet problem \eqref{pde} will be  proved in Theorems \ref{THM1} and  \ref{THM2}, respectively.
\end{Remark}

\section{Numerical scheme}\label{sec_Notations}
In this section we generalize the finite volume method proposed in \cite{FLMS_FVNS} to approximate the penalized problem \eqref{ppde}.

\subsection{Space discretization}

\vspace{0.1cm}

\paragraph{\bf Mesh.} Let $\grid$ be a uniform structured (square for $d=2$ or cuboid for $d=3$) mesh of $\tor$ with $h$ being the mesh parameter.
We denote by $\faces$ the set of all faces of $\grid$ and by $\facesi$, $i=1,\dots, d$, the set of all faces that are orthogonal to $\ve_i$ -- the basis vector of the canonical system.
Moreover, we denote by $\facesK$ the set of all faces of a generic element $K\in\grid$. Then, we write $\sigma= K|L$ if $\sigma \in \faces$ is the common face of neighbouring elements $K$ and $L$. Further, we denote by $\vx_K$ and $|K|=h^d$ (resp. $\vx_\sigma$ and $|\sigma|=h^{d-1}$) the center and the Lebesgue measure of an element $K\in\grid$ (resp. a face $\sigma \in \faces$), respectively.

With the above notations, we further define $\Ofh$ as the set of all elements inside the physical domain $\Of$, i.e.
\begin{align*}
\Ofh = \left\{ K \in \grid~ \big| ~ K\subset \Of \right\} \quad \mbox{and}\quad \Osh = \grid \setminus \Ofh,
\end{align*}
cf. Figure \ref{fig:domain}.
In what follows we write that $a \aleq b$ for $a, b \in \R$, if there exists a generic constant $C > 0$ independent on the mesh and penalty parameters $h, \penl$, such that $a\leq Cb$. Further, we write
$a \approx b$ if $a \aleq b$, $b \aleq a$.
It is easy to verify that
\begin{equation}
\Ofh \subset \Of, \quad \Os \subset \Osh \quad \mbox{and} \quad \mbox{dist}(\partial \Of, \Ofh) \approx h,\quad |\Of  \setminus \Ofh| \approx h, \quad |\Osh \setminus \Os| \approx h,
\end{equation}
where generic constants depend on the geometry of the fluid domain $\Of$.

\paragraph{\bf Dual mesh.}  For any $\sigma=K|L\in \facesi$, we define a dual  cell  $D_\sigma := D_{\sigma,K} \cup D_{\sigma,L}$, where $D_{\sigma,K}$ (resp.~$D_{\sigma,L}$) is defined as
\[
D_{\sigma,K} = \left\{ x \in K \mid  x_i\in\co{(x_K)_i}{(x_\sigma)_i}\right\} \ \mbox{ for any } \ \sigma \in \facesi, \; i=1,\dots,d
\]
 with $\co{A}{B} \equiv [ \min\{A,B\} , \max\{A,B\}]$.
We refer to Figure~\ref{fig:mesh} for a two-dimensional illustration of a dual cell.

\begin{figure}[!h]
\centering
\begin{tikzpicture}[scale=1.0]
\draw[-,very thick](0,-2)--(5,-2)--(5,2)--(0,2)--(0,-2)--(-5,-2)--(-5,2)--(0,2);

\draw[-,very thick, green=90!, pattern=north east lines, pattern color=green!30] (0,-2)--(2.5,-2)--(2.5,2)--(0,2)--(0,-2);
\draw[-,very thick, blue=90!, pattern= north west lines, pattern color=blue!30] (0,-2)--(0,2)--(-2.5,2)--(-2.5,-2)--(0,-2);

\path node at (-3.5,0) { $K$};
\path node at (3.5,0)  { $L$};
\path node at (-2.5,0) {$ \bullet$};
\path node at (-2.8,-0.3) {$ \vx_K$};
\path node at (2.5,0) {$\bullet$};
\path node at (2.7,-0.3) {$ \vx_L$};
\path node at (0,0) {$\bullet$};
\path node at (0.3,-0.3) {$ \vx_\sigma$};

\path (-0.4,0.8) node[rotate=90]{ $\sigma=K|L$}; %{ $\sigma=\overrightarrow{K|L}$};
 \path (-1.5,1.4) node[] { $D_{\sigma,K}$};
 \path (1.5,1.4) node[] { $D_{\sigma,L}$};
 \end{tikzpicture}
\caption{Dual mesh $D_\sigma = D_{\sigma,K} \cup D_{\sigma,L}$.}
 \label{fig:mesh}
\end{figure}
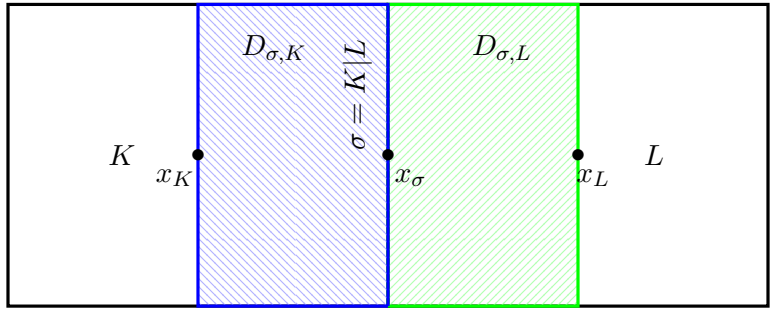

\paragraph{\bf Function space.}
We introduce a space $Q_h$ consisting of piecewise constant functions on $\grid$.
The notation $Q_h^n$ stands for the corresponding $n$-dimensional function space.
Further, we define the standard projection operator $\Pim$ associated with $Q_h$
\begin{equation*}
\Pim: L^1(\tor) \to Q_h, \quad  \Pim  \phi (x) = \sum_{K \in \grid}  \frac{ \mathds{1}_{K}(x) }{|K|} \int_K \phi \dx .
\end{equation*}

\paragraph{Discrete differential operators.}
To begin, we introduce the average and jump operators for $v \in Q_h$
\[
\avs{v}(x) = \frac{v^{\rm in}(x) + v^{\rm out}(x) }{2},\ \ \ \
\jump{ v }(x)  = v^{\rm out}(x) - v^{\rm in}(x) \]
with
\[
v^{\rm out}(x) = \lim_{\delta \to 0+} v(x + \delta \vc{n}),\ \ \ \
v^{\rm in}(x) = \lim_{\delta \to 0+} v(x - \delta \vc{n}),\
\]
whenever $x \in \sigma \in \faces$  and $\vc{n}$ is the outer normal vector to $\sigma$. In addition, if $\sigma \in \facesi$, we also write $\avs{v}$ and $\jump{v}$ as $\avsi{v}$ and $\jumpi{v}$, respectively.

Next, we define the following discrete gradient, divergence and Laplace operators for piecewise constant functions $r_h \in Q_h, \; \bfv_h = (v_{1,h},\cdots, v_{d,h})\in Q_h^d$
\begin{equation*}
\begin{aligned}
\Gradd r_h(\vx) & =  \sum_{\sigma \in \faces} \left( \Gradd r_h \right)_{D_\sigma}  \mathds{1}_{D_\sigma}{(\vx)}, \quad&
 \left( \Gradd r_h \right)_{D_\sigma}  &=  \frac{1}{h}  \jump{r_h} \vc{n} , \quad  \Gradd \bfv_h = \left( \Gradd v_{1,h}, \dots,   \Gradd v_{d,h}\right)^T ,
 \\
\Divh \bfv_h(\vx) &= \sum_{K\in\grid} \left( \Divh \bfv_h\right)_K \mathds{1}_K{(\vx)}, \quad&
\left( \Divh \bfv_h\right)_K & =    \frac{1}{h} \sum_{\sigma \in \facesK} \avs{\bfv_h} \cdot \vc{n} ,
\\
\Delta_h r_h(\vx) &=  \sum_{K \in \grid} \left(\Delta_h r_h\right)_K  \mathds{1}_K{(\vx)}, \quad&
\left(\Delta_h r_h\right)_K &=   \frac{1}{h^2}  \sum_{\sigma \in \facesK} \jump{r_h} .
\end{aligned}
\end{equation*}
It is easy to verify the interpolation errors  estimates
\begin{equation}\label{proj}
\abs{ \jump{\Pim \phi  } }\aleq h \norm{\Grad \phi}_{L^\infty(\Omega)}, \;
\norm{ \Pim \phi  - \phi }_{L^\infty(\Omega)}  \aleq  h \norm{\Grad \phi}_{L^\infty(\Omega)}, \;
\norm{\Gradd \Pim \phi  }_{L^\infty(\Omega)}  \aleq  \norm{ \Grad \phi   }_{L^\infty(\Omega)}
\end{equation}
for any $\phi \in W^{1,\infty}(\Omega)$.

\paragraph{Time discretization.}
Given a time step $\TS >0$ we divide the time interval $[0,T]$ into $N_T=T/\TS$ uniform parts,  and denote $t_k= k\TS.$
Consider a pointwise function $v$, which is only known at time level $t_k,\, k=0,\dots,N_T$.
Let $L_{\TS}(0,T)$ denote a space of all piecewise constant functions $v$, such that
\begin{align*}
&  v(t) =v(t_0)  \ \mbox{ for } \  t < \Delta t, \quad  v(t)=v(t_k) \ \mbox{ for } \ t\in [k\TS,(k+1)\TS), \ \   k=1,\cdots, N_T.
\end{align*}
Further, for $v \in L_{\TS}(0,T)$ we define the backward Euler time discretization operator $D_t$  as follows
\[
 D_t v(t) = \frac{v (t) - v^\triangleleft}{\TS} \quad \mbox{with} \quad v^\triangleleft  := v (t - \TS).
\]
If $\vv$ is a vector function then $D_t \vv$ acts componentwisely.

\subsection{Finite volume method for the  penalized problem}
We  are now ready to propose a finite volume method for the penalized problem \eqref{ppde} that will be presented in the weak form.

\begin{Definition}[{\bf Finite volume method}]
  We say that
$(\vrh^{\penl},\vuh^{\penl}) \in L_{\TS}(0,T; Q_h^{d+1} )$ is a finite volume approximation of the penalized problem \eqref{ppde} if the following algebraic equations hold
\begin{subequations}\label{VFV}
\begin{align}\label{VFV_D}
&\intTd{ D_t \vrh^{\penl} \, \phi_h} - \intfaces{  F_h^{\viso} (\vrh^{\penl} ,\vuh^{\penl} )
\jump{\phi_h}   } = 0, \hspace{4.5cm}  \mbox{for all}\ \phi_h \in Q_h, \\ \label{VFV_M}
&\intTd{ D_t  (\vrh^{\penl}  \vuh^{\penl} ) \cdot  \bfvarphi_h }  - \intfaces{ {\bf F}_h^{\viso}  (\vrh^{\penl}  \vuh^{\penl} ,\vuh^{\penl} ) \cdot \jump{ \bfvarphi_h }   } - \intTd{ p_h^{\penl}  \Divh  \bfvarphi_h }
+ \frac{1}{\penl}\intOsh{\vuh^{\penl} \cdot \bfvarphi_h}
 \br
 &\hspace{1cm}+  \mu  \intTd{ \Gradd \vuh^{\penl}  : \Gradd \bfvarphi_h }
+ \nu  \intTd{\Divh   \vuh^{\penl}   \; \Divh \bfvarphi_h }
=  0,
\hspace{1.5cm} \mbox{for all }  \bfvarphi_h \in Q_h^d,
 \end{align}
\end{subequations}
where
\begin{align*}
 \vrh^{\penl}(0, \cdot) = \Pim\tvr_{0}, \quad (\vrh^{\penl}\vuh^{\penl})(0, \cdot) = \Pim \widetilde{\vm}_{0}, \quad \nu = \lambda+\frac{d-2}{d}\mu.
\end{align*}
The numerical flux function $F_h^{\viso} (r_h,\vuh)$ reads
%\begin{equation}\label{num_flux}
%F_h^{\viso} (r_h,\vuh)
%={Up}[r_h, \vuh] - h^{\viso} \jump{ r_h },
%\quad \mbox{ with}\quad Up [r_h, \bm{u}_h]   = r_h^{\rm up} \avs{\vu_h} \cdot \vn
%\mbox{ and }   \viso >-1.
%\end{equation}
\begin{align}\label{num_flux}
& F_h^{\viso} (r_h,\vuh)
={Up}[r_h, \vuh] - h^{\viso} \jump{ r_h }, \\
&Up [r_h, \bm{u}_h]   = r_h^{\rm up} \avs{\vu_h} \cdot \vn, \quad
r_h^{\rm up} =
\begin{cases}
(r_h)^{\rm in} & \mbox{if} \ \avs{\bm{u}} \cdot \vc{n} \geq 0, \\
(r_h)^{\rm out} & \mbox{if} \ \avs{\bm{u}} \cdot \vc{n} < 0.
\end{cases} \nonumber
\end{align}
Here $ \viso >-1$ is the artificial viscosity parameter.
\end{Definition}

\begin{Remark}
In what follows we shall write $(\vrh,\vuh)$ and $(\vrh^0,\vuh^0)$ instead of more precise notation $(\vrh^{\penl},\vuh^{\penl})$ and $(\vrh^{\penl}(0,\cdot), \vuh^{\penl}(0,\cdot))$ for simplicity, if there is no confusion. 
Consequently, we shall work with the couple $(\vrh,\vuh)$ that represents the (piecewise constant in space and time)  discrete density and velocity, respectively. Moreover, we set $\mh = \vrh \vuh$ and $p_h=p(\vrh)$.
\end{Remark}

\section{Convergence}\label{sec_convergence}
In this section we study the convergence of the finite volume method \eqref{VFV}. To this goal we first discuss its stability and consistency.
\subsection{Stability}\label{sec:stability_ns}
We begin with the following lemma reported by Feireisl et al. \cite[Lemmas 11.2 and 11.3]{FeLMMiSh}.
\begin{Lemma}[Properties {\cite[Lemmas 11.2 and 11.3]{FeLMMiSh}}]
\label{lem_p1}
Let $\tvr_{0} > 0$. Then there exists at least one solution to the FV method \eqref{VFV}. Moreover, any solution $(\vrh ,\vuh )$ to \eqref{VFV_D} satisfies for all $t \in(0,T)$ that
\begin{itemize}
\item Positivity of the density.
\begin{equation*}
\vrh(t)>0.
\end{equation*}
\item Mass conservation.
\begin{equation*}
 \intTd{\vrh(t)} = \intTd{\tvr_{0}}.
\end{equation*}
\item Internal energy balance.
\begin{multline}\label{IEB}
\intTd{ D_t \Hc(\vrh )    }
 + \intTd{  p(\vrh)  \Divh \vuh     }
\\
 = - \intTd{ \frac{ \TS}{2} \Hc''({\vrh^{\star}}) | D_t \vrh  |^2    }
  -  \intfaces{ \left( h^\viso + \frac12 |\avs{\vuh } \cdot \vn | \right) \Hc''( \vr_{h,\dagger} ) \jump{ \vrh  }^2  },
\end{multline}
where $\vrh^{\star} \in \co{\vrh^\triangleleft }{\vrh }$ and $\vr_{h,\dagger} \in \co{\vrh^{\rm in}}{\vrh^{\rm out}}$ for any $ \sigma  \in \faces$.
\end{itemize}
\end{Lemma}

\begin{Lemma}[Energy stability]\label{thm_ST}
Let $(\vrh ,\vuh )$ be a  numerical solution of the FV method \eqref{VFV}. Then  it holds
\begin{align}\label{ST}
& D_t \intTd{ \left(\frac{1}{2}  \vrh  |\vuh |^2 +\Hc(\vrh) \right) }
  +  \intTdB{ \mu |\Gradd \vuh|^2 + \nu |\Divh \vuh|^2}
\br
& \quad \quad = -\frac{1}{\penl} \intOsh{  |\vuh|^2   } - D_{num}
= -\frac{1}{\penl} \intOs{  |\vuh|^2   } - D_{num}^{new},
\end{align}
where $D_{num}^{new}\geq D_{num} \geq 0$ represent the numerical dissipations, which read
\[
\begin{split}
D_{num}^{new} & = D_{num} + \frac{1}{\penl} \int_{\Osh\setminus \Os}  |\vuh|^2  \dx, \\
 D_{num} &=
  h^\viso \intfaces{  \avs{ \vrh  }  \abs{\jump{\vuh }}^2 }
+ \frac{\TS}{2} \intTd{ \vrh^\triangleleft|D_t \vuh |^2  }
+ \frac12 \intfaces{ \vrh^{\rm up} |\avs{\vuh } \cdot \vc{n} |  \abs{\jump{ \vuh }}^2   }
\\&+  \intTd{ \frac{ \TS}{2} \Hc''({\vrh^{\star}}) | D_t \vrh  |^2    }
 +  \intfaces{ \left(h^{\viso} +\frac12 |\avs{\vuh } \cdot \vn | \right)  \Hc''( \vr_{h,\dagger} ) \jump{ \vrh  }^2  }.
\end{split}
\]
\end{Lemma}
\begin{proof}
Following the calculation in \cite[equation (3.4)]{FLMS_FVNS} we obtain the kinetic energy balance
\begin{align*}
 D_t \intTd{ \frac{1}{2} \vrh |\vuh |^2  }
+   \mu   \intTd{  | \Gradd \vuh |^2}   + \nu \intTd{  |\Divh \vuh |^2}
+\frac{1}{\penl} \intOsh{|\vuh|^2}
- \intTd{p_h  \Divh \vuh  }
\\
+  h^\viso \intfaces{  \avs{ \vrh  }  \abs{ \jump{ \vuh  }}^2   }
+ \frac{\TS}{2} \intTd{ \vrh^\triangleleft|D_t \vuh |^2  }
+ \frac12 \intfaces{ \vrh ^{\rm up} |\avs{\vuh } \cdot \vc{n}|  \abs{ \jump{ \vuh  }}^2   } =0.
\end{align*}
Combining this with the internal energy balance \eqref{IEB} finishes the proof.
\end{proof}

Next, thanks to the energy balance \eqref{ST} and the Sobolev-Poincar\'e inequality, see Lemma \ref{lmSP}, we obtain  the following a priori  bounds for the finite volume solutions $\{\vrh, \vuh \}_{h \searrow 0}$. 
\begin{Lemma}[Uniform bounds]\label{lm_ub}
Let $(\vrh ,\vuh )$ be a  numerical solution of the FV method \eqref{VFV}. Then  the following hold
\begin{subequations}\label{ap}
\begin{align}\label{ap1}
& \norm{\vrh}_{L^\infty(0,T; L^\gamma(\tor))}  + \norm{\vrh \vuh }_{L^\infty (0,T; L^{\frac{2\gamma}{\gamma+1}}(\tor;\R^d)) }  + \norm{p_h}_{L^\infty (0,T; L^1(\tor))} + \br
& \quad + \norm{\vuh}_{L^2(0,T; L^p(\tor;\R^d))} + \norm{\Gradd \vuh}_{L^2((0,T)\times\tor;\R^{d\times d})} + \norm{\Divh \vuh}_{L^2((0,T)\times\tor)}  \leq C ,
\end{align}
\begin{align} \label{ap3}
&  h^\viso \int_0^T \intfaces{  \avs{ \vrh  }  \abs{\jump{\vuh }}^2 } \dt +  \int_0^T  \intfaces{ \left(h^{\viso} +\frac12 |\avs{\vuh } \cdot \vn | \right) \Hc''( \vr_{h,\dagger} ) \jump{ \vrh  }^2 } \dt \leq C,
\end{align}
\begin{equation}\label{ap2}
\frac{1}{\penl} \norm{\vuh}^2_{L^2((0,T)\times\Os;\R^d)}  \leq \frac{1}{\penl} \norm{\vuh}^2_{L^2((0,T)\times\Osh;\R^d)} \leq C,
\end{equation}
\end{subequations}
where $\vr_{h,\dagger} \in \co{\vrh^{\rm in}}{\vrh^{\rm out}}$. The parameter $p\in[1,\infty)$ for $d=2$ and $p=6$ for $d=3$.
The generic constant $C$ depends on the initial mass $M_0 := \intTd{\tvr_0} > 0$ and the initial energy $E_0 := \intTd{\left( \frac12 \tvr_0 \abs{\tvu_0}^2 + \Hc(\tvr_0) \right)} > 0$, but it is independent of  the computational parameters $(h, \TS)$ as well as the penalty  parameter $\penl$.
\end{Lemma}

\subsection{Consistency formulation}
We proceed with the consistency analysis of the finite volume method \eqref{VFV}.
To begin, let us define two consistency errors: 
\begin{itemize}
    \item the consistency error $e_\vr$ 
\begin{align} \label{CS1-new}
e_\vr(\tau, \TS, h, \phi) := \left[ \intTd{ \vrh \phi  } \right]_{t=0}^{t=\tau} -
\int_0^\tau \intTdB{   \vrh \partial_t \phi + \vrh   \vuh \cdot \Grad \phi } \dt
\end{align}
 for all $\phi\in W^{1,\infty}((0,T)\times \tor)$;
\item 
the consistency error $e_{\vm}$ 
\begin{align}\label{CS2-new}
e_{\vm}(\tau, \TS, h, \bfvarphi) := &\left[  \intTd{ \vrh   \vuh \cdot \bfvarphi } \right]_{t=0}^{t=\tau} -
\int_0^\tau \intTdB{  \vrh \vuh \cdot \partial_t  \bfvarphi + \vrh \vuh \otimes \vuh  : \Grad \bfvarphi  + p_h \Div \bfvarphi }\dt
\br
& +    \int_0^\tau \intTdB{ \mu \Gradd \vuh : \Grad \bfvarphi  +  \nu \Divh \vuh \; \Div \bfvarphi} \dt  
+ \frac{1}{\penl} \int_0^\tau \intOs{ \vuh \cdot \bfvarphi } \dt
\end{align}
for all $\bfvarphi\in W^{1,\infty}((0,T)\times \tor;\R^d)$.
\end{itemize}
In order to have the consistency formulation, we estimate the consistency errors $e_{\vr}, e_{\vm}$ by using more regular test functions.
Since the consistency proof is quite technical and the ideas are analogous to \cite[Section 2.7]{FLS_IEE} and \cite[Section 11.3]{FeLMMiSh}, we postpone it to Appendix \ref{sec_cs}.
We point out that the result presented in Lemma~\ref{thm_CS} is new. Indeed, the consistency errors are improved since they are more precise and the test function in the continuity equation \eqref{CS1-new} 
is now more general.

\begin{Lemma}[Consistency formulation]\label{thm_CS}
Let $(\vrh, \vuh)$ be a solution of the FV scheme \eqref{VFV}  with $(\TS,h,\penl) \in (0,1)^3$ and $\viso >-1$.
Then for any $\tau \in [0, T]$ it holds 
\begin{equation}\label{CS1-1-new}
\abs{e_\vr(\tau, \TS, h, \phi) } \leq   C_\vr \big(\TS  + h^{(1+\viso)/2} + h^{(1+\betacf )/2} +  h^{1+\beta_D} \big);
\end{equation}
for all $\phi \in W^{1, \infty}((0,T) \times \tor),\ \partial_t^2 \phi \in L^{\infty}((0,T)\times \tor) $;
\begin{align}\label{CS2-1-new}
& \abs{e_{\vm}(\tau, \TS, h, \bfvarphi) } \leq C_{\bm{m}}  \big( \sqrt{\TS} + h + h^{1+\viso}  + h^{1+\beta_M}   + (h/\penl)^{1/2} +  (\TS/\penl)^{1/2} \big)
\end{align}
for all $\bfvarphi \in W^{1,\infty}((0,T) \times \tor; \R^d) \cap L^\infty(0,T; W^{2,\infty}(\tor;  \R^d)), \partial_t^2 \bfvarphi \in L^{\infty}((0,T)\times \tor; \R^d) $.

\vspace{0.15cm}

Here the constant $C_\vr$ depends on $E_0,  T,   \norm{\phi}_{ W^{1,\infty}((0,T) \times \tor) },    \norm{\partial_t^2\phi}_{ L^{\infty}([0,T] \times \tor)}$ 
and $C_{\vm}$ depends on $E_0, T$, $ \norm{\bfvarphi}_{ W^{1,\infty}((0,T) \times \tor;\R^d) }$, $\norm{ \bfvarphi}_{ L^\infty(0,T; W^{2,\infty}(\tor;\R^d)) }$,  $\norm{\partial_t^2 \bfvarphi}_{ L^{\infty}([0,T] \times \tor;\R^d) }$.
Further, $\beta_D, \betacf, \beta_M$ are defined as
\begin{align*}
& \beta_D =
\begin{cases}
 \min\limits_{_{p \in \left[ 1, \infty \right)}}\left\{ \frac{p(\viso+1)+4}{2p}, 1 \right\} \cdot \frac{\gamma-2}{\gamma} \ & \mbox{if } d= 2, \gamma \in(1,2), \\
\min\left\{ \frac{\viso+2}{3} , 1 \right\} \cdot \frac{3(\gamma-2)}{2\gamma} & \mbox{if } d= 3, \gamma \in(1,2), \\
0 &\mbox{if } \gamma \geq 2,
\end{cases}
\\
& \betacf =
\begin{cases}
0 & \mbox{if} \ d = 2,\\
 \min\left\{   \frac{1+\viso}{2} , 1\right\} \cdot \frac{5\gamma-6}{2\gamma} \ & \mbox{if} \ d = 3, \gamma \in \left(1, \frac65\right),\\
 0 &\mbox{if } d = 3, \gamma \geq \frac65,
\end{cases} \quad
\\
& \beta_M =
\begin{cases}
\max\limits_{p \in \left[ \frac{2\gamma}{\gamma-1}, \infty \right)}\left\{ - \frac{p(\viso+1)+4}{2p\gamma}, \frac{p(\gamma-2)-2\gamma}{p\gamma}\right\} \
& \mbox{ if } d = 2, \gamma \leq 2, \\
0
& \mbox{ if } d = 2, \gamma > 2, \\
\max \left\{ - \frac{\viso+2}{2\gamma}, \frac{\gamma-3}{\gamma},-\frac{3}{2\gamma} \right\}
& \mbox{ if } d = 3, \gamma \leq 2, \\
 \frac{\gamma-3}{\gamma}
&  \mbox{ if } d = 3, \gamma \in (2,3),\\
0
& \mbox{ if } d = 3, \gamma \geq 3.
\end{cases}
\end{align*}
\end{Lemma}

\begin{Remark}[Observations on the parameters $\betacf,\beta_D,\beta_M$ and $\viso$] \label{rmk-1}

Our consistency errors (involving the terms  $\beta_D, \betacf, \beta_M$) are better than the results obtained in \cite{FeLMMiSh}. 
Further, 
\begin{itemize}
%\hfill
\item
It is easy to verify that
\begin{equation*}
 0 \geq \betacf \geq \beta_D \geq \beta_M \text{ and } \beta_D> -1.
\end{equation*}
Moreover,  $\beta_{M} > -1$ if one of the following conditions holds
\begin{itemize}
\item $d=2$;
\item $d=3$ and $\gamma>\frac32$;
\item $d=3$ and $\gamma \leq \frac32$ with $\viso < 2(\gamma -1)$.
\end{itemize}
Consequently, we obtain a weaker constrain on $\viso$ than the one obtained in \cite{FeLMMiSh}, where $\viso < 2 \gamma - 1 -d/3$ was needed for all $\gamma \in (1,2).$

\item
If $\viso \geq 1,$ the parameters $\beta_D, \betacf, \beta_M$ are independent of $\viso$. Indeed,  for $\viso \geq 1$ we have simpler forms of $\beta_D, \betacf, \beta_M$, i.e.
\begin{align*}
&\beta_D =
\begin{cases}
 \frac{d(\gamma-2)}{2\gamma} \ & \mbox{if } \gamma < 2, \\
0 &\mbox{otherwise} ,
\end{cases} \quad \quad
\betacf =
\begin{cases}
\frac{5\gamma-6}{2\gamma}  \ & \mbox{if } d = 3, \gamma < \frac65, \\
0 &\mbox{otherwise} ,
\end{cases}
\\
&\beta_M =
\begin{cases}
\max\limits_{p \in \left[ \frac{2\gamma}{\gamma-1}, \infty \right)}\frac{p(\gamma-2)-2\gamma}{p\gamma} \
& \mbox{ if } d = 2, \gamma \leq 2, \\
0
& \mbox{ if } d = 2, \gamma > 2, \\
\max \left\{ \frac{\gamma-3}{\gamma},-\frac{3}{2\gamma} \right\}
& \mbox{ if } d = 3, \gamma <3, \\
0
& \mbox{ if } d = 3, \gamma \geq 3.
\end{cases}
\end{align*}
\end{itemize}
\end{Remark}

\subsection{Weak convergence for the penalized problem}\label{sec-converge-tor}
In this section we consider $\penl$ fixed and pass to the limit with $\TS \approx h \to 0$. The corresponding (weak) limit of $(\vrh,\vuh)$ will be denoted by $(\vr_{\penl},\vu_{\penl})$.
First, we deduce from a priori estimates \eqref{ap1} that up to a subsequence
\begin{align*}
\vrh &\to \vr_{\penl} \ \mbox{weakly-(*) in}\ L^\infty(0,T; L^\gamma(\tor)),\ \vr \geq 0 ,
\\
\vuh &\to \vu_{\penl} \ \mbox{weakly in}\ L^2(0,T; L^6(\tor; \R^d)),
\\
 \Gradd \vuh &\to \Grad  \vu_{\penl} \ \mbox{weakly in} \ L^2((0,T) \times \tor;\R^{d \times d}), \quad  \mbox{where} \ \vu_{\penl}   \in L^2(0,T; W^{1,2}(\tor; \R^d))
 \end{align*}
and
 \begin{align*}
 \vrh \vuh  &\to \vm_{\penl} \ \  \mbox{weakly-(*) in}\ L^\infty(0,T; L^{\frac{2\gamma}{\gamma + 1}}(\tor; \R^d)) \quad 
 \mbox{ for } h \to 0.
 \end{align*}
Realizing that $(\vrh, \vuh)$ satisfies the consistency formulation \eqref{CS1-new} for the mass conservation equation, applying \cite[Lemma 3.7]{AbbFeiNov} (see similar result in \cite[Lemma 7.1]{Karper})  we obtain
\begin{align*}
& \vm_{\penl} = \vr_{\penl} \vu_{\penl}.
\end{align*}

Further, due to the fact that the total energy $E = \frac12 \vr |\vu|^2 + \Hc(\vr)$
is a convex function of $(\vr, \vm)$ and $\abs{\Grad \vu}^2 + \abs{\Div \vu}^2 = \abs{\Grad \vu}^2 + \abs{\mbox{tr}(\Grad \vu)}^2$ is a convex function of $\Grad \vu$, we deduce that, cf.~\cite[Lemma 2.7]{FGSGW}
  \begin{align*}
 \frac12 \vrh |\vuh|^2 + \Hc(\vrh) &\to \Ov{\frac{|\vm|^2}{2 \vr}  + \Hc(\vr)} \quad \mbox{weakly-(*) in}\ L^\infty(0,T; \mathcal{M}^+(\tor)), \\
\vrh \vuh \otimes \vuh + p(\vrh) \I & \to \Ov{ \frac{\vm \otimes \vm}{\vr}  + p(\vr) \I} \quad \mbox{weakly-(*)  in}\ L^\infty(0,T; \mathcal{M}^+(\tor; \R^{d\times d}_{\rm sym})),
\\
 \mu |\Gradd \vuh|^2 + \nu |\Divh \vuh|^2 & \to \Ov{  \mu |\Grad \vu|^2 + \nu |\Div \vu|^2 } \quad \mbox{weakly-(*)  in}\  \mathcal{M}^+([0,T]\times \tor) \quad 
 \mbox{ for } h \to 0
\end{align*}
with the defects
 \begin{align*}
\mathfrak{E} & =  \Ov{\frac{|\vm|^2}{2 \vr}  + \Hc(\vr)} - \left( \frac12 \vr_{\penl} |\vu_{\penl}|^2 + \Hc(\vr_{\penl}) \right) \geq 0, \\
\mathfrak{R} & =  \Ov{ \frac{\vm \otimes \vm}{\vr}  + p(\vr) \I } - \left( \vr_{\penl} \vu_{\penl} \otimes \vu_{\penl} + p(\vr_{\penl})  \I \right)  \geq 0,\\
\mathfrak{D} & =  \Ov{  \mu |\Grad \vu|^2 + \nu |\Div \vu|^2 } - \left(  \mu |\Grad \vu_{\penl}|^2 + \nu |\Div \vu_{\penl}|^2 \right)  \geq 0
\end{align*}
satisfying
\begin{equation*}
\underline{d} \mathfrak{E}\leq \mbox{tr}[\mathfrak{R}] \leq \overline{d} \mathfrak{E}, \quad  \underline{d} = \min\left(2, {d(\gamma-1)} \right),\  \overline{d} =  \max\left(2, {d(\gamma-1)} \right).
\end{equation*}

Together with
\begin{align*}
\lim_{h\to 0} \intTd{E(\vrh^0,\mh^0)}  = \intTd{E(\tvr_{0} ,\widetilde{\vm}_{0} )},
\end{align*}
the consistency formulations \eqref{CS1-new} and \eqref{CS2-new}, and the energy balance \eqref{ST},
the limit $(\vr_{\penl}, \vu_{\penl})$ is a DW solution of the penalized problem \eqref{ppde} in the sense of Definition \ref{PDW}.
We summarize the obtained result on the weak convergence of FV solutions in the following theorem.

\begin{Theorem}\rm\label{THM1}
(\textbf{Weak convergence for the penalized problem}).
Let $p$ satisfy \eqref{EOS} with $\gamma>1$ and $\penl > 0$ be a fixed penalty parameter.
 Let $\{ \vrh, \vuh \}_{h \searrow 0}$ be a family of numerical solutions obtained by the FV method \eqref{VFV} with $ \TS \approx h \in (0,1), \viso > -1$ and initial data satisfying \eqref{ic}. 
 If $d=3$ and $\gamma \leq \frac32$ we assume  in addition that $\viso < 2(\gamma -1)$.

Then, up to a subsequence, the FV solutions $\{ \vrh, \vuh \}_{h \searrow 0}$ converge in the following sense
\begin{align}
		\vrh &\longrightarrow  \ \vr_{\penl} \ \mbox{weakly-(*) in}\ L^{\infty}(0,T; L^{\gamma}(\tor)), \br
		\vuh &\longrightarrow \ \vu_{\penl} \ \mbox{weakly in}\ L^2((0,T) \times \tor;\R^{d }),\br
		\Gradd \vuh & \longrightarrow \  \Grad  \vu_{\penl} \ \mbox{weakly in} \ L^2((0,T) \times \tor;\R^{d \times d}), \br
		\Divh \vuh & \longrightarrow \ \Div  \vu_{\penl} \ \mbox{weakly in} \ L^2((0,T) \times \tor) \quad
		\mbox{for } \ h \to 0,
\end{align}
where $(\vr_{\penl}, \vu_{\penl})$ is a DW solution of the penalized problem  \eqref{ppde} in the sense of Definition \ref{PDW}.	
\end{Theorem}

\subsection{Weak convergence for the Dirichlet problem}\label{sec-converge-of}
We proceed by considering the limit process for  $\TS \approx h \to 0$ and $\penl \to 0$.  
The corresponding (weak) limit on $\tor$ will be denoted  by $(\tvr,\tvu)$. Moreover, we set $(\vr,\vu) := (\tvr,\tvu)|_{\Of}$.
Analogously to Section~\ref{sec-converge-tor} we have the following convergence results:
\begin{align*}
\vrh &\to \tvr \ \mbox{weakly-(*) in}\ L^\infty(0,T; L^\gamma(\tor)),\ \vr \geq 0 ,
\\
\vuh &\to \tvu \ \mbox{weakly in}\ L^2(0,T; L^6(\tor; \R^d)),
\\
 \Gradd \vuh &\to \Grad  \tvu \ \mbox{weakly in} \ L^2((0,T) \times \tor;\R^{d \times d}), \quad  \mbox{where} \  \tvu   \in L^2(0,T; W^{1,2}(\tor; \R^d)),
\\
 \vrh \vuh  &\to \tvm=\tvr\tvu \ \  \mbox{weakly-(*) in}\ L^\infty(0,T; L^{\frac{2\gamma}{\gamma + 1}}(\tor; \R^d)) \quad
 \mbox{for }\ h \to 0, \ \penl \to 0. 
 \end{align*}

Further, according to a priori bound \eqref{ap2} we obtain that
\begin{align*}
\vuh &\to 0 \ \mbox{ strongly in}\ L^2((0,T)\times \Os; \R^d) \quad
 \mbox{for }\ h \to 0, \ \penl \to 0
 \end{align*}
yielding
\begin{equation*}
\vu \in\ L^2(0,T; W_0^{1,2}(\Of; \R^d)).
\end{equation*}
Together with the fact that  $\tvr$ satisfies the equation of continuity in the weak sense, we have
\[
\partial_t  \tvr = 0 \ \mbox{in}\ \mathcal{D}'((0,T) \times \Os ).
\]
Since $\tvr \in C_{weak}(0,T;L^\gamma(\tor))$, cf. \cite[Remark 1 of Section 3.3]{FN}, taking any test function $\phi \in C(\Ov{\Os})$ directly gives
\[
\intOs{\tvr (t, \cdot) \phi} = \intOs{\tvr_0 \phi}  \quad \mbox{for all } t\in [0,T], 
\]
which means
\begin{equation}\label{RHOS}
\tvr = \tvr_{0} \ \mbox{in}\ \Os \quad \mbox{for all } t\in [0,T].
\end{equation}

Consequently,
we have
\begin{align}
&\lim_{(h, \penl)  \to 0}  \intTd{ \left( \frac{1}{2} \frac{|\vm_{h}^0|^2}{\vr_{h}^0} + \Hc(\vr_{h}^0) \right) }  = \intOf{ \left( \frac{1}{2} \frac{|\vm_{0}|^2}{\vr_{0}} + \Hc(\vr_{0})  \right) } + \intOs{ \Hc(\tvr_{0})  }
\label{C9}
\end{align}
and
\begin{equation}\label{C10}
\begin{aligned}
&\lim_{(h, \penl) \to 0}  \intOf{ \left( \frac{1}{2} \vrh |\vuh|^2 + \Hc(\vrh) \right) }
=  \int_{\Ov \Of} d \mathfrak{E}(\tau) + \intOf{ \left( \frac{1}{2} \vr |\vu|^2 + \Hc(\vr) \right) } ,
\\
& \lim_{(h, \penl) \to 0}  \intOs{  \frac{1}{2} \vrh |\vuh|^2 } \geq  0, \quad  \lim_{(h, \penl) \to 0} \intOs{  \Hc(\vrh) }
 \geq \intOs{  \Hc(\tvr_{0})  },
\end{aligned}
\end{equation}
which yields the energy inequality \eqref{dw_E}.

Together with the consistency formulations \eqref{CS1-new},  \eqref{CS2-new} and the energy balance \eqref{ST},
the limit $(\vr, \vu)$ is a DW solution of the Navier--Stokes system \eqref{pde} with the Dirichlet boundary conditions in the  sense of Definition~\ref{DW}.

\begin{Theorem}\rm\label{THM2}
(\textbf{Weak convergence for the Dirichlet problem}).
In addition to the assumption of Theorem \ref{THM1},  let $h,\, \penl$ satisfy	
\begin{align}\label{penl3}
h^3/\penl \to 0, \quad \quad \mbox{as} \quad h,\penl \to 0.
\end{align}
Then, up to a subsequence, the FV solutions $\{ \vrh, \vuh \}_{h,\penl \searrow 0}$ converge in the following sense
\begin{align}
		\vrh &\longrightarrow  \ \vr \ \mbox{weakly-(*) in}\ L^{\infty}(0,T; L^{\gamma}(\Of)), \br
		\vuh &\longrightarrow  \vu \ \mbox{weakly in}\ L^2(0,T;  L^2(\Of; \R^d) ),\br
		\Gradd \vuh & \longrightarrow \Grad  \vu \ \mbox{weakly in} \ L^2((0,T) \times \Of;\R^{d \times d}), \ \  \mbox{ where } \vu \in L^2(0,T; W^{1,2}_0(\Of; \R^d)), \br
		\Divh \vuh & \longrightarrow \Div  \vu \ \mbox{weakly in} \ L^2((0,T) \times \Of) \quad
 \mbox{for }\ h \to 0, \ \penl \to 0 
\end{align}
where $(\vr, \vu)$ is a DW solution of the Dirichlet problem of Navier--Stokes system \eqref{pde} in the sense of Definition \ref{DW}.	
\end{Theorem}

\begin{proof}
Compared to Theorem~\ref{THM1} we additionally require \eqref{penl3} in order to get a better control of the consistency error term  $E_{\penl}$ stated in \eqref{eeD-1}. 
With the test function $\bfphi \in  C^2_c([0,T) \times \Of)$ %{\cgrey $C^1_c([0,T] \times \Ov{\Of})$}  
the estimate of $E_{\penl}$ can be improved with 
\begin{align*}
\abs{E_{\penl} } =  \abs{\frac{1}{\penl} \intn \int_{\Osh \setminus \Os} \vuh \cdot \bfphi  \dxdt} 
\aleq  h^{1/2} \penl^{-1/2}\norm{\bfphi}_{L^\infty((0,T)\times\Osh \setminus \Os)}  
\aleq  h^{3/2} \penl^{-1/2}
\end{align*}
resulting $e_{\vm} \to 0$. 
This concludes the proof.
\end{proof}

\subsection{Strong convergence for the Dirichlet problem}\label{sec-convergence-strong}
In this section we study the strong convergence of the numerical solutions $\{ \vrh, \vuh \}_{h,\penl \searrow 0}$.
To begin, let us introduce the definition of the strong solution of the Navier--Stokes system \eqref{pde}.
For the local existence, we refer a reader to 
Valli and Zajaczkowski \cite{Valli}, and Kawashima and Shizuta \cite{KawShi}.

\begin{Definition}[{\bf Strong solution}]\label{def_ST}
Let $\Of \subset \R^d, d=2,3,$ be a bounded domain with a smooth boundary $\partial \Of$.
We say that $(\vr, \vu)$ is the strong solution of the Navier--Stokes problem \eqref{pde}  if
\begin{equation}\label{STC}
\begin{aligned}
&\vr \in C^1([0,T]\times \Ov{\Of}) \cap C(0,T; W^{4,2}(\Of)),
\\&
\vu \in C^1([0,T]\times \Ov{\Of}; \R^d)  \cap C(0,T; W^{4,2}(\Of;\R^d))
\end{aligned}
\end{equation}
and equations \eqref{pde} are satisfied pointwise.
\end{Definition}

\begin{Theorem}\rm\label{THM3}
(\textbf{Strong convergence for the Dirichlet problem}). 	
%Let $\gamma > 1$.
Let the initial data $(\vr_0, \vu_0)$ satisfy
\begin{align*}
\vr_0 \in W^{4,2}(\Of),\ \vr_0 > 0, \ \vu_0 \in W^{4,2}(\Of; \R^d)
\end{align*}
and $(\vr,\vu)$ be the corresponding strong solution of the Navier--Stokes system \eqref{pde}  belonging to the class \eqref{STC}.
%with the initial data $(\vr_0, \vu_0)$ satisfying $\vr_0 > 0$.
Let $\{ \vrh, \vuh \}_{h,\penl \searrow 0}$ be a family of numerical solutions obtained by the FV method \eqref{VFV}.  
Moreover, we assume that the parameters $\viso, \TS, h, \penl$ satisfy the same conditions as in Theorem~\ref{THM2}.

Then the FV solutions $\{ \vrh, \vuh \}_{h,\penl \searrow 0}$ converge strongly to the strong solution $(\vr, \vu)$ in the following sense
\begin{align*}
		\vrh &\to \vr \ \mbox{strongly in}\  L^r(0,T; L^\gamma(\Of)),\br 
		\vuh &\to \vu \ \mbox{strongly in}\ L^2((0,T)\times\Of; \R^d) \quad
 \mbox{for }\ h \to 0, \ \penl \to 0 
\end{align*}
for any $1 \leq r < \infty$.
\end{Theorem}

\begin{proof}
By virtue of the weak-strong uniqueness principle established in \cite[Theorem 4.1]{FGSGW}, analogously as in \cite[Theorem 7.12]{FeLMMiSh}, we obtain that the numerical solutions $\{ \vrh, \vuh \}_{h,\penl \searrow 0}$ converge strongly to the strong solution $(\vr,\vu)$ on $\Of$.
\end{proof}

\section{Error estimates}\label{sec_EE}
Having proven the convergence of the FV method \eqref{VFV}, we proceed with the error analysis between the FV approximation of the penalized problem \eqref{ppde} and the strong solution to the Dirichlet problem \eqref{pde}.
 For simplicity of the presentation of main ideas, here and hereafter we consider a semi-discrete version of the FV method. In other words, we only study the error with respect to the spatial discretization.

In order to measure the distance between a FV solution of the penalized problem \eqref{ppde} and the strong solution to the Navier--Stokes problem \eqref{pde} we introduce the relative  energy functional
\begin{align}\label{RE}
\RE(\vrh, \vuh| \tvr,\tvu) = \intTd{\left( \frac12 \vrh \abs{\vuh- \tvu}^2 + \bbE(\vrh|\tvr) \right)}, \;
\bbE(\vrh|\tvr) =\Hc(\vrh) - \Hc'(\tvr) (\vrh -\tvr) -\Hc(\tvr).
\end{align}
 Here, $(\tvr,\tvu)$ is an extension of the strong solution as prescribed below.

\begin{Definition}[{\bf Extension of the strong solution}]
\label{def_ES}
Let $(\vr,\vu)$ be the strong solution in the sense of Definition \ref{def_ST}.
We say that $(\tvr, \tvu)$ is the extension of the strong solution $(\vr, \vu)$ if
\begin{equation}\label{EXT1}
(\tvr, \tvu)(x) =
\begin{cases}
(\vr_{0}^s , \, 0) \ & \mbox{if} \ \vx \in \Os,\\
(\vr, \, \vu) &\mbox{if} \  \vx \in \Of,
\end{cases} \quad  \mbox{ for any } t \in [0,T].
\end{equation}
\end{Definition}

\begin{Theorem}{\rm ({\bf Error estimates})}. \label{THM:ES}
Let $\gamma>1$, $(\vr_0, \vu_0) \in  W^{4,2}(\Of) \times W^{4,2}(\Of; \R^d), \vr_0^s \in C^2(\Os)$ and $\tvr_0 > 0$. 
Let $\{ \vrh, \vuh \}_{h,\penl \searrow 0}$ be a family of numerical solutions obtained by the semi-discrete version of FV method \eqref{VFV} with $(h, \penl) \in (0,1)^2$ and $\viso > 0$. 
Suppose that the numerical density $\vrh$ is uniformly bounded 
\begin{equation} \label{ass}
0<\vrh<  \Ov{\vr} \quad \mbox{uniformly for} \quad h,\penl \to 0
\end{equation}
and the Navier--Stokes system \eqref{pde} admits a global-in-time strong solution $(\vr,\vu)$  belonging to the class \eqref{STC}.

Then the following error estimate holds %for $(h, \penl) \in (0,1]^2$
\begin{align}\label{REI2}
& \RE(\vrh, \vuh| \tvr, \tvu)(\tau)  +  \intTauTdB{\mu \abs{\Gradd \vuh- \Grad \tvu}^2 +  \nu \abs{\Divh \vuh- \Div \tvu}^2 }   + \frac{1}{\penl} \intTauOs{|\vuh|^2} \br
& \aleq    h^{\beta_{RE}} +  \frac{h^{3}}{\penl}  + \frac{\penl}{h},\quad 
\beta_{RE} = \min \{1,(1+\viso)/2,\viso \}.
\end{align}
%where $\beta_{RE}$ is defined in \eqref{re}. 
Further, there hold 
%\begin{align}\label{ee-d}
%&
%\norm{\vrh-\tvr }^2_{L^{\gamma}(\tor)} + \norm{\vrh \vuh - \tvr\tvu}_{L^{\frac{2\gamma}{\gamma+1}}(\tor)}^2 \aleq h^{\beta_{RE}} +  \frac{h^{3}}{\penl}  + \frac{\penl}{h} \hspace{3.5cm} \quad \mbox{ if }  \gamma \leq 2,
%\\ & 
% \norm{\vrh-\tvr }^{\gamma}_{L^{\gamma}(\tor)}   +
% \norm{\vrh-\tvr }^{2}_{L^{2}(\tor)} + \norm{\vrh \vuh - \tvr\tvu}_{L^{\frac{2\gamma}{\gamma+1}}(\tor)}^2
% \aleq h^{\beta_{RE}} +  \frac{h^{3}}{\penl}  + \frac{\penl}{h} \hspace{1cm} \mbox{ if }  \gamma > 2
%\end{align}
\begin{align}\label{ee-u}
  \norm{(\tvr,\tvu,\tvr\tvu) - (\vrh,\vuh,\vrh\vuh)}^2_{L^{2}((0,T)\times\tor;\R^{2d+1})} 
%+ \norm{\tvr\tvu - \vrh\vuh}^2_{L^{2}((0,T)\times\tor)} + \norm{\tvu - \vuh}^2_{L^{2}((0,T)\times\tor)} 
+ \norm{\Gradd \vuh- \Grad \tvu}^2_{L^{2}((0,T)\times\tor;\R^{d\times d})}\aleq h^{\beta_{RE}} +  \frac{h^{3}}{\penl}  + \frac{\penl}{h}
\end{align}
and 
\begin{equation}\label{ee-d}
  \norm{\vrh-\tvr }^{\gamma}_{L^{\gamma}(\tor)} \aleq h^{\beta_{RE}} +  \frac{h^{3}}{\penl}  + \frac{\penl}{h} \quad \mbox{ if }  \gamma > 2. 
\end{equation} 
\end{Theorem}

\begin{Remark}
The error estimates in Theorem~\ref{THM:ES} confirm the strong convergence of the penalty method \eqref{VFV}. By a closer inspection we observe the optimal first order convergence rate for the relative energy, i.e.\ $h^{\beta_{RE}} +  \frac{h^{3}}{\penl}  + \frac{\penl}{h} \approx h $, by choosing $\penl = h^2$ and $\viso=1$. %Consequently, $h^{\beta_{RE}} +  \frac{h^{3}}{\penl}  + \frac{\penl}{h} \approx h $.
\end{Remark}

\begin{Remark}
We point out that the above error estimates can be generalized to the fully discrete method \eqref{VFV} by the same argument as in \cite[Theorem 6.2 or Appendix D]{LSY_penalty}.

On the other hand, the error estimates may be proven without assuming the upper bound on density \eqref{ass}. However, we would obtain worse convergence rate and need some constrain on $\gamma$. We leave the details to interested readers.  
\end{Remark}

\begin{proof}[Proof of Theorem~\ref{THM:ES}]
The estimates \eqref{ee-u} and \eqref{ee-d} directly follow from  Lemmas~\ref{lmRE} and \ref{lmSP2} once the estimates \eqref{REI2} is proven. Hence, the key of the proof is to show \eqref{REI2}.  In what follows we show the main idea of the proof leaving the technical details to Appendix~\ref{sec_pee}. 

In order to show \eqref{REI2}, we study the relative energy balance by combining the energy estimate and the consistency formulations with suitable test functions. 
More precisely, we collect the energy estimate \eqref{ST},  the density consistency formulation \eqref{CS1-new-1} with the test function $\phi = \frac12  \abs{\tvu}^2-\Hc'(\tvr)$ and the momentum consistency formulation \eqref{CS2-new} with the test function $\bfvarphi=- \tvu$. It yields
\begin{multline}\label{REb}
 \left[ \RE(\vrh, \vuh| \tvr, \tvu)  \right]_{t=0}^{t=\tau}  +  \intTauOB{\mu \abs{\Gradd \vuh- \Grad \tvu}^2 +  \nu \abs{\Divh \vuh- \Div \tvu}^2 }
 \\
 + \frac{1}{\penl} \intTauOsh{|\vuh|^2}
  =
- \int_0^{\tau} D_{num}\, \dt + e_S  +e_R,
\end{multline}
see Appendix~\ref{sec_REB} for details. Here, $D_{num} \geq 0$ given in Lemma~\ref{thm_ST} represents the numerical dissipations. The consistency errors $e_S$ and residual errors $e_R$ are stated in \eqref{ee-SR}. They satisfy the following estimate
\begin{multline*}
\abs{e_S} +|e_R| \aleq
  h^{\beta_{RE}}  +  \int_0^\tau \RE(\vrh, \vuh| \tvr, \tvu) \dt + \frac{\delta}{\penl} {\norm{\vuh}_{L^2(0,\tau)\times\Osh)}^2}  
%\frac{\norm{\vuh}_{L^2(0,\tau)\times\Osh)}^2}{\penl}  
   \\  + \delta \mu\norm{ \Gradd  \vuh - \Grad \tvu}_{L^2((0,\tau)\times\tor)}^2 + \delta \nu\norm{ \Divh  \vuh - \Div \tvu}_{L^2((0,\tau)\times\tor)}^2. 
\end{multline*}
This inequality is obtained by splitting the grid into interior, exterior, and close to boundary parts,  since $\tvr, \Grad \tvu$ may lose regularity across the boundary. Then delicate estimates of the consistency and residual errors by means of the terms on the left hand side of \eqref{REb}
yield the desired result, see Appendix \ref{sec_ceb} (Lemma~\ref{lm_es}) and Appendix \ref{sec_er} (Lemma~\ref{lm_er}) for details. 

With the above estimate in hand, we choose any fixed $\delta \in(0,1)$ to obtain the following relative energy inequality
\begin{multline*}
\left[ \RE(\vrh, \vuh| \tvr, \tvu)  \right]_{t=0}^{t=\tau} +  \intTauTd{\left(\mu \abs{\Gradd \vuh- \Grad \tvu}^2 +  \nu \abs{\Divh \vuh- \Div \tvu}^2 \right)}
\\
 + \frac{1}{\penl} \intTauOsh{|\vuh|^2}
 \aleq   \int_0^\tau  \RE(\vrh, \vuh| \tvr, \tvu) \dt     + h^{\beta_{RE}} + \frac{h^3}{\penl}   + \frac{\penl}{h}. 
\end{multline*}
Finally, by Gronwall's lemma and the continuity of the initial data, we obtain from the above inequality that
\begin{multline*}
\RE(\vrh, \vuh| \tvr, \tvu)(\tau)  +  \intTauTdB{\mu \abs{\Gradd \vuh- \Grad \tvu}^2 +  \nu \abs{\Divh \vuh- \Div \tvu}^2 }   + \frac{1}{\penl} \intTauOsh{|\vuh|^2}
\\ \aleq
 \frac{h^3}{\penl} + \frac{\penl}{h} +h^{\beta_{RE}} + \RE(\vrh^0, \vuh^0 | \Pim \tvr_0, \Pim \tvu_0) \aleq  \frac{h^3}{\penl} + \frac{\penl}{h}+ h^{\beta_{RE}}.
\end{multline*}
This concludes the proof. 
\end{proof}

%\iffalse

\section{Numerical experiments}\label{sec_num}
In this section our aim is to validate the obtained theoretical convergence results.
To this end, we compute the following errors. Firstly, the errors with respect to the discretization parameter $h$, i.e.~we fix a penalty parameter  $\penl$ and use a reference solution with a small $h_{ref}$:
\begin{align*}
&E_{\vr}^{\penl} = \| \vrh^{\penl} - \vr_{h_{ref}}^{\penl}  \|_{L^\gamma(\mathbb{T}^2)}, \hspace{2.3cm}
E_{\vu}^{\penl}= \| \vuh^{\penl} - \vu_{h_{ref}}^{\penl}  \|_{L^{2}(\mathbb{T}^2)}, \quad
\\&
E_{\nabla \vu}^{\penl}= \| \Gradd \vuh^{\penl} - \Gradd \vu_{h_{ref}}^{\penl}  \|_{L^2(\mathbb{T}^2)}, \hspace{1cm}
\RE^{\penl}=\RE \left(\vrh^{\penl}, \vuh^{\penl} ~ \big| ~ \vr_{h_{ref}}^{\penl}  , \vu_{h_{ref}}^{\penl}  \right).
\end{align*}
Secondly, we consider the errors with respect to two parameters $h,\, \penl$. Thus,  the reference solution has a fixed parameter pair $(h_{ref}, \penl_{ref})$:
\begin{align*}
&E_{\vr} = \| \vrh^{\penl} - \vr_{h_{ref}}^{\penl_{ref}} \|_{L^\gamma(\mathbb{T}^2)}, \hspace{2.4cm}
E_{\vu}= \| \vuh^{\penl} - \vu_{h_{ref}}^{\penl_{ref}} \|_{L^{2}(\mathbb{T}^2)}, \quad
\\&
E_{\nabla \vu}= \| \Gradd \vuh^{\penl} - \Gradd \vu_{h_{ref}}^{\penl_{ref}} \|_{L^2(\mathbb{T}^2)}, \hspace{1cm}
\RE=\RE \left(\vrh^{\penl}, \vuh^{\penl} ~ \big| ~ \vr_{h_{ref}}^{\penl_{ref}} , \vu_{h_{ref}}^{\penl_{ref}} \right).
\end{align*}
In the simulation we take the following parameters
\begin{equation*}
\viso = 0.6, \ T=0.1, \;  \mu = 0.1, \ \nu = 0, \   a = 1, \ \gamma = 1.4.
\end{equation*}

We recall that $E_\vr^{\penl}, E_{\vu}^{\penl}, E_{\Grad \vu}^{\penl}, R_E^{\penl}$ are used to verify the convergence rate with respect to mesh parameter $h$, cf. Theorem \ref{THM1}.
Errors $E_\vr, E_{\vu}, E_{\Grad \vu}, R_E$ (with respect to the parameter pair $(h,\penl(h))$) are used to illustrate our convergence results in Theorem~\ref{THM2}, Theorem~\ref{THM3} and Theorem~\ref{THM:ES}.

\subsection{Experiment~1:  Ring domain - continuous extension}
In this experiment we take the physical fluid domain to be a ring, i.e. $\Of \equiv B_{0.7}\setminus  \Ov{B_{0.2}}$, where $ B_{r} = \{x ~\big|~ |x| < r\}$.
The initial data (including the smooth extension) are given by
\begin{equation*}
	(\vr, \vu)(0,x)
	\; = \; \begin{cases}
	(1, \, 0,  \, 0 ) , & x \in B_{0.2}, \\
	\left(1, \,  \frac{ \sin(4\pi (|x|-0.2)) x_2}{|x|} , \, -\frac{ \sin(4\pi (|x|-0.2)) x_1}{|x|} \right) , & x \in  \Of \equiv {B}_{0.7}\setminus \Ov{B_{0.2}}, \\
	(1, \, 0 ,  \, 0) , & x \in \mathbb{T}^2\setminus B_{0.7}.
	\end{cases}
\end{equation*}
Figure \ref{fig:ex1} shows the numerical solutions $\vrh$ and $\vuh$ at time $T=0.1$ with fixed mesh size $h=0.2\cdot 2^{-4}$ and various penalty parameter $\penl = 4^{-3},\dots,4^{-6}$.
We can observe that the velocity vanishes in the penalized region with decreasing $\penl$.
Further, in Figure~\ref{fig:ex1-1} we present the errors $E_\vr^{\penl}, E_{\vu}^{\penl}, E_{\Grad \vu}^{\penl}, R_E^{\penl}$ with respect to $h = 0.2\cdot 2^{-m}, m = 0,\dots,3 $ for fixed $\penl \in \{4^{-2},4^{-3},4^{-4},4^{-5},4^{-6} \} $ and $h_{ref} = 0.2\cdot 2^{-4}$.
Figure \ref{fig:ex1-3} depicts the errors $E_\vr, E_{\vu}, E_{\Grad \vu}, R_E$
with respect to the three parameter pairs $(h,\penl(h)) = (h, \mathcal{O}(h^{1/2})), \ (h, \mathcal{O}(h^2))$ and $ (h, \mathcal{O}(h^4))$ given by
\begin{align*}
& (h,\penl(h)) = \left(0.2\cdot 2^{-m}, 2^{-(m+14)/2} \right), \quad m = 0,\dots,3 \quad \mbox{with} \quad h_{ref} = 0.2\cdot 2^{-4}, \ \penl_{ref} = 2^{-9}; \br
& (h,\penl(h)) = \left(0.2\cdot 2^{-m}, 4^{-(m+2)} \right), \quad m = 0,\dots,3 \quad \mbox{with} \quad h_{ref} = 0.2\cdot 2^{-4}, \ \penl_{ref} = 4^{-6}; \br
& (h,\penl(h)) = \left(0.2\cdot 2^{-m}, 16^{-m} \right), \quad m = 0,\dots,3 \quad \mbox{with} \quad h_{ref} = 0.2\cdot 2^{-4}, \ \penl_{ref} = 16^{-4}.
\end{align*}

Our numerical results indicate  first order convergence rate for $\vr,\vu, \Grad \vu$ and  second order convergence rate for $\RE$. 
Note that the experimental convergence rates are better than our theoretical result.

\begin{figure}[htbp]	\centering
	\setlength{\abovecaptionskip}{0.cm}
	\setlength{\belowcaptionskip}{-0.cm}
\vspace{-1em}
	\begin{subfigure}{0.48\textwidth}
		\includegraphics[width=\textwidth]{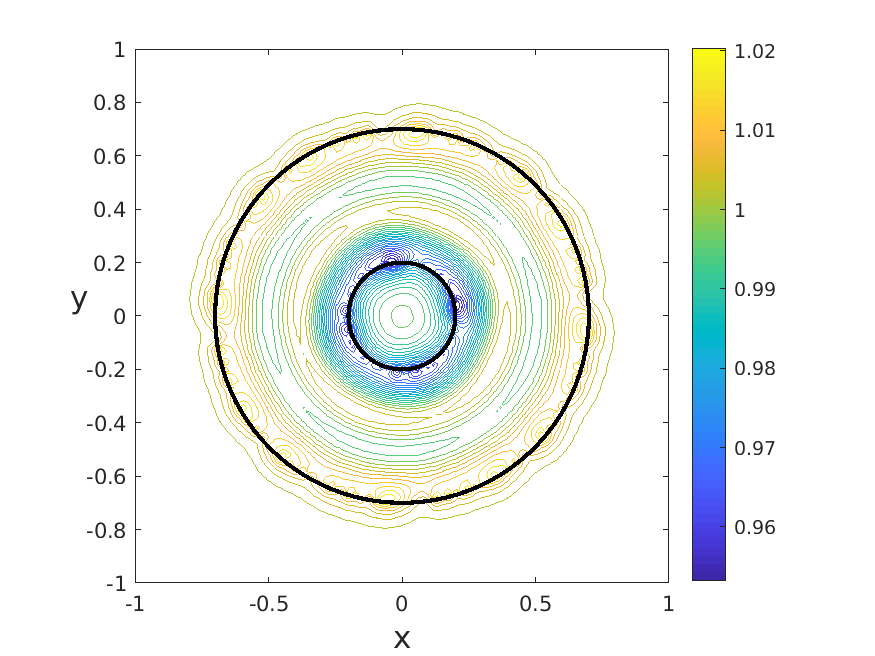}
	\end{subfigure}
		\begin{subfigure}{0.48\textwidth}\centering
			\includegraphics[width=\textwidth]{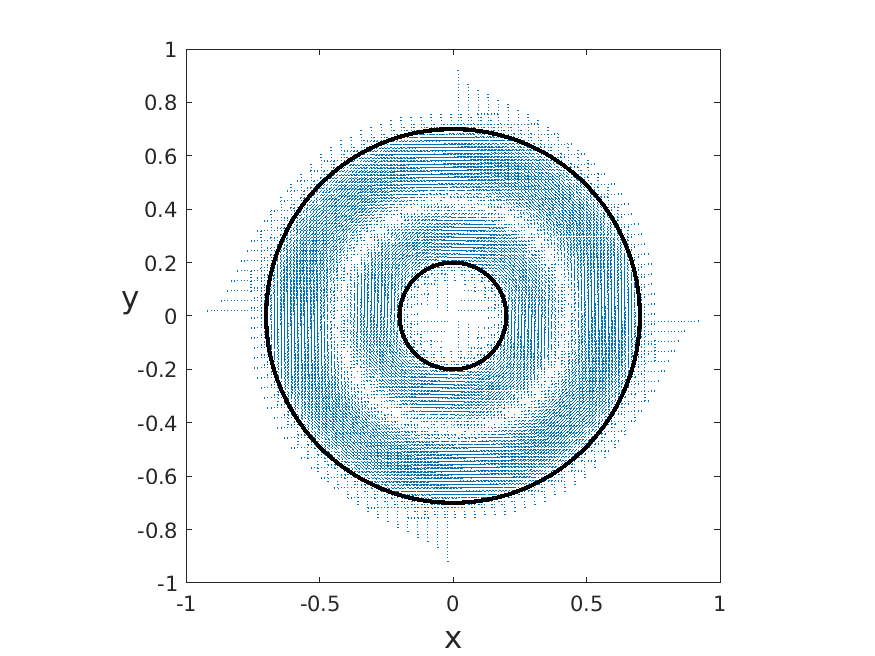}
	\end{subfigure}
\\ \vspace{-1em}
	\begin{subfigure}{0.48\textwidth}
		\includegraphics[width=\textwidth]{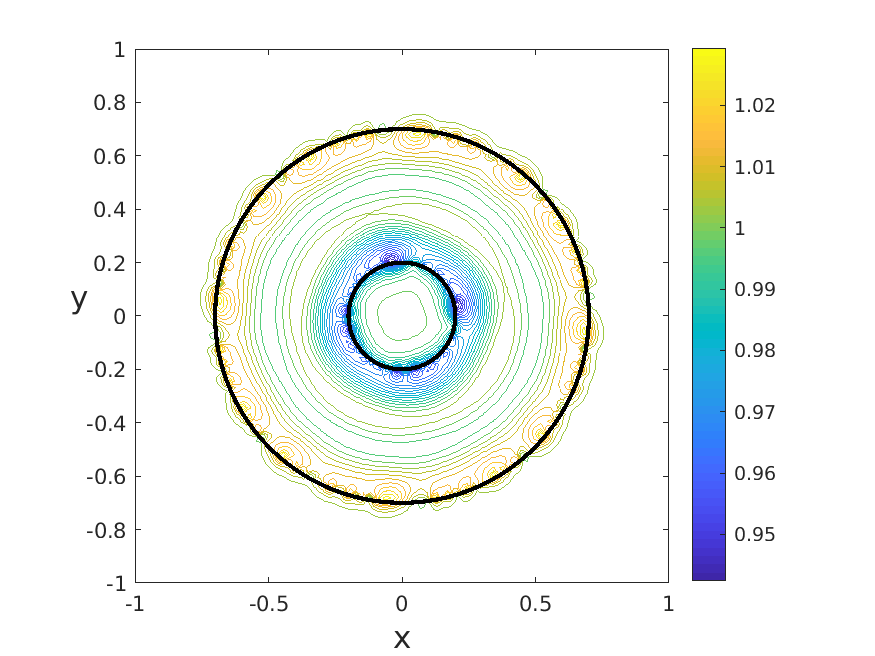}
	\end{subfigure}
	\begin{subfigure}{0.48\textwidth}
		\includegraphics[width=\textwidth]{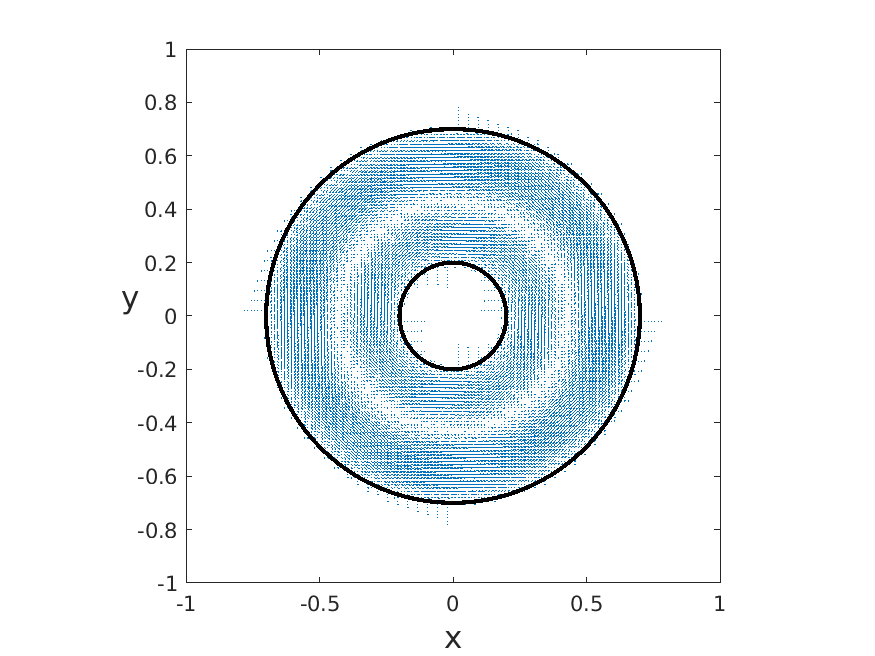}
	\end{subfigure}	
\\ \vspace{-1em}
	\begin{subfigure}{0.48\textwidth}
		\includegraphics[width=\textwidth]{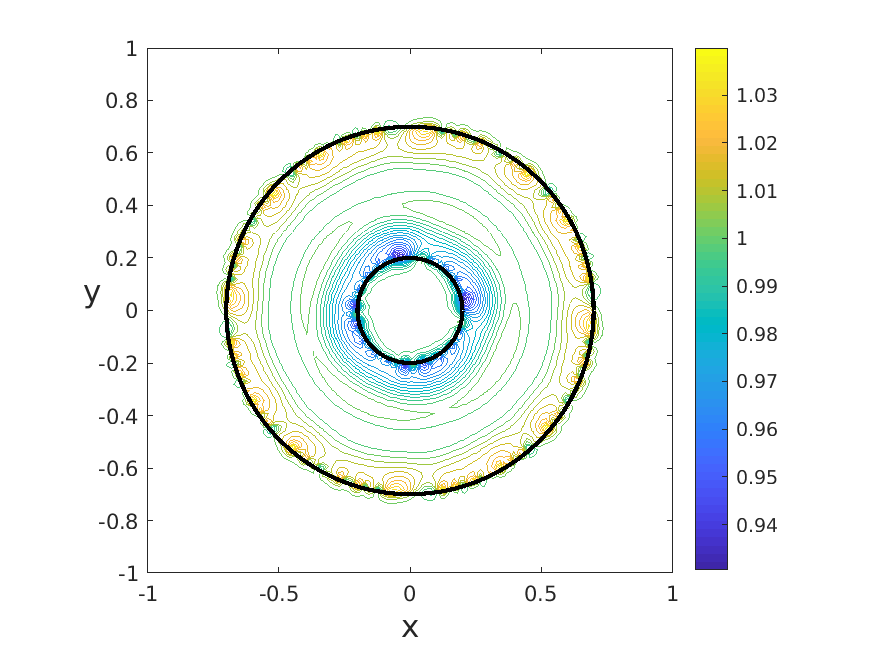}
	\end{subfigure}
	\begin{subfigure}{0.48\textwidth}
		\includegraphics[width=\textwidth]{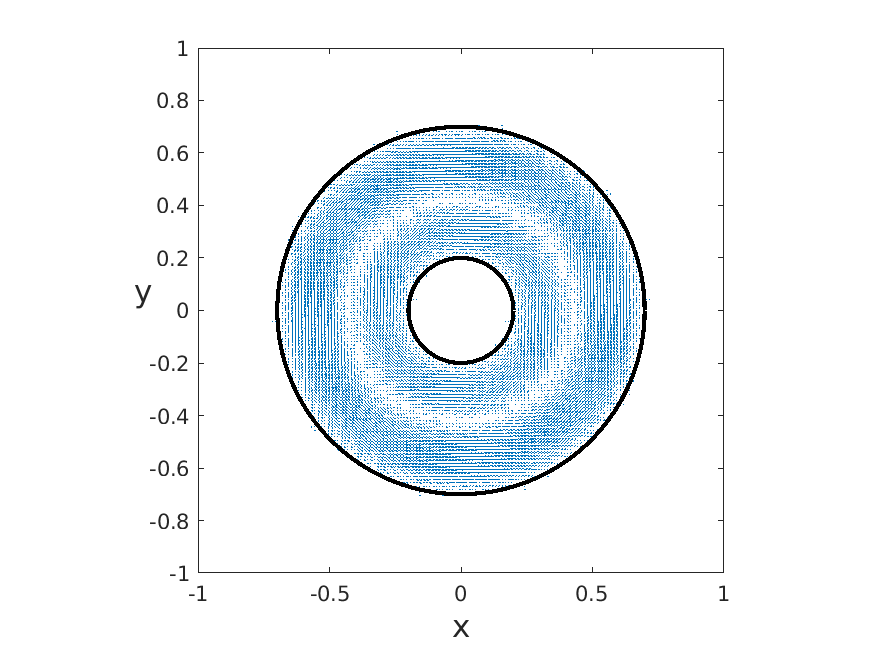}
	\end{subfigure}		
\\ \vspace{-1em}
	\begin{subfigure}{0.48\textwidth}
		\includegraphics[width=\textwidth]{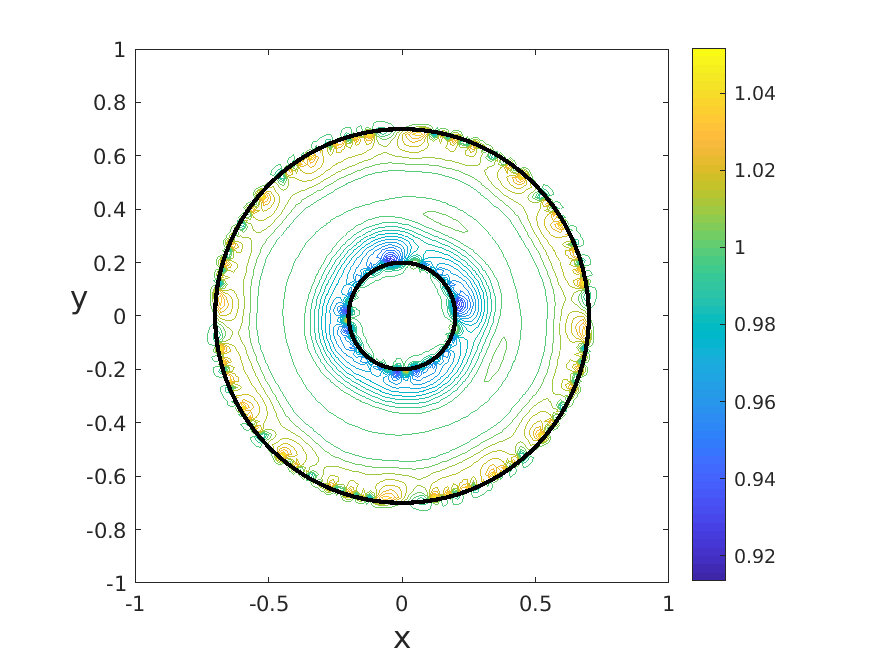}
	\end{subfigure}
	\begin{subfigure}{0.48\textwidth}
		\includegraphics[width=\textwidth]{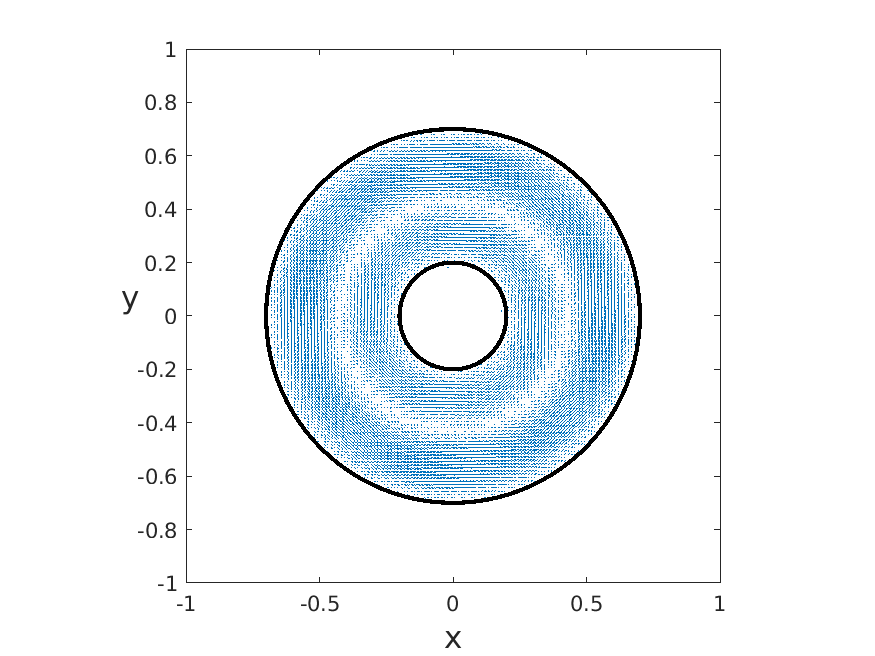}
	\end{subfigure}
	\caption{Experiment~1:
	Numerical solutions $\vrh$  (left) and $\vuh$ (right) obtained with $h = 0.2 \cdot 2^{-4}$ for different $\penl = 4^{-m-2}, m = 1, \dots, 4$ from top to bottom.}
	\label{fig:ex1}
\end{figure}

\begin{figure}[htbp]
	\setlength{\abovecaptionskip}{0.cm}
	\setlength{\belowcaptionskip}{-0.cm}
	\centering
	 \vspace{-1em}
	\begin{subfigure}{0.48\textwidth}
		\includegraphics[width=\textwidth]{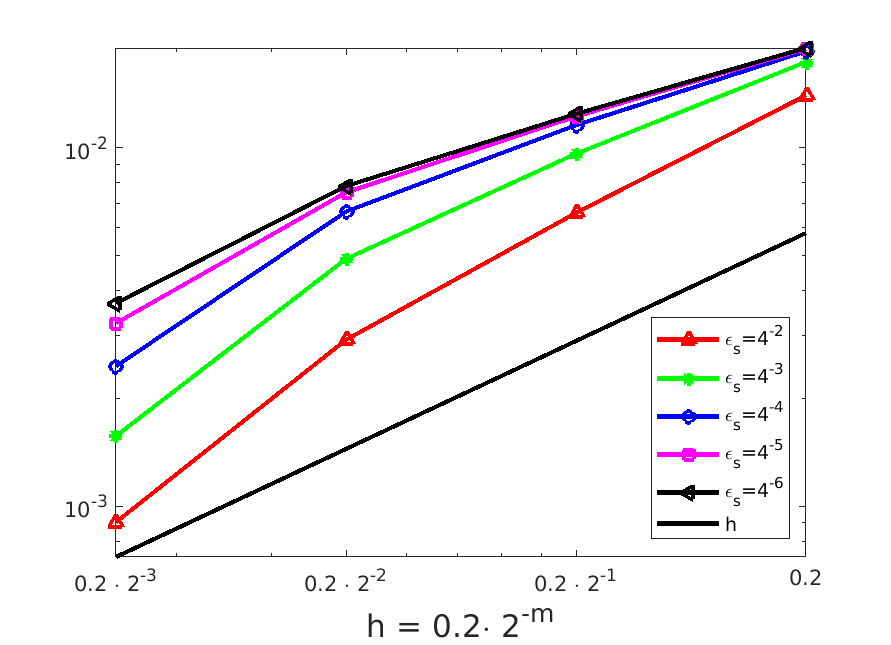}
		\caption{  $E_{\vr}^{\penl} $}
	\end{subfigure}
	\begin{subfigure}{0.48\textwidth}
		\includegraphics[width=\textwidth]{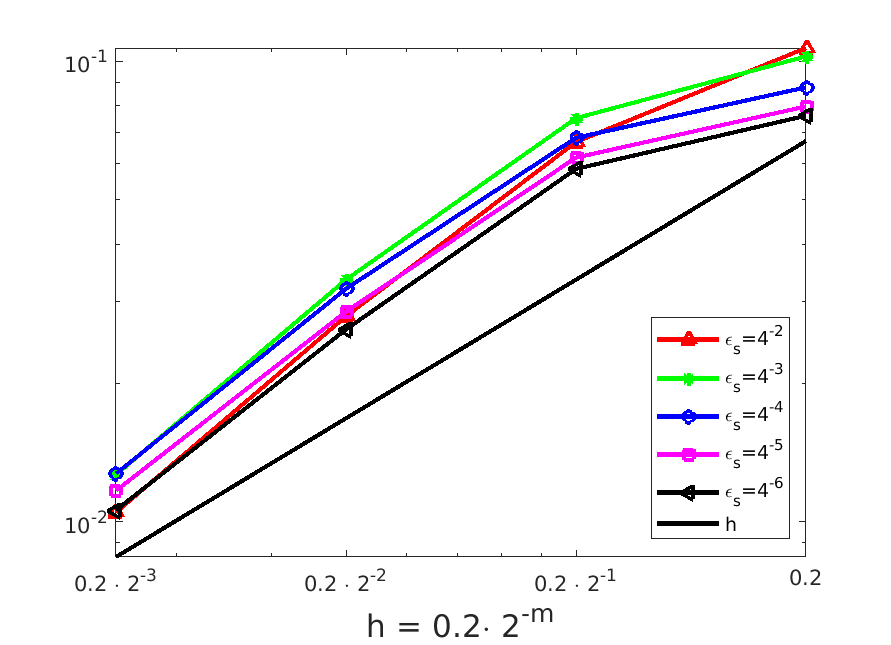}
		\caption{  $E_{\vu}^{\penl}  $}
	\end{subfigure}\\
	\begin{subfigure}{0.48\textwidth}
		\includegraphics[width=\textwidth]{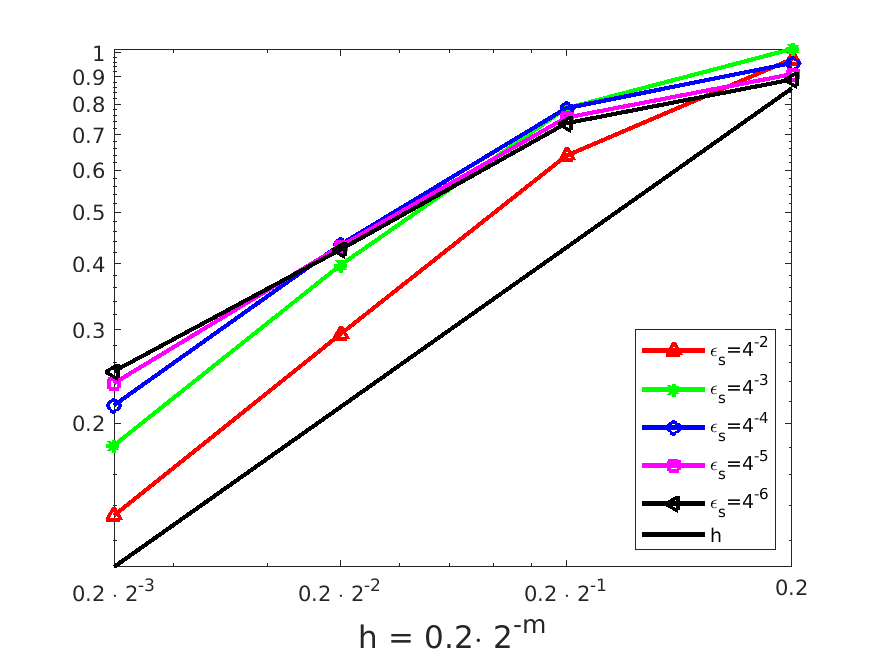}
		\caption{$E_{\Grad \vu}^{\penl}$ }
	\end{subfigure}
	\begin{subfigure}{0.48\textwidth}
		\includegraphics[width=\textwidth]{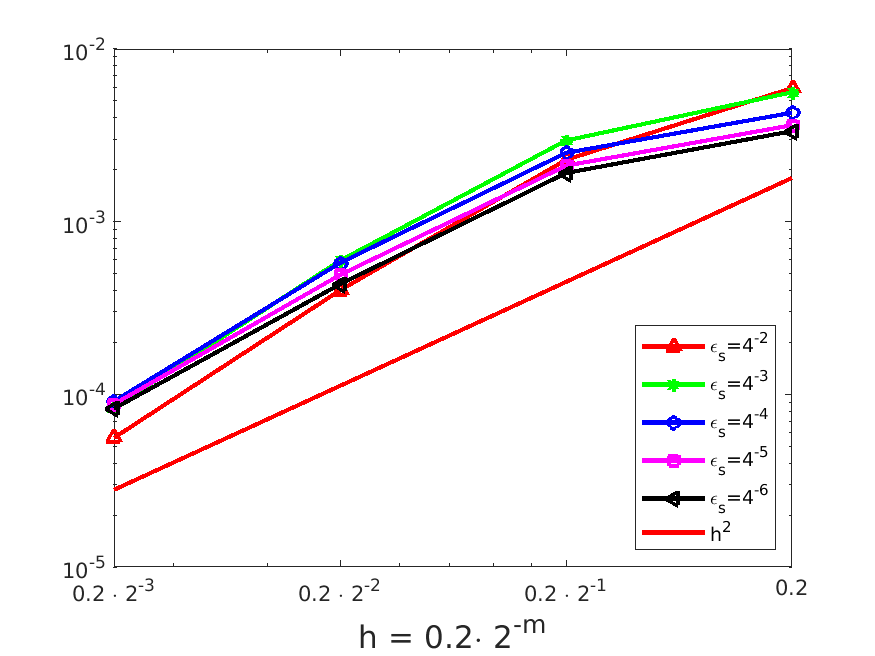}
		\caption{$R_E^{\penl}$ }
	\end{subfigure}
	\caption{Experiment~1:
	   The errors $E_\vr^{\penl}, E_{\vu}^{\penl}, E_{\Grad \vu}^{\penl}, R_E^{\penl}$ with respect to $h$ for different but fixed $\penl$. The black and red solid lines without any marker denote the reference slope of $h$ and $h^2$, respectively.
	}\label{fig:ex1-1}
	 \vspace{-1em}
\end{figure}

\begin{figure}[htbp]
	\setlength{\abovecaptionskip}{0.cm}
	\setlength{\belowcaptionskip}{-0.cm}
	\centering
	\begin{subfigure}{0.32\textwidth}
		\includegraphics[width=\textwidth]{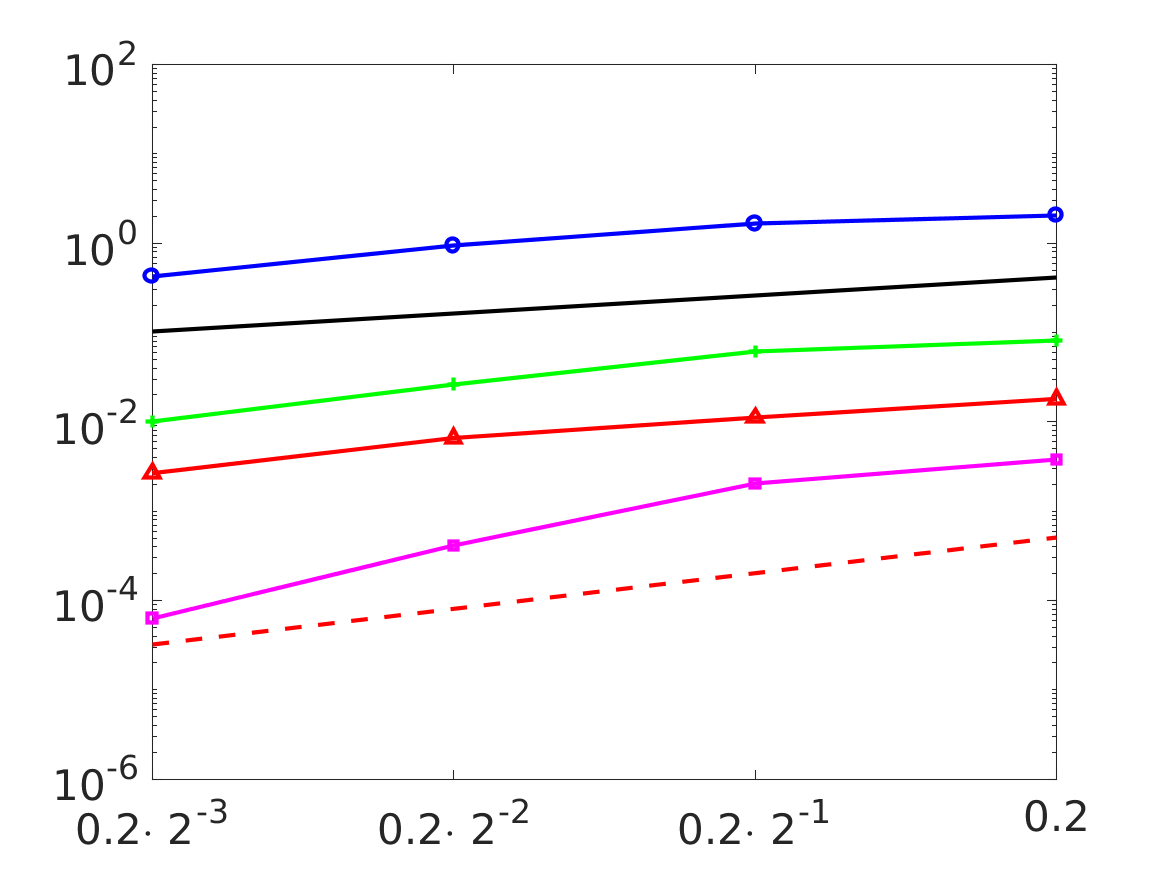}
%		\caption{ \bf $(h, \penl^{(1)}(h))$}
	\end{subfigure}
	\begin{subfigure}{0.32\textwidth}
		\includegraphics[width=\textwidth]{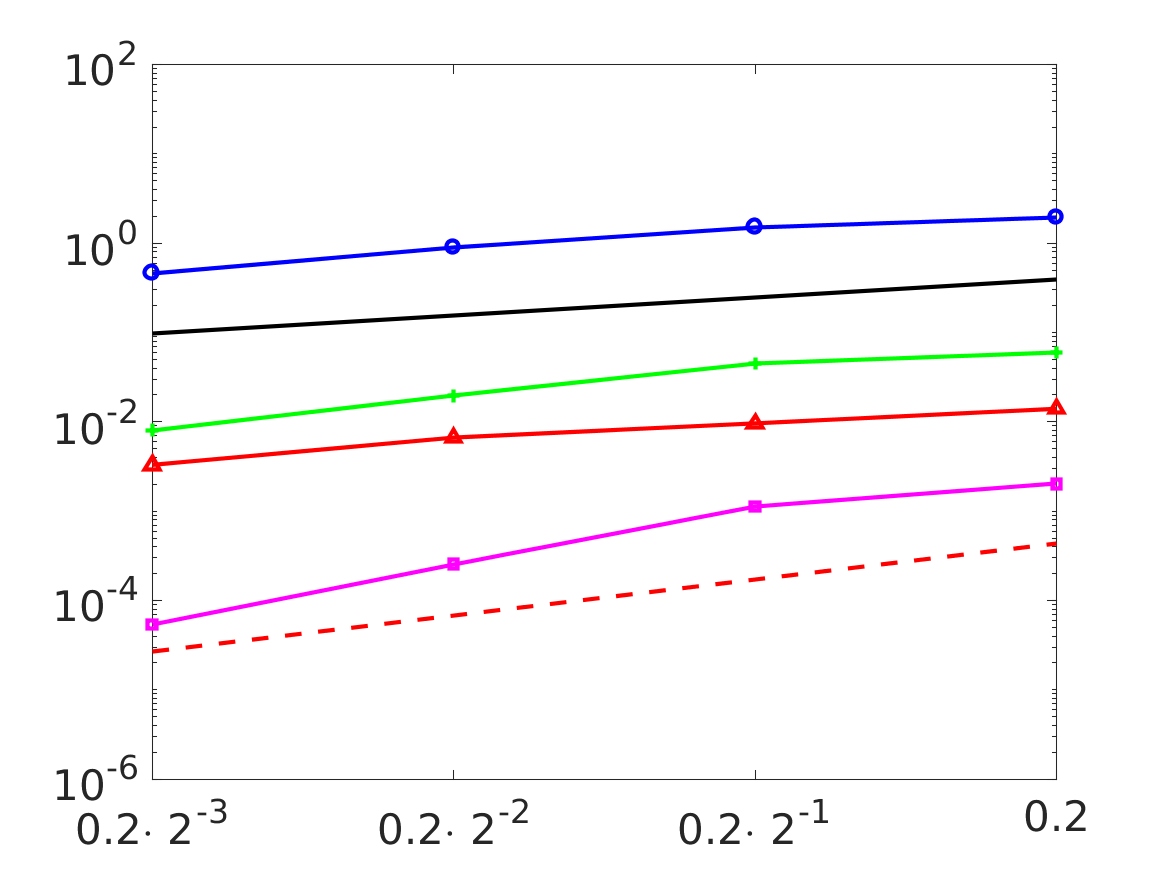}
%		\caption{ \bf $(h, \penl^{(2)}(h))$}
	\end{subfigure}
	\begin{subfigure}{0.32\textwidth}
		\includegraphics[width=\textwidth]{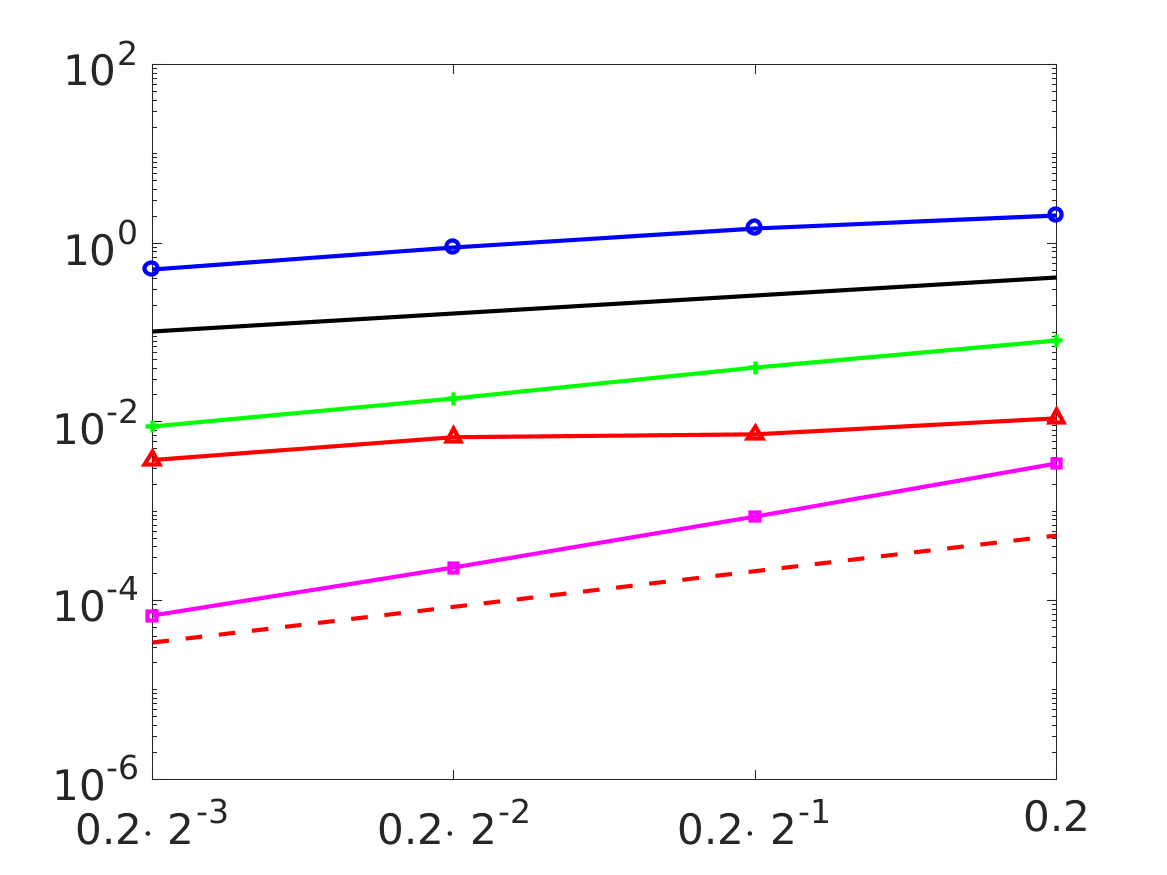}
%		\caption{ \bf $(h, \penl^{(2)}(h))$}
	\end{subfigure}\\
	\begin{subfigure}{0.8\textwidth}
		\includegraphics[width=\textwidth]{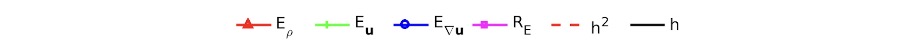}
	\end{subfigure}
	\caption{Experiment~1:
	Errors $E_\vr, E_{\vu}, E_{\Grad \vu}$ and relative energy $\RE$ with respect to the pairs $(h,\penl(h)) = \left(0.2\cdot 2^{-m}, 2^{-(m+14)/2} \right)$ (left), $\ \left(0.2\cdot 2^{-m}, 4^{-(m+2)} \right)$ (middle) and $\left(0.2\cdot 2^{-m}, 16^{-m} \right)$ (right), $m = 0,1,2,3$.
	The black solid and red dashed lines without any marker denote the reference slope of $h$ and $h^2$, respectively.
	}
	\label{fig:ex1-3}
		 \vspace{-1em}
\end{figure}

\subsection{Experiment~2: Ring domain - discontinuous extension}
In the second experiment we consider the same physical fluid domain, but different initial extension of density, i.e.
\begin{equation*}
	(\vr, \vu)(0,x)
	\; = \; \begin{cases}
	(0.01,\, 0, \,  0 ) , & x \in B_{0.2}, \\
	\left(1, \,  \frac{ \sin(4\pi (|x|-0.2)) x_2}{|x|} ,\,  -\frac{ \sin(4\pi (|x|-0.2)) x_1}{|x|} \right) , & x \in  \Of \equiv {B}_{0.7}\setminus \Ov{B_{0.2}}, \\
	(2,\,  0 , \,  0) , & x \in \mathbb{T}^2\setminus B_{0.7}.
	\end{cases}
\end{equation*}
The effect of different penalty parameters,  $\penl = 4^{-3},\dots,4^{-6},$ is present in Figure ~\ref{fig:ex2}.
The errors $E_\vr^{\penl}, E_{\vu}^{\penl}, E_{\Grad \vu}^{\penl}, R_E^{\penl}$ with respect to $h$ for fixed penalty parameters are shown in Figure~\ref{fig:ex2-1}. Figure~\ref{fig:ex2-3} presents the errors  $E_\vr, E_{\vu}, E_{\Grad \vu}, R_E$ with respect to the pair $(h,\penl(h)) = (h, \mathcal{O}(h^{1/2}))$, $(h, \mathcal{O}(h^2))$ and $(h, \mathcal{O}(h^4))$.
Figures \ref{fig:ex2-1} and \ref{fig:ex2-3} indicate similar convergence behaviours as in Experiment~1.

\begin{figure}[htbp]
	\setlength{\abovecaptionskip}{0.cm}
	\setlength{\belowcaptionskip}{-0.cm}
	 \vspace{-1em}
	\centering
	\begin{subfigure}{0.48\textwidth}
		\includegraphics[width=\textwidth]{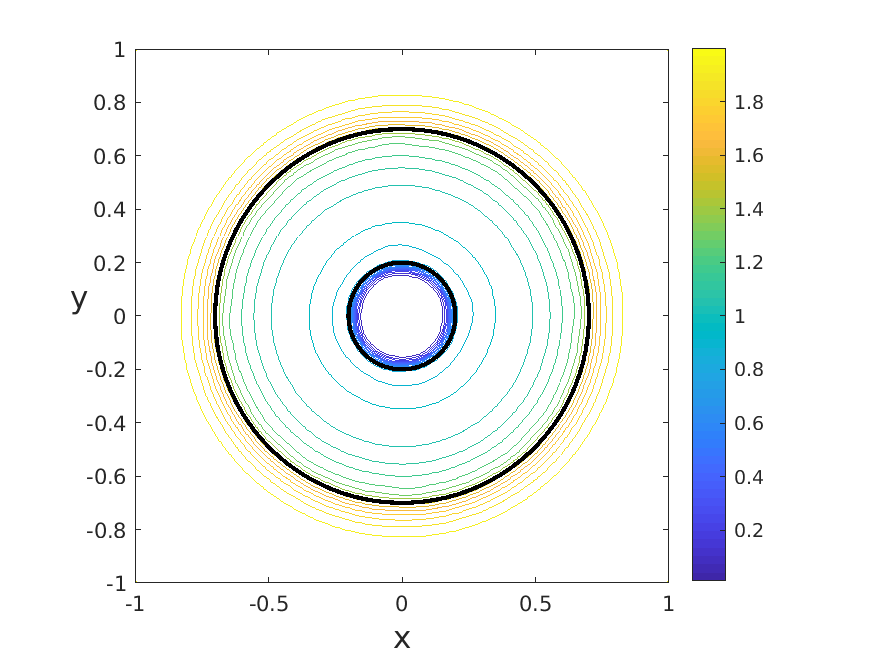}
	\end{subfigure}
	\begin{subfigure}{0.48\textwidth}
	\centering
			\includegraphics[width=\textwidth]{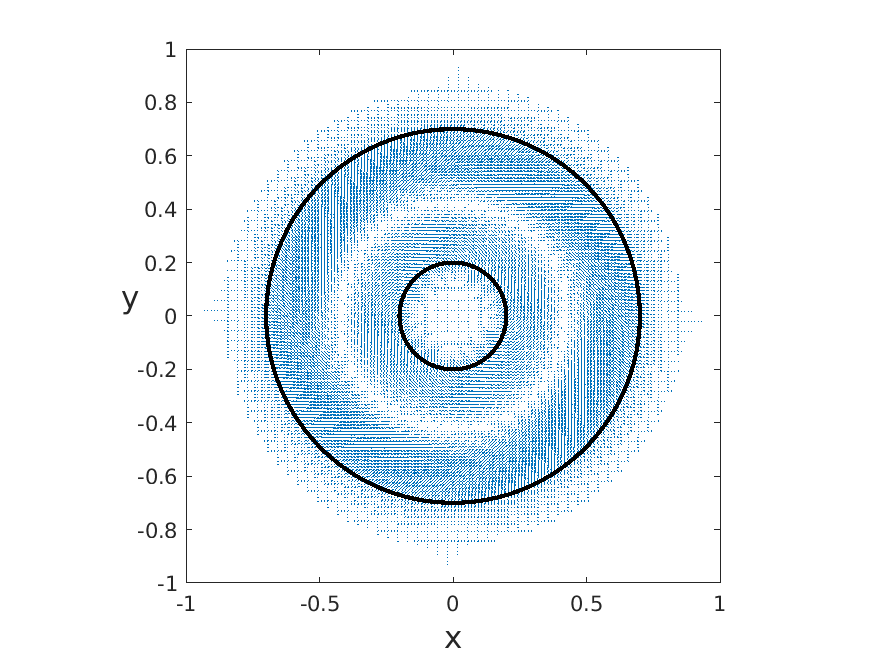}
	\end{subfigure}
\\ \vspace{-1em}
	\begin{subfigure}{0.48\textwidth}
		\includegraphics[width=\textwidth]{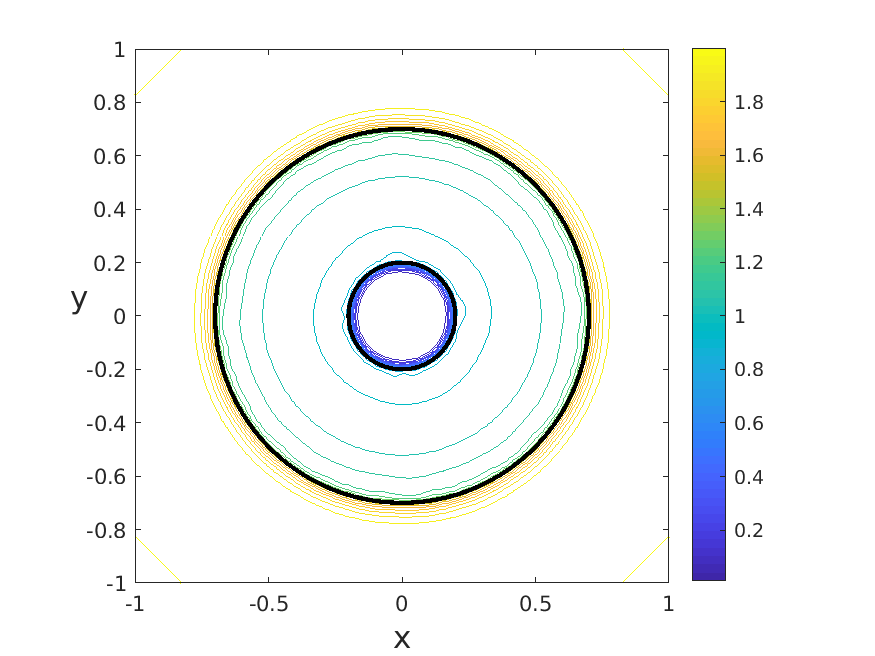}
	\end{subfigure}
		\begin{subfigure}{0.48\textwidth}
		\includegraphics[width=\textwidth]{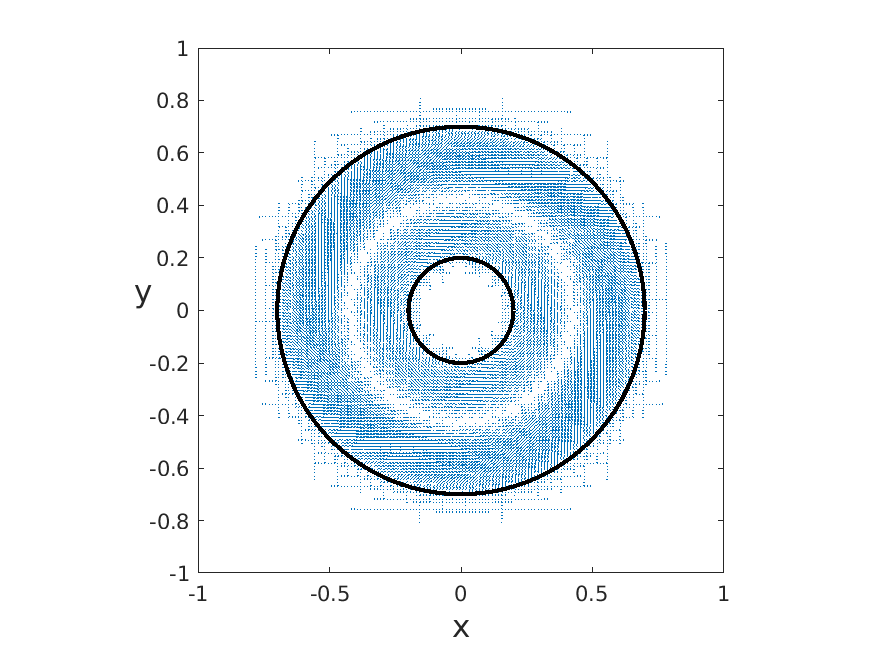}
	\end{subfigure}	
\\ \vspace{-1em}
	\begin{subfigure}{0.48\textwidth}
		\includegraphics[width=\textwidth]{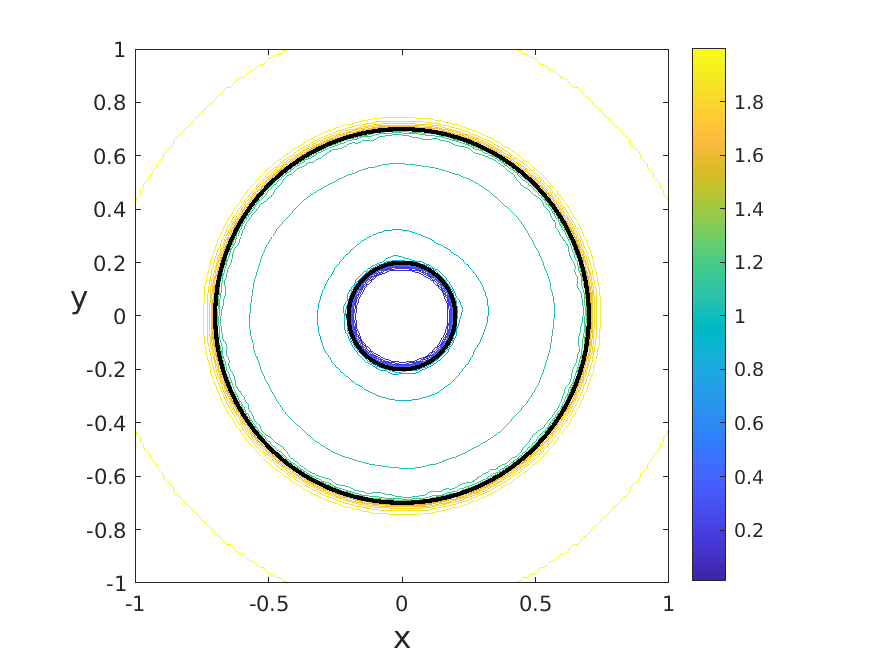}
	\end{subfigure}
		\begin{subfigure}{0.48\textwidth}
		\includegraphics[width=\textwidth]{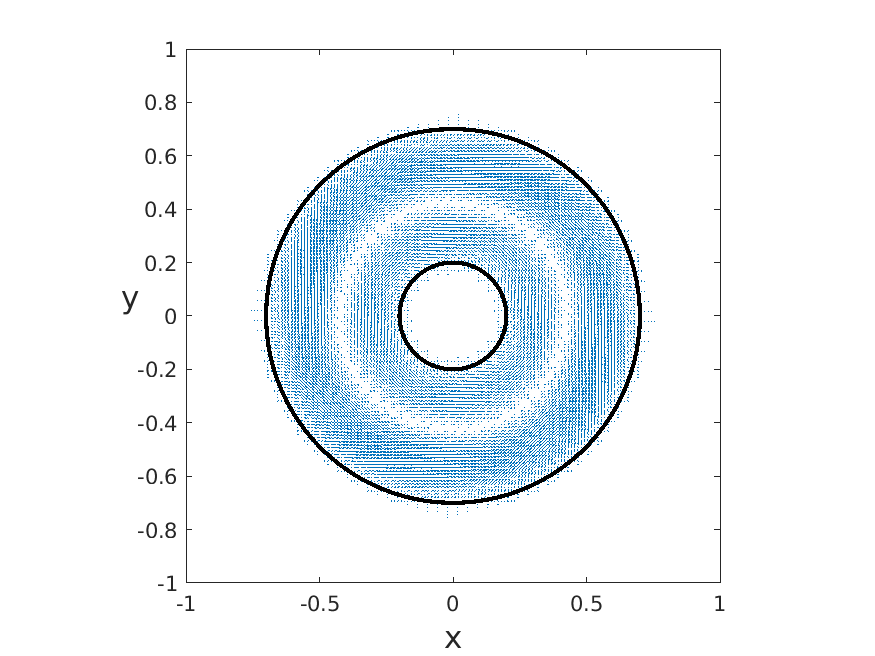}
	\end{subfigure}		
\\ \vspace{-1em}	
	\begin{subfigure}{0.48\textwidth}
		\includegraphics[width=\textwidth]{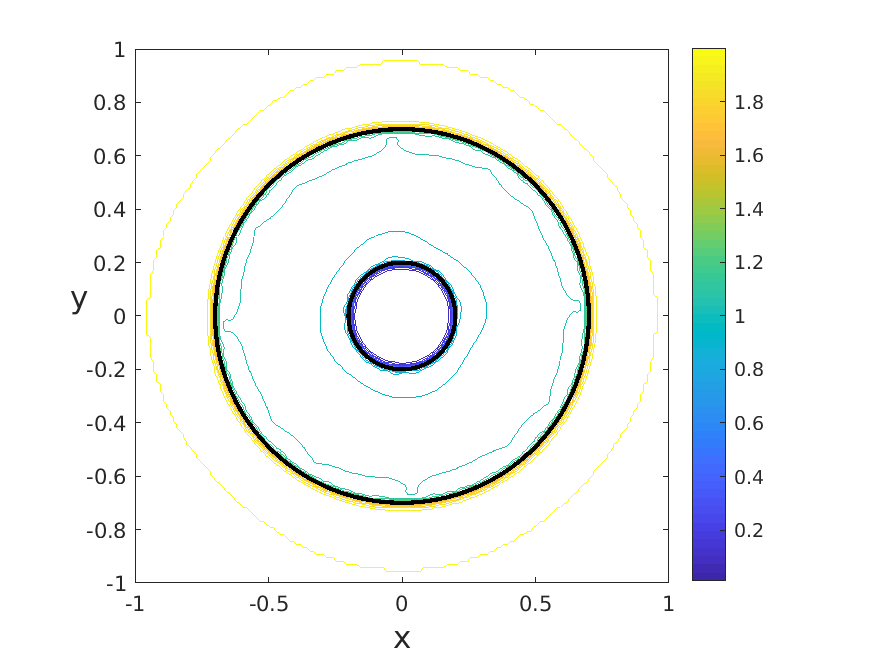}
	\end{subfigure}
	\begin{subfigure}{0.48\textwidth}
		\includegraphics[width=\textwidth]{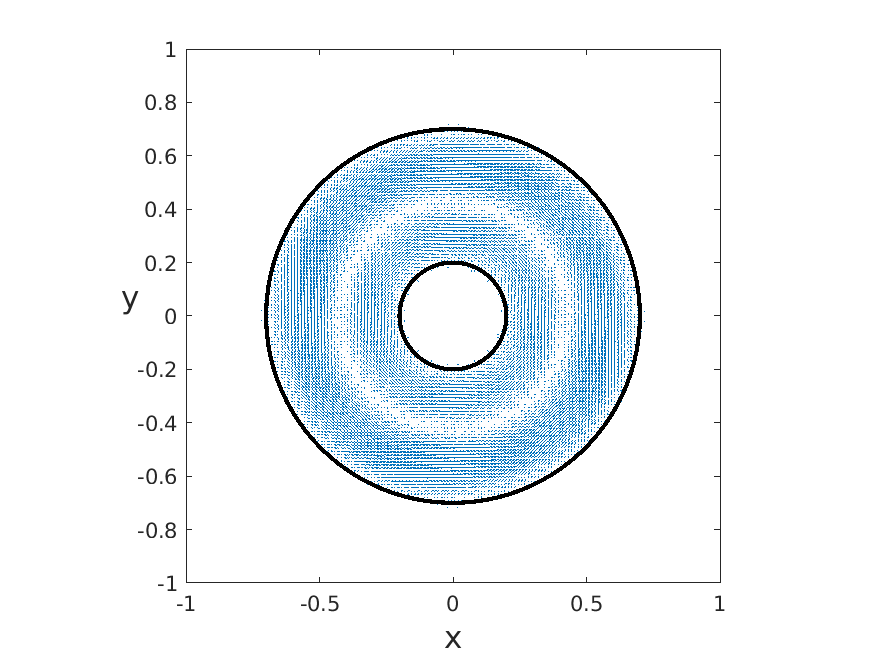}
	\end{subfigure}
	\caption{Experiment~2:
	Numerical solutions $\vrh$  (left) and $\vuh$ (right) obtained with $h = 0.2 \cdot 2^{-4}$ for different $\penl = 4^{-m-2}, m = 1, \dots, 4,$ from top to bottom.}
	\label{fig:ex2}
	\vspace{-2em}	
\end{figure}

\begin{figure}[htbp]
	\setlength{\abovecaptionskip}{0.cm}
	\setlength{\belowcaptionskip}{-0.cm}
	\centering
	\vspace{-2em}	
	\begin{subfigure}{0.48\textwidth}
		\includegraphics[width=\textwidth]{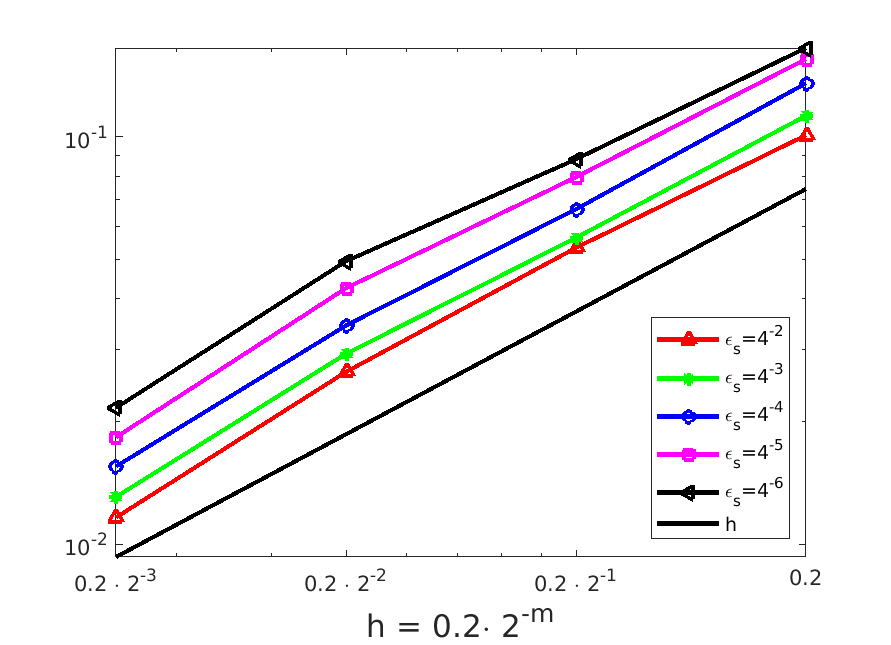}
		\caption{  $E_{\vr}^{\penl} $}
	\end{subfigure}
	\begin{subfigure}{0.48\textwidth}
		\includegraphics[width=\textwidth]{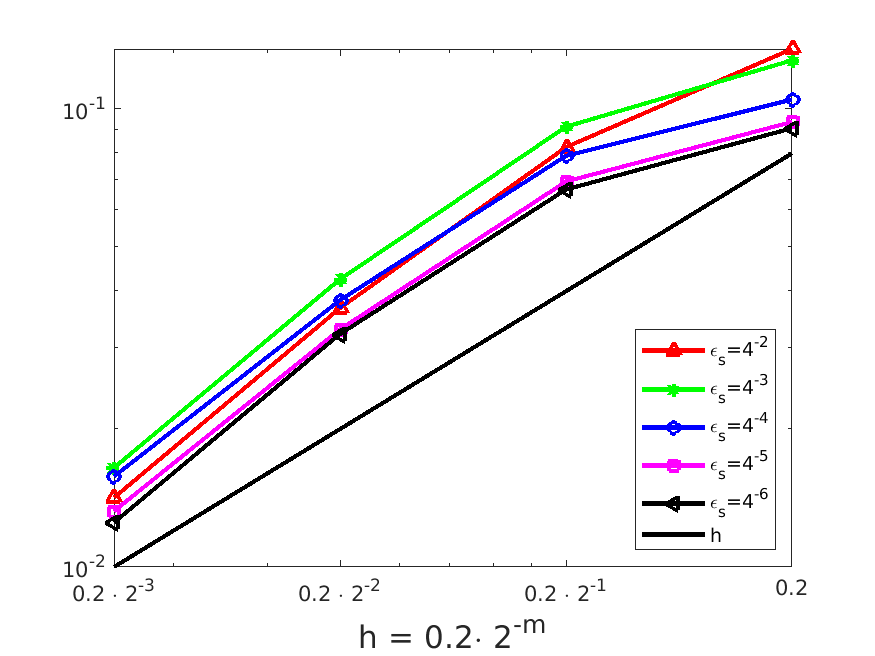}
		\caption{  $E_{\vu}^{\penl}  $}
	\end{subfigure}\\
	\begin{subfigure}{0.48\textwidth}
		\includegraphics[width=\textwidth]{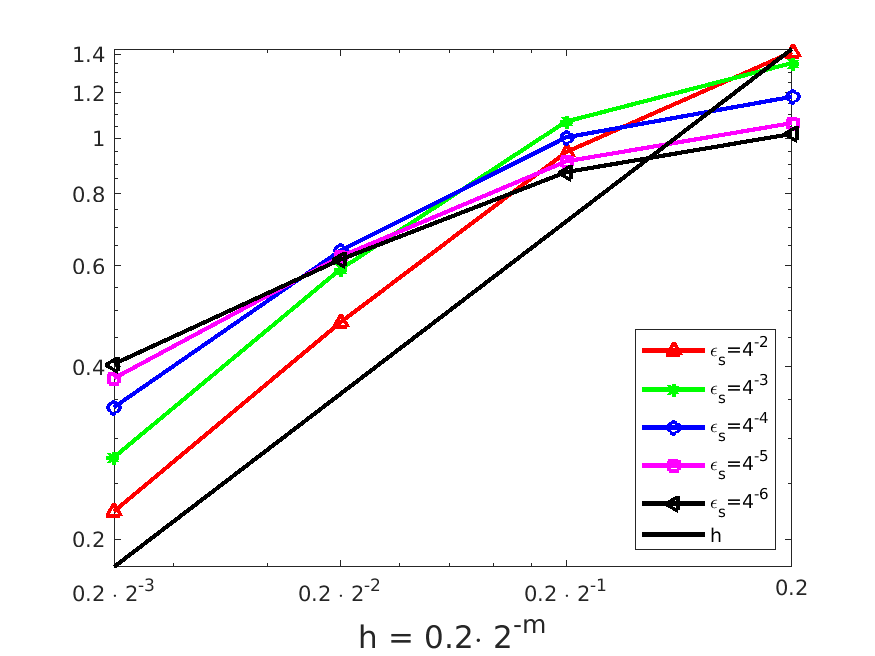}
		\caption{$E_{\Grad \vu}^{\penl}$ }
	\end{subfigure}
	\begin{subfigure}{0.48\textwidth}
		\includegraphics[width=\textwidth]{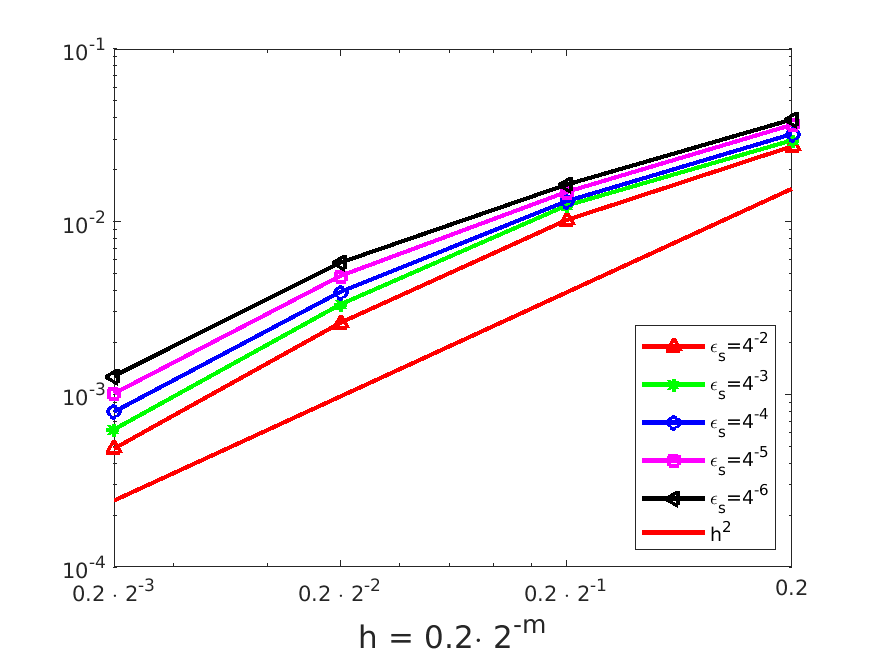}
		\caption{$R_E^{\penl}$ }
	\end{subfigure}
	\caption{Experiment~2:
	   The errors $E_\vr^{\penl}, E_{\vu}^{\penl}, E_{\Grad \vu}^{\penl}, R_E^{\penl}$ with respect to $h$ for different but fixed $\penl$. The black and red solid lines without any marker denote the reference slope of $h$ and $h^2$, respectively.
	}\label{fig:ex2-1}
		\vspace{-1em}	
\end{figure}

\begin{figure}[htbp]
	\setlength{\abovecaptionskip}{0.cm}
	\setlength{\belowcaptionskip}{-0.cm}
	\centering
	\begin{subfigure}{0.32\textwidth}
		\includegraphics[width=\textwidth]{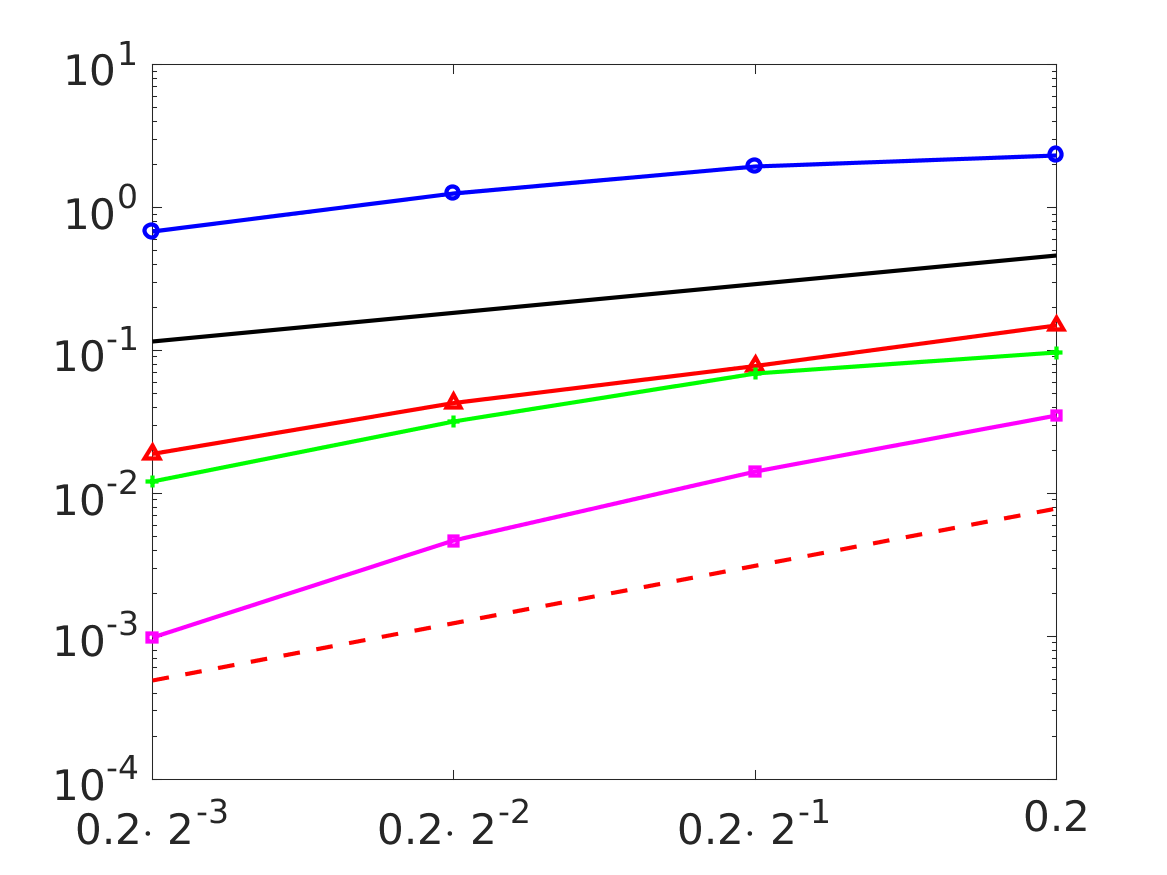}
%		\caption{ \bf $(h, \penl^{(1)}(h))$}
	\end{subfigure}
	\begin{subfigure}{0.32\textwidth}
		\includegraphics[width=\textwidth]{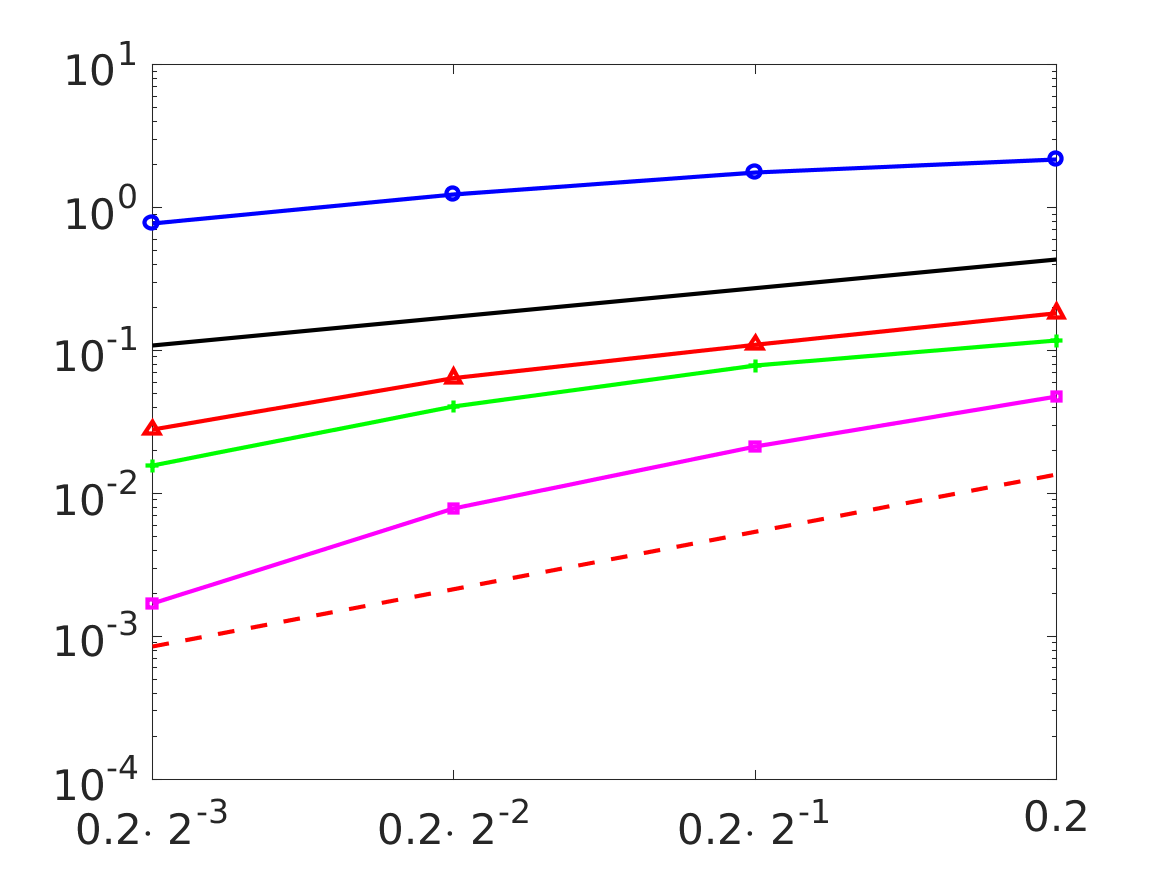}
%		\caption{ \bf $(h, \penl^{(2)}(h))$}
	\end{subfigure}
	\begin{subfigure}{0.32\textwidth}
		\includegraphics[width=\textwidth]{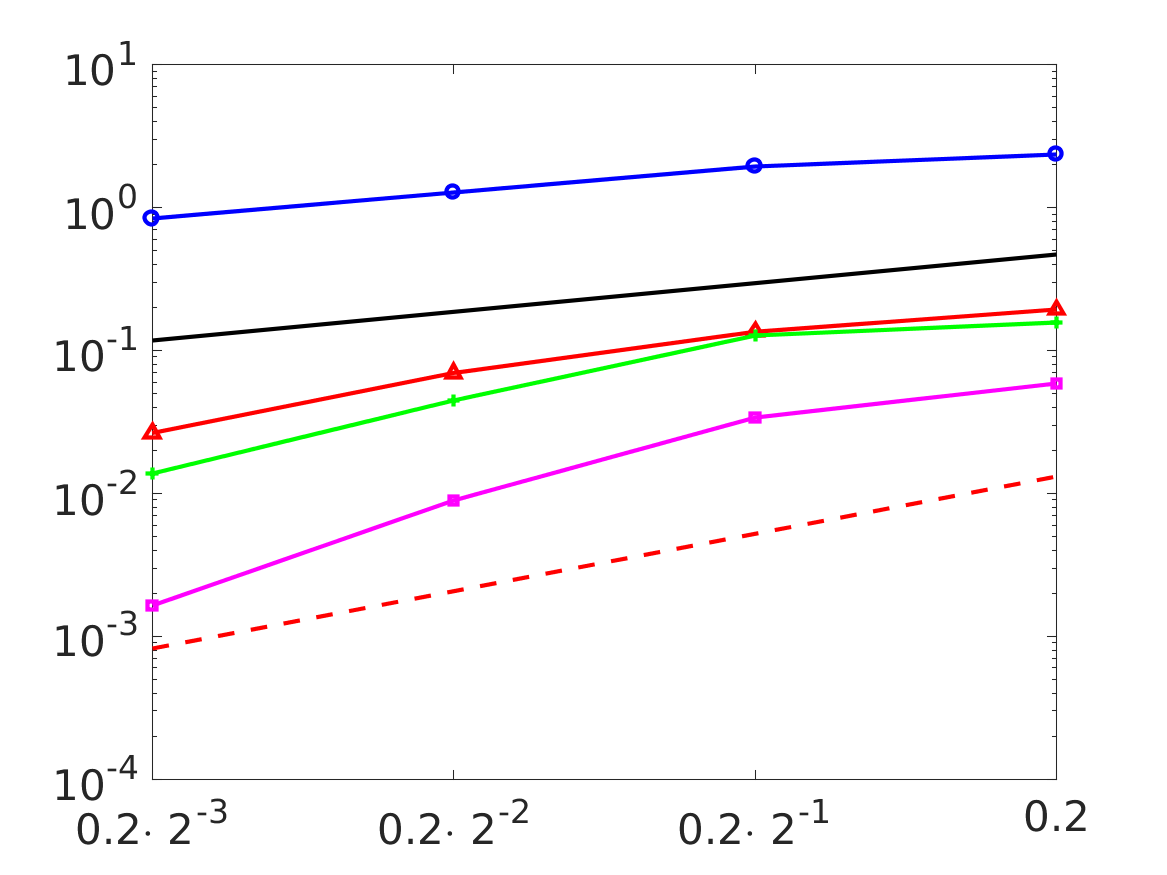}
%		\caption{ \bf $(h, \penl^{(2)}(h))$}
	\end{subfigure}\\
	\begin{subfigure}{0.8\textwidth}
		\includegraphics[width=\textwidth]{img/legend}
	\end{subfigure}
	\caption{Experiment~2:
	Errors $E_\vr, E_{\vu}, E_{\Grad \vu}$ and relative energy $\RE$ with respect to the pairs $(h,\penl(h)) = \left(0.2\cdot 2^{-m}, 2^{-(m+14)/2} \right)$ (left), $\ \left(0.2\cdot 2^{-m}, 4^{-(m+2)} \right)$ (middle) and $\left(0.2\cdot 2^{-m}, 16^{-m} \right)$ (right), $m = 0,1,2,3$.
	The black solid and red dashed lines without any marker denote the reference slope of $h$ and $h^2$, respectively.
	}
	\label{fig:ex2-3}
\end{figure}

\subsection{Experiment~3: Complex domain - discontinuous extension}
In the last experiment, we consider a more complicated geometry of the fluid domain, i.e.
\begin{equation*}
\Of = \hat{B}_{0.7}\setminus \Ov{B_{0.2}}, \quad \hat{B}_{0.7} := \left\{x ~\bigg|~ |x| < (0.7+\delta) + \delta \cos(8\phi), ~ \phi = \arctan \frac{x_1}{x_2} \right\}.
\end{equation*}
The initial data (including a discontinuous extension) are given by
\begin{equation*}
	(\vr,\vu)(0,x)
	\; = \; \begin{cases}
	(0.01, \, 0, \, 0 ) , & x \in B_{0.2}, \\
	\left(1,  \, \frac{ \big[1-\cos(8\pi (|x|-0.2)) \big]x_2}{|x|} , \,  -\frac{ \big[1-\cos(8\pi (|x|-0.2)) \big] x_1}{|x|}\right) , & x \in   {B}_{0.45}\setminus \Ov{B_{0.2}}, \\
	\left(1,  \, \frac{ \big[-1+\cos(8\pi (|x|-0.2)) \big]x_2}{|x|} , \,  -\frac{ \big[-1+\cos(8\pi (|x|-0.2)) \big] x_1}{|x|}\right) , & x \in   {B}_{0.7}\setminus B_{0.45}, \\
	(1, \, 0 , \, 0) , & x \in \hat{B}_{0.7}\setminus B_{0.7}, \\	
	(0.01, \, 0 , \, 0) , & x \in \mathbb{T}^2 \setminus \hat{B}_{0.7} .
	\end{cases}
\end{equation*}
In the simulation we set $\delta = 0.05$ and the final time to $T=0.1$.
The numerical solutions $\vrh$ and $\vuh$ at time $T=0.1$ for a fixed mesh size $h=0.2\cdot 2^{-4}$ and various penalty parameters $\penl = 4^{-m-2}, m=1,\dots,4,$ are presented in Figure \ref{fig:ex3}.
Figure~\ref{fig:ex3-1} shows the errors with respect to $h$ for a fixed $\penl$. The errors with respect to both parameters $(h,\penl(h))$ are displayed in  Figure~\ref{fig:ex3-3}.
Analogously as above, the numerical results indicate the first order convergence rate for the numerical solutions $\vr,\vu,\Grad \vu$ and the second order convergence rate for the relative energy $\RE$.

 Let us point out that the initial data (including the extension)  in Experiments~2 and 3 belong to the class $L^{\infty}(\tor)$,   which is consistent with Theorem~\ref{THM:ES}. 
%Nevertheless, the numerical results still indicate the strong convergence with first order rate.
Our results obtained in Theorem~\ref{THM:ES} requires $\penl \in (\mathcal{O}(h^3), \mathcal{O}(h))$.
However, the numerical results presented in Figures \ref{fig:ex1-3}, \ref{fig:ex2-3}, \ref{fig:ex3-3} 
indicate that the convergence rates  hold for a more general setting, e.g., $(h,\penl(h)) = (h, \mathcal{O}(h^{1/2})), (h, \mathcal{O}(h^4))$. 
%{\cgrey Extension of our theoretical error estimates to a more general setting as indicated in Experiments~2 and 3 is an interesting task for future work.}

\begin{figure}[htbp]
	\setlength{\abovecaptionskip}{0.cm}
	\setlength{\belowcaptionskip}{-0.cm}
	\centering
	\begin{subfigure}{0.48\textwidth}
		\includegraphics[width=\textwidth]{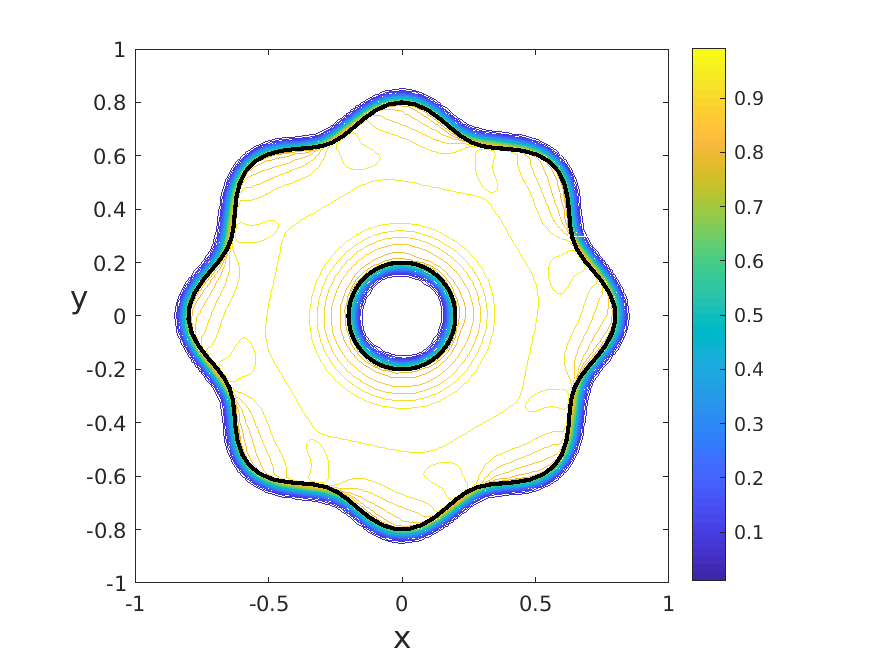}
	\end{subfigure}
	\begin{subfigure}{0.48\textwidth}
			\includegraphics[width=\textwidth]{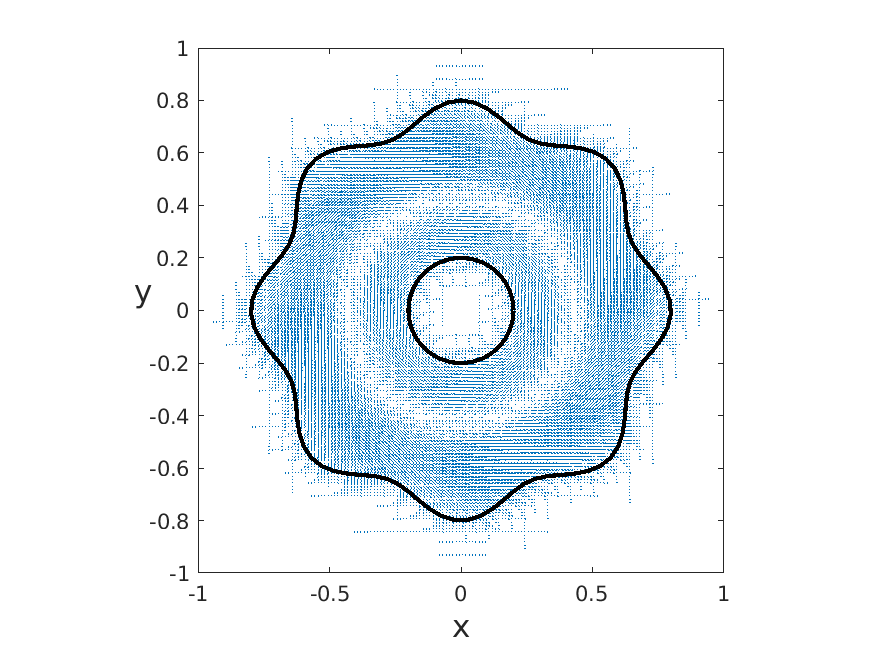}
	\end{subfigure}
\\  \vspace{-1em}
	\begin{subfigure}{0.48\textwidth}
		\includegraphics[width=\textwidth]{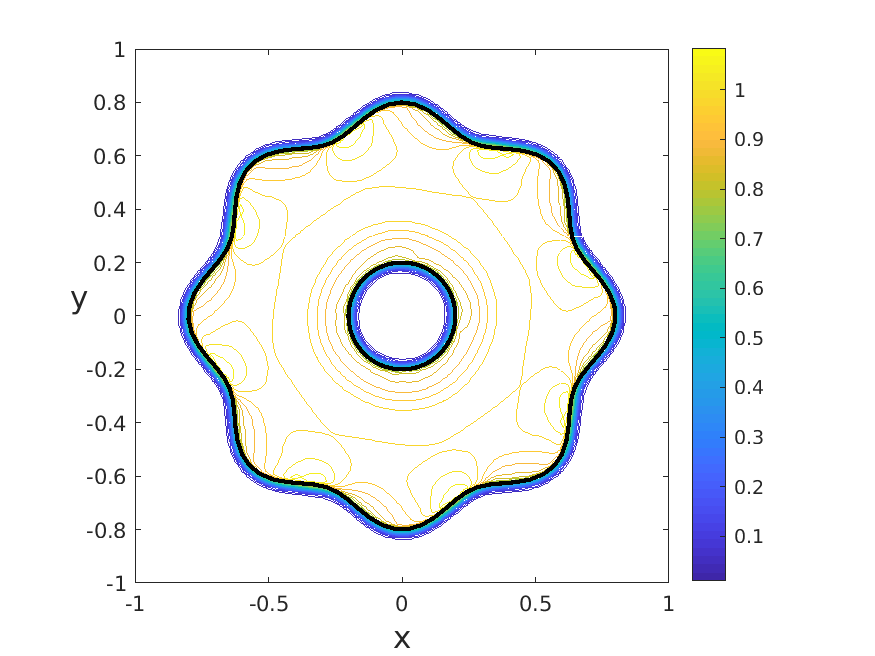}
	\end{subfigure}
	\begin{subfigure}{0.48\textwidth}
		\includegraphics[width=\textwidth]{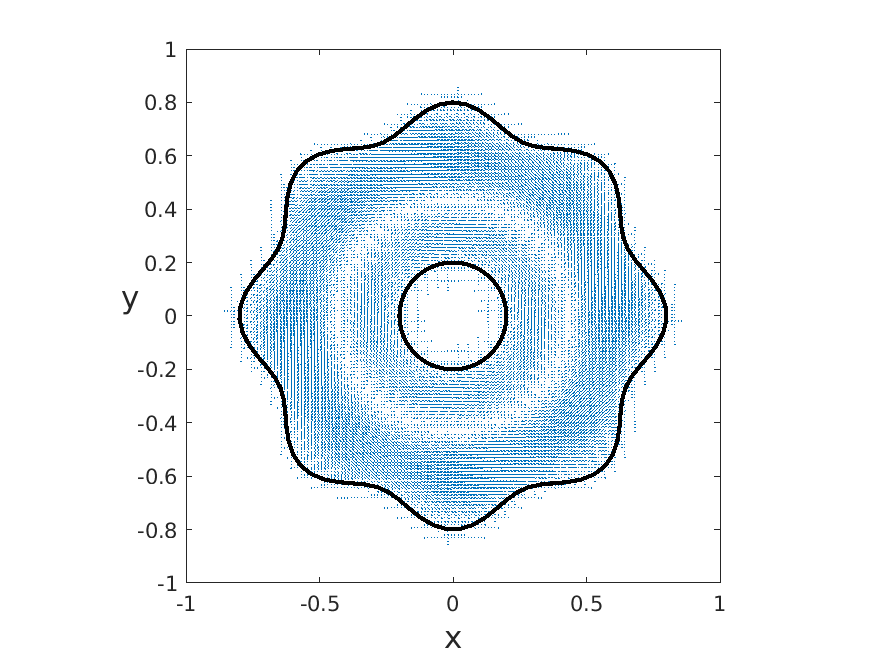}
	\end{subfigure}	
\\  \vspace{-1em}
	\begin{subfigure}{0.48\textwidth}
		\includegraphics[width=\textwidth]{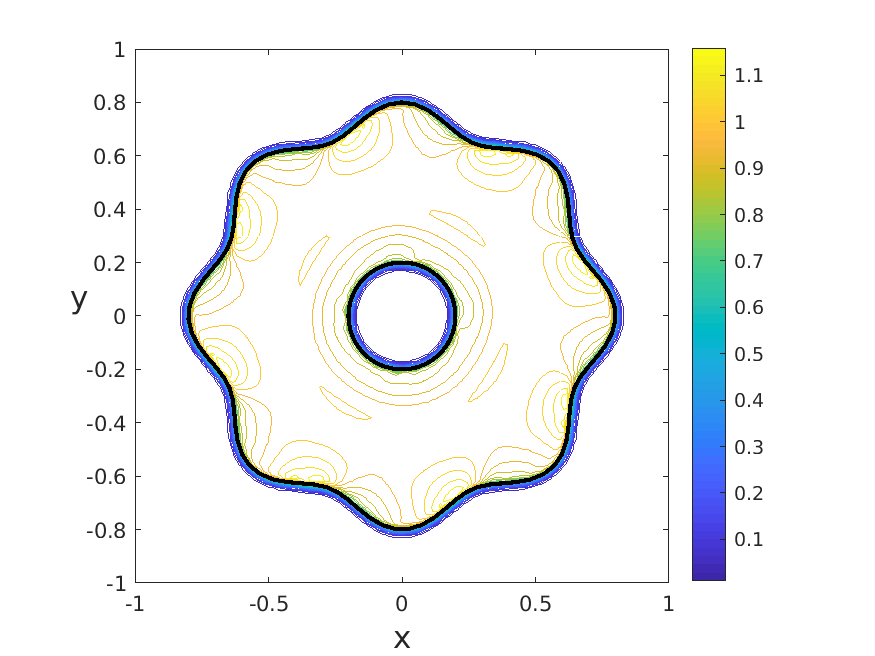}
	\end{subfigure}
	\begin{subfigure}{0.48\textwidth}
		\includegraphics[width=\textwidth]{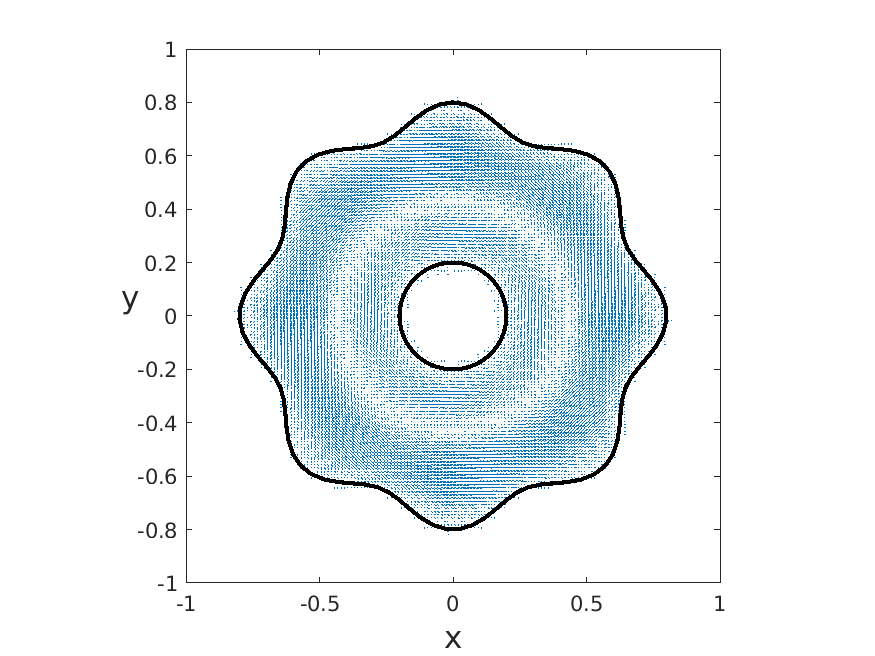}
	\end{subfigure}		
\\  \vspace{-1em}
	\begin{subfigure}{0.48\textwidth}
		\includegraphics[width=\textwidth]{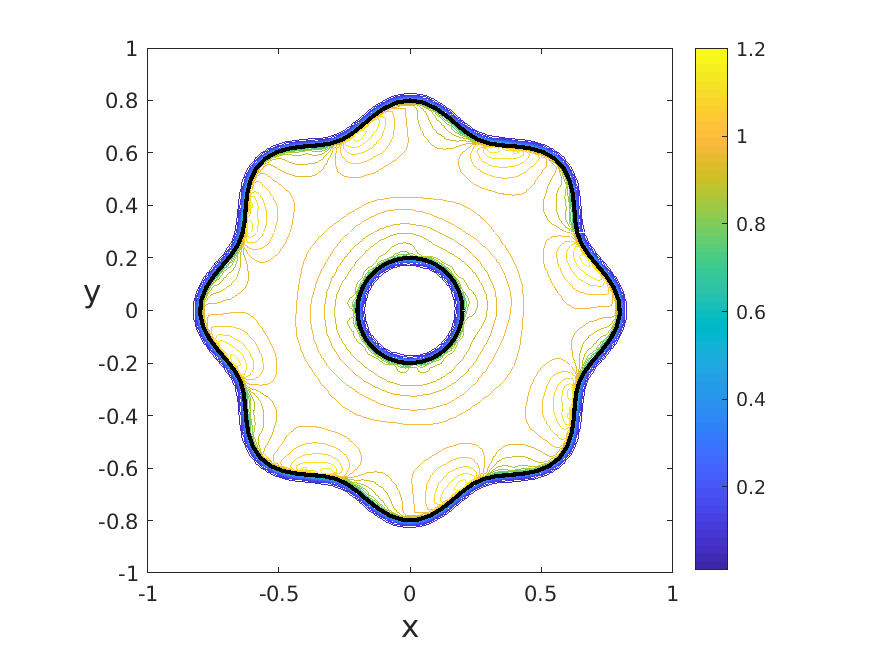}
	\end{subfigure}
	\begin{subfigure}{0.48\textwidth}
		\includegraphics[width=\textwidth]{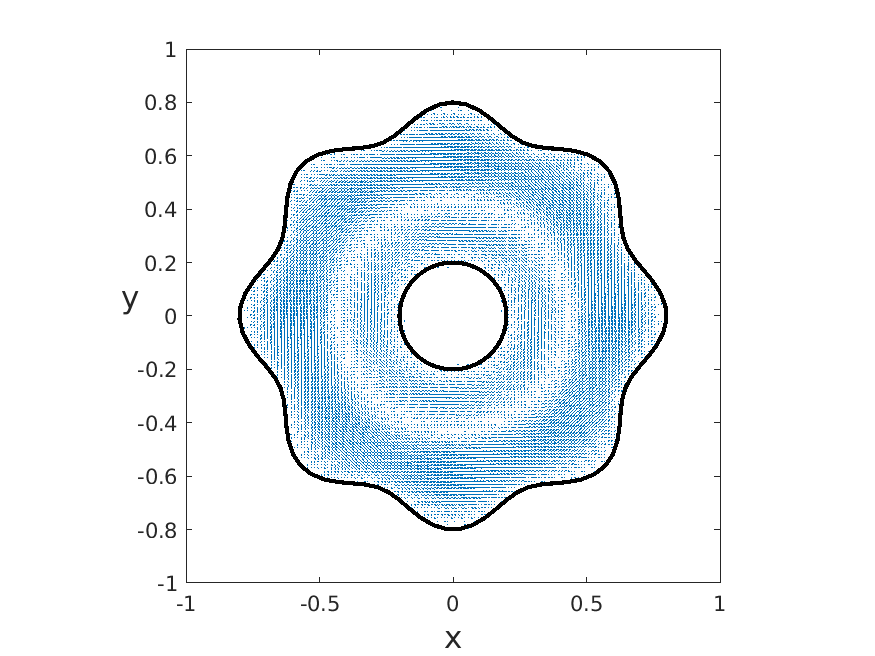}
	\end{subfigure}
	\caption{Experiment~3:
	Numerical solutions $\vrh$  (left) and $\vuh$ (right) obtained with $h = 0.2 \cdot 2^{-4}$ for different $\penl = 4^{-m-2}, m = 1, \dots, 4,$ from top to bottom.}
	\label{fig:ex3}
	 \vspace{-1em}
\end{figure}

\begin{figure}[htbp]
	\setlength{\abovecaptionskip}{0.cm}
	\setlength{\belowcaptionskip}{-0.cm}
	\centering
	 \vspace{-1em}
	\begin{subfigure}{0.48\textwidth}
		\includegraphics[width=\textwidth]{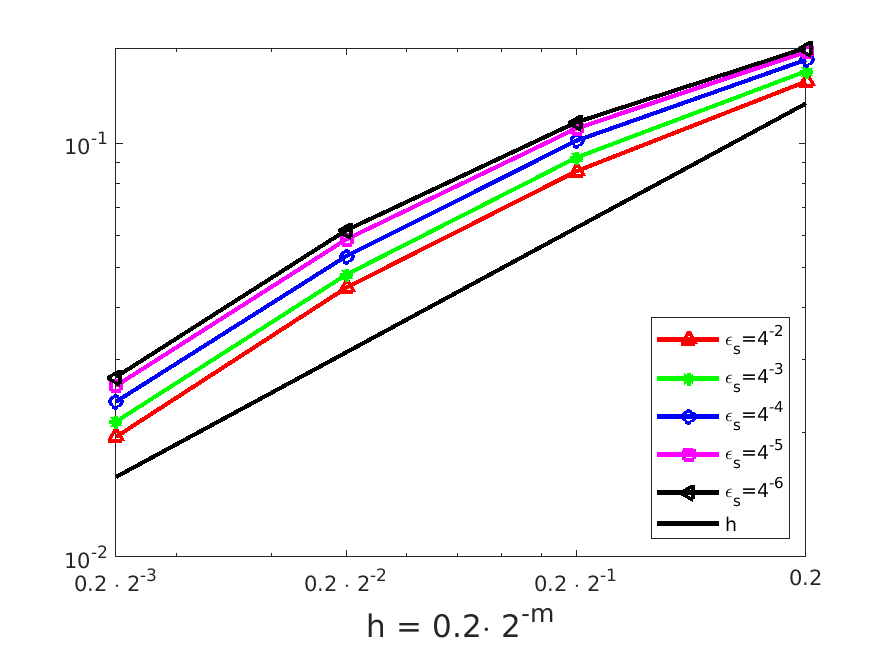}
		\caption{  $E_{\vr}^{\penl} $}
	\end{subfigure}
	\begin{subfigure}{0.48\textwidth}
		\includegraphics[width=\textwidth]{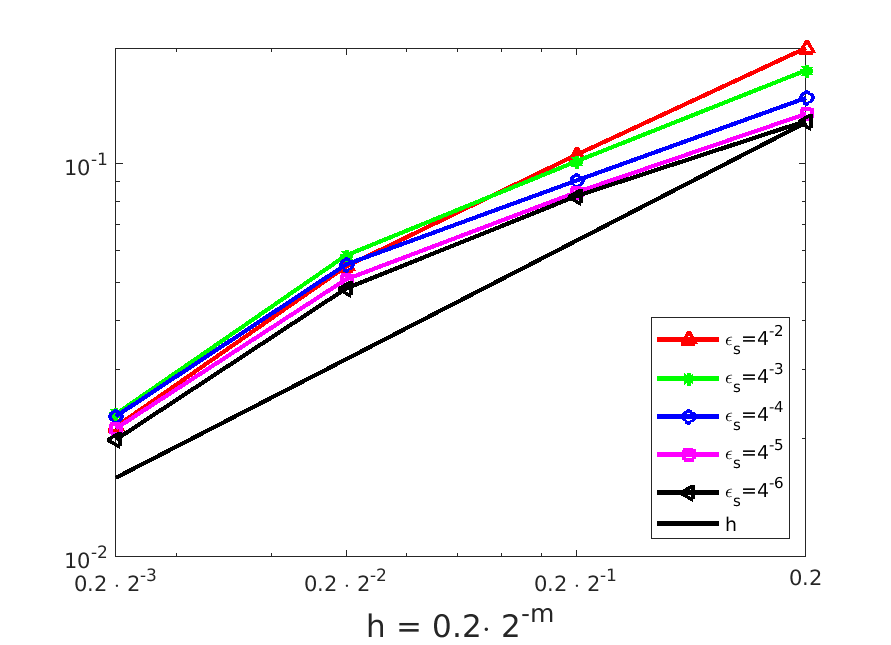}
		\caption{  $E_{\vu}^{\penl}  $}
	\end{subfigure}\\
	\begin{subfigure}{0.48\textwidth}
		\includegraphics[width=\textwidth]{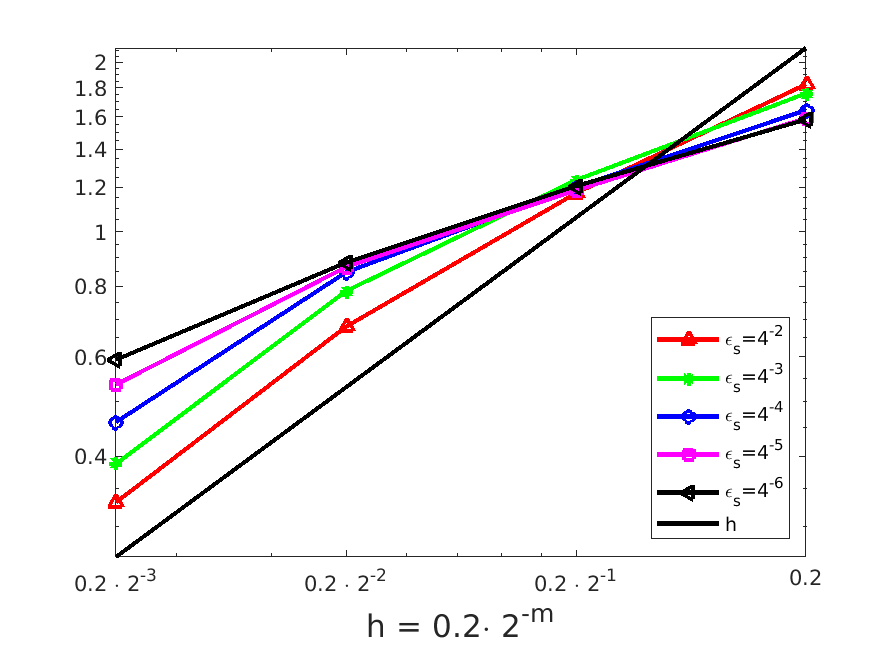}
		\caption{$E_{\Grad \vu}^{\penl}$ }
	\end{subfigure}
	\begin{subfigure}{0.48\textwidth}
		\includegraphics[width=\textwidth]{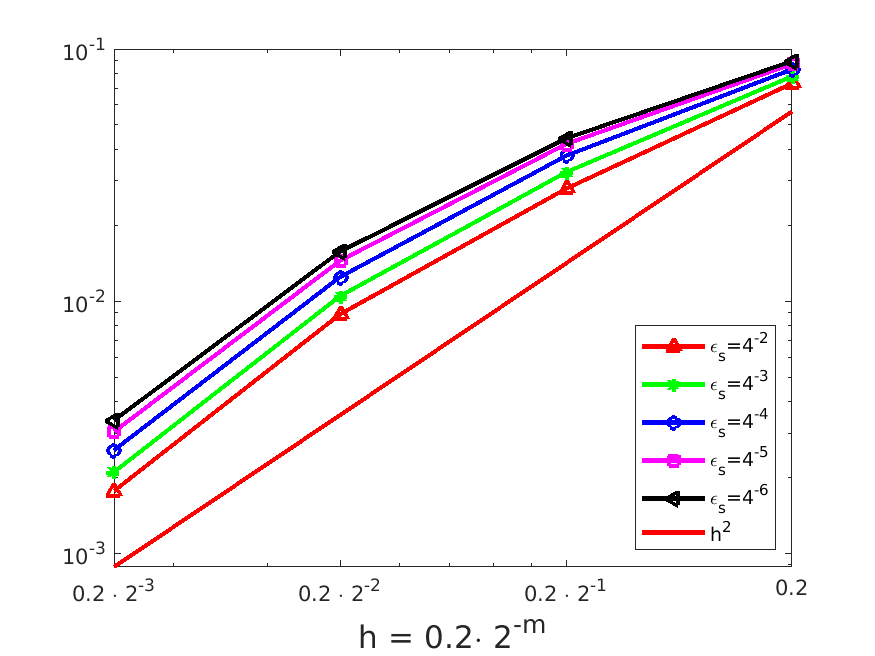}
		\caption{$R_E^{\penl}$ }
	\end{subfigure}
	\caption{Experiment~3:
	   The errors $E_\vr^{\penl}, E_{\vu}^{\penl}, E_{\Grad \vu}^{\penl}, R_E^{\penl}$ with respect to $h$ for different but fixed $\penl$. The black and red solid lines without any marker denote the reference slope of $h$ and $h^2$, respectively.
	}\label{fig:ex3-1}
	 \vspace{-1em}
\end{figure}

\begin{figure}[htbp]
	\setlength{\abovecaptionskip}{0.cm}
	\setlength{\belowcaptionskip}{-0.cm}
	\centering
	\begin{subfigure}{0.32\textwidth}
		\includegraphics[width=\textwidth]{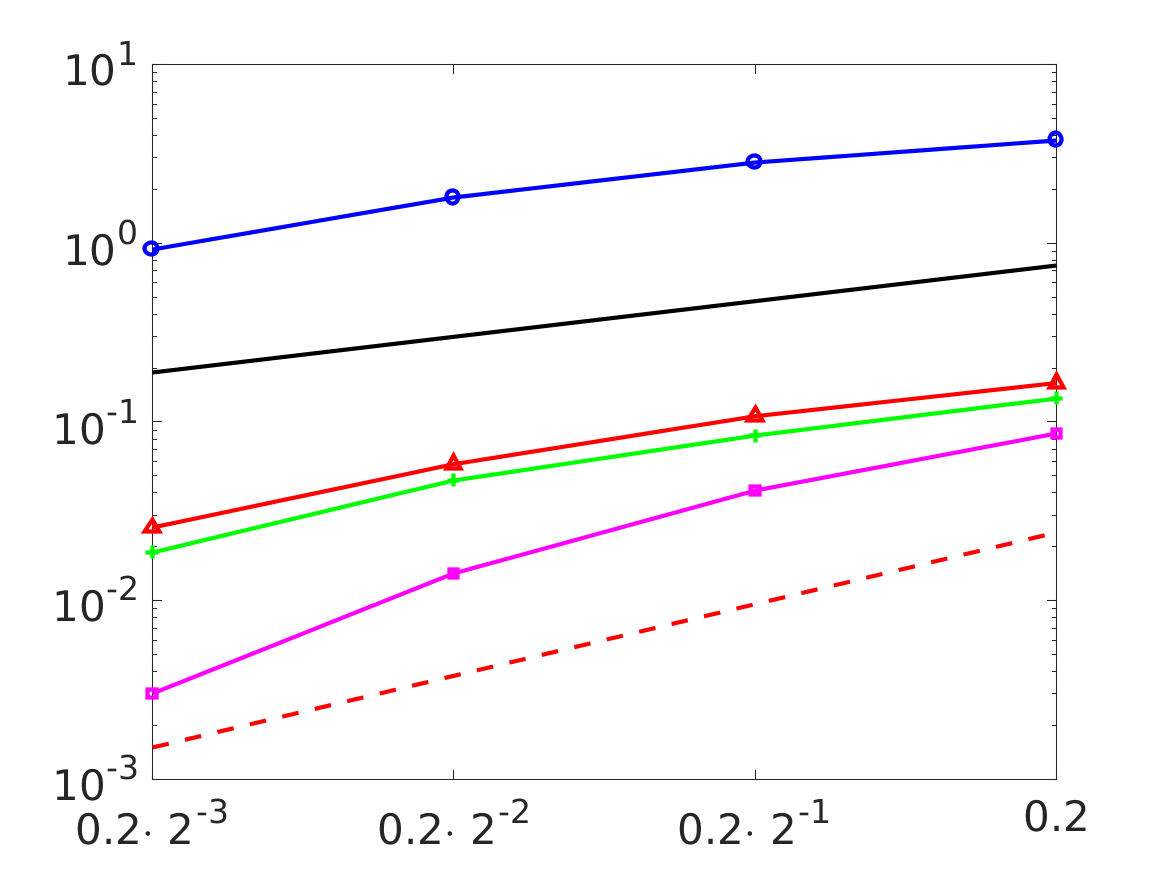}
%		\caption{ \bf $(h, \penl^{(1)}(h))$}
	\end{subfigure}
	\begin{subfigure}{0.32\textwidth}
		\includegraphics[width=\textwidth]{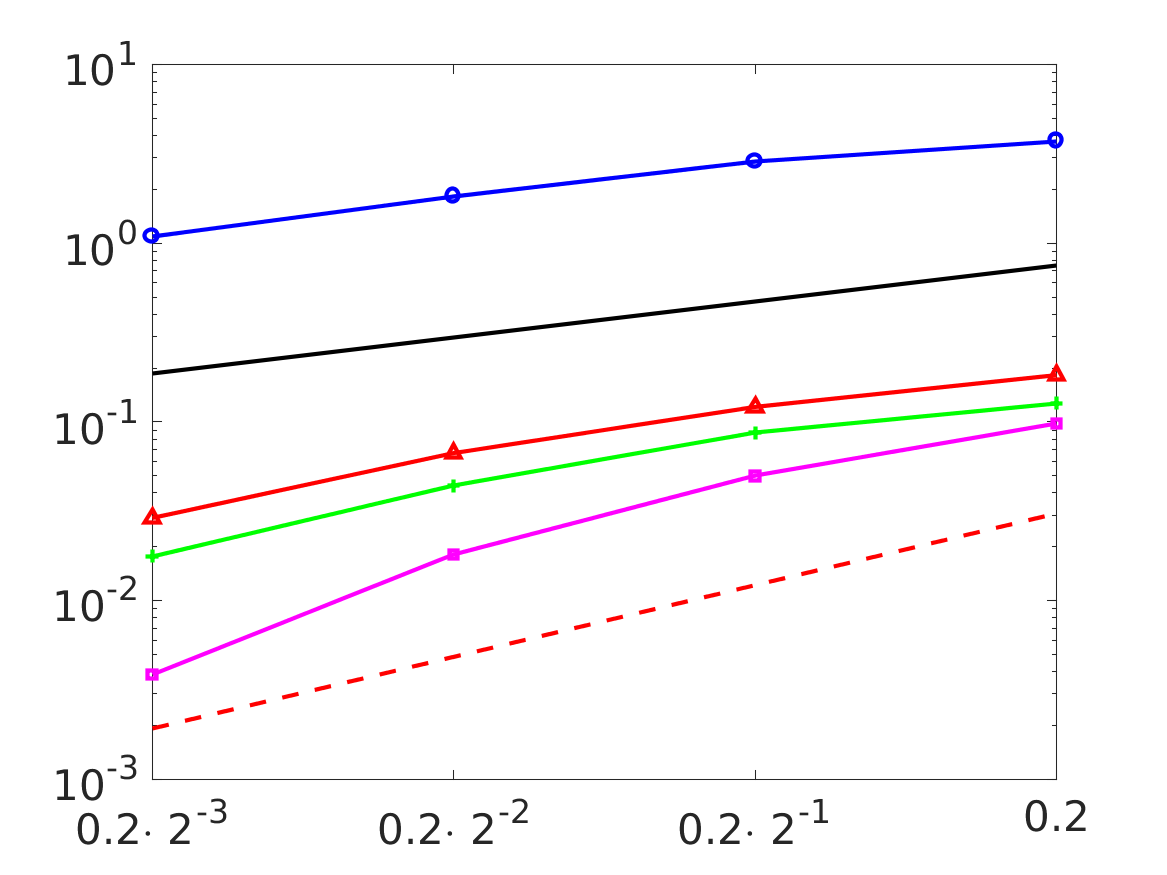}
%		\caption{ \bf $(h, \penl^{(2)}(h))$}
	\end{subfigure}
	\begin{subfigure}{0.32\textwidth}
		\includegraphics[width=\textwidth]{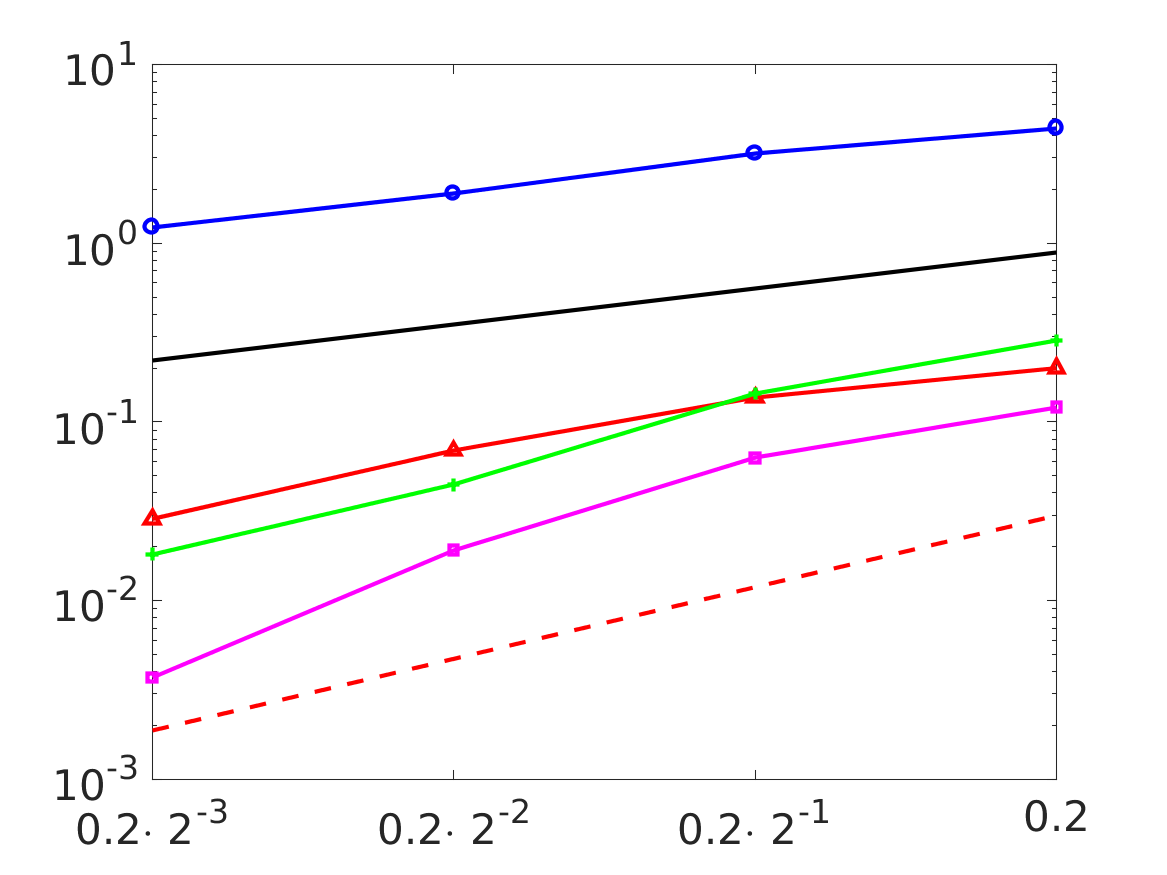}
%		\caption{ \bf $(h, \penl^{(2)}(h))$}
	\end{subfigure}\\
	\begin{subfigure}{0.8\textwidth}
		\includegraphics[width=\textwidth]{img/legend}
	\end{subfigure}
	\caption{Experiment~3:
	Errors $E_\vr, E_{\vu}, E_{\Grad \vu}$ and relative energy $\RE$ with respect to the pairs $(h,\penl(h)) = \left(0.2\cdot 2^{-m}, 2^{-(m+14)/2} \right)$ (left), $\ \left(0.2\cdot 2^{-m}, 4^{-(m+2)} \right)$ (middle) and $\left(0.2\cdot 2^{-m}, 16^{-m} \right)$ (right), $m = 0,1,2,3$.
	The black solid and red dashed lines without any marker denote the reference slope of $h$ and $h^2$, respectively.
	}
	\label{fig:ex3-3}
	 \vspace{-1em}
\end{figure}

%%%%%%%%%%%%%%%%%%%%%%%%%%%%%%%%%%%%%%%%%%%%%%

\appendix
\section{Preliminaries}
In this section we list three useful results that are important to derive our theoretical results.
First, we recall the generalized Sobolev-Poincar\'e inequality, see \cite[Theorem 17]{FeLMMiSh}.
\begin{Lemma}[{\cite[Theorem 17]{FeLMMiSh}}]\label{lmSP}
Let $\vrh > 0$ satisfy
\[
0<c_M \leq \intTd{\vrh} \ \mbox{ and } \ \intTd{\vrh^\gamma} \leq c_E,
\]
where $\gamma>1$, $c_M$ and $c_E$ are positive constants.
Then there exists $c=c(c_M, c_E, \gamma)$ independent of $h$ such that
\[
\norm{f_h}_{L^q(\tor)}^2  \leq c\left( \norm{ \Gradd f_h}_{L^2(\tor)}^2  + \intTd{\vrh |f_h|^2 } \right),
\]
where $q = 6$ if $d = 3$, and $q \in[1,\infty)$ if $d = 2$.
\end{Lemma}
Next we recall a slightly reformulated lemma from \cite[Lemma B.4]{FLS_IEE}. 
\begin{Lemma}[{\cite[Lemma B.4]{FLS_IEE}}] \label{difru}
Let $\gamma >1$ and $\vrh>0, \vrh \in L^\gamma(\tor),  \vuh \in L^2(\tor)$.  
Let  $ 0<\Un{\vr} < \tvr < \Ov{\vr},  \tvu \in L^2(\tor)$. 
Then for any constant $\delta \in(0,1)$ it holds that
\begin{align*}
\intTauOf{\abs{ (\vrh -\tvr) (\vuh -\tvu)}}    \aleq  \delta \norm{ \vuh - \tvu}_{L^2(\tor)}^2  + \RE(\vrh, \vuh| \tvr, \tvu).
\end{align*}
\end{Lemma}
% Next, we recall the following essential-residual splitting from \cite[Lemma 14.3]{FeLMMiSh}, which is useful in  the error analysis in Appendix \ref{App:errors}.
% \begin{Lemma}[\cite{FeLMMiSh}]\label{RL2}
% Let $\gamma>1$, $\vr \geq 0$,  $\underline{r} = \frac12 \min\limits_{(t,x) \in Q_T} \tvr >0 $ and $\overline{r} = 2 \max\limits_{(t,x) \in Q_T} \tvr$.  Then there exists $ C= C(\underline{r}, \overline{ r}) >0$ such that
% \begin{equation}\label{EDS}
% (\vr -\tvr)^2   \mathds{1}_{\rm ess}(\vr)  +  \vr  \mathds{1}_{\rm res}(\vr) \aleq
% (\vr -\tvr)^2   \mathds{1}_{\rm ess}(\vr)  + (1+ \vr^\gamma)  \mathds{1}_{\rm res}(\vr) \leq C
% \bbE(\vr \, |\, \tvr),
% \end{equation}
% where $\bbE(\vr \, |\, \tvr) $ is defined in \eqref{RE} and
% \begin{equation}\label{essres}
% ( \mathds{1}_{\rm ess}(\vr),  \mathds{1}_{\rm res}(\vr)) \colon = ( \mathds{1}_{[\underline{r}, \Ov{r}]}(\vr),  \mathds{1}_{\R^+\setminus[\underline{r}, \Ov{r}]}(\vr)) =
% \begin{cases}
% (1,0) & \mbox{ if } \vr \in [\underline{r}, \Ov{r}], \\
% (0,1) & \mbox{ if } \vr \in \R^+\setminus[\underline{r}, \Ov{r}].
% \end{cases}
% \end{equation}
% \end{Lemma}
Further, we recall \cite[Lemma C.1]{FLS_IEE} for the following estimates, which are used to derive $L^p$-error estimates, cf. Theorem \ref{THM:ES}.
\begin{Lemma}[{\cite[Lemma C.1]{FLS_IEE}}]\label{lmRE}
 Let $\gamma>1$, 
%\begin{align*}
$\tvr \in [\Un{\vr}, \Ov{\vr}], \; \abs{\tvu} \leq \Ov{u}, \; \Un{\vr},\, \Ov{\vr}, \, \Ov{u} > 0.$
%\end{align*}
Moreover, let $0<\vrh<\Ov{\vr}$ uniformly for $h\to0$. 
 Then the following estimates hold
 \begin{align*}
  \norm{\vrh-\tvr }^{2}_{L^{2}(\tor)}  %+   \norm{\vuh-\tvu }^{2}_{L^{2}(\tor;\R^d)}
+  \norm{\vrh \vuh - \tvr\tvu}^2_{L^{2}(\tor;\R^d)} 
 \aleq \RE(\vrh, \vuh | \tvr, \tvu)
 \end{align*}
 and 
  \begin{align*}
  \norm{\vrh-\tvr }^{\gamma}_{L^{\gamma}(\tor)}  %+   \norm{\vuh-\tvu }^{2}_{L^{2}(\tor;\R^d)}
 \aleq \RE(\vrh, \vuh | \tvr, \tvu) \quad \mbox{ if } \gamma>2.
 \end{align*}
%
%\begin{align}
%&  \label{lmRE1}
%\norm{\vrh-\tvr }_{L^{\gamma}(\tor)}  \aleq \left( \RE(\vrh, \vuh | \tvr, \tvu) \right)^{1/2}+\left(\RE(\vrh, \vuh | \tvr, \tvu)\right)^{1/\gamma} \hspace{1.2cm} \mbox{ if }  \gamma \leq 2,
%\\ & \label{lmRE2}
% \norm{\vrh-\tvr }^{\gamma}_{L^{\gamma}(\tor)}   +
% \norm{\vrh-\tvr }^{2}_{L^{2}(\tor)}
% \aleq \RE(\vrh, \vuh | \tvr, \tvu) \hspace{3cm} \mbox{ if }  \gamma > 2%,
%\end{align}
%and
%\begin{align}
%\label{lmRE3}
%  \norm{\vrh \vuh - \tvr\tvu}_{L^{\frac{2\gamma}{\gamma+1}}(\tor)}  \aleq \left( \RE(\vrh, \vuh | \tvr, \tvu) \right)^{1/2}+ \norm{\vrh-\tvr }_{L^{\min(\gamma,2)}(\tor)}.
%\end{align}
\end{Lemma}

\section{Proof of consistency}\label{sec_cs}
In this section our aim is to estimate the consistency errors $e_{\vr}$ and $e_{\vm}$  stated in Lemma \ref{thm_CS}. 
To begin, let us reformulate  $e_{\vr}$ and $e_{\vm}$.
\begin{Lemma}\label{thm_eeD}
Let  $(\vrh,\vuh)$ be a solution of the FV scheme \eqref{VFV} with $(\TS,h,\penl) \in (0,1)^3$ and $\viso>-1$. 
Then it holds 
\begin{align*}
& e_{\vr}(\phi) = E_t(\vrh, \phi) + E_F(\vrh, \phi),
 \\
& e_{\vm}(\bfphi) = E_t(\vmh, \bfphi) + E_F(\vmh, \bfphi) + E_{\Grad \vu}(\bfphi) - E_{p}(\bfphi) + E_{\penl}(\bfphi) 
\end{align*}
with 
\begin{equation}\label{eeD-1}
\begin{aligned}
 E_t(r_h, \phi) &=  \left[\intTd{ r_h \phi}\right]_{t=0}^{\tau}
-  \intn \intTd{ D_t r_h(t) \Pim \phi(t)  } \dt - \int_0^{\tau} \intTd{ r_h(t) \pd_t \phi (t) } \dt, 
\\
 E_F(r_h, \phi) &=  \intn  \intfaces{ F_h^{up}[ r_h, \vuh]  \jump{\Pim \phi} }\dt - \intTauTd{ r_h \vuh \cdot \Grad \phi }   
\\
& = - \sum_{i=1}^4 E_i(r_h,\phi) + \int_{\tau}^{t_{n+1}} \intTd{ r_h \vuh \cdot \Grad \phi }\dt,
\\
E_{\Grad \vu}(\bfphi) &= 
 \intTauTdB{ \mu \Gradd \vuh : (\Grad \bfphi -\Gradd( \Pim \bfphi)) + \nu \Divh \vuh ( \Div \bfphi  -  \Divh( \Pim \bfphi) )}
\\ 
& \quad - \int_{\tau}^{t_{n+1}} \intTdB{ \mu \Gradd \vuh : \Grad \bfphi + \nu \Divh \vuh  \Div \bfphi}\dt,
\\
E_{p}(\bfphi)  &=   \intTauTd{  p_h (\Div  \bfphi  -   \Divh(\Pim \bfphi) ) }  - \int_{\tau}^{t_{n+1}} \intTdB{p_h \Div  \bfphi }\dt,
\\ 
E_{\penl}(\bfphi) & = \frac{1}{\penl} \intTauOs{ \vuh \cdot \bfphi }  - \frac{1}{\penl} \intn \intOsh{ \vuh \cdot \Pim \bfphi } \dt
\\
&=  -\frac{1}{\penl} \int_\tau^{t_{n+1}} \int_{\Os} \vuh \cdot \bfphi  \dxdt  -  \frac{1}{\penl} \intn \int_{\Osh \setminus \Os} \vuh \cdot \bfphi  \dxdt,
\end{aligned}
\end{equation}
and
\begin{equation*}
\begin{aligned}
E_1(r_h,\phi) & = \frac12 \intn  \intfaces{ \abs{\avs{\vuh} \cdot \vc{n} } \jump{r_h} \jump{\Pim \phi} }\dt , \quad  E_2(r_h,\phi) = \frac14 \intn  \intfaces{ \jump{\vuh} \cdot \vc{n}  \jump{r_h} \jump{\Pim \phi} }\dt ,
\br
E_3(r_h,\phi) & = h^{\viso}  \intn  \intfaces{  \jump{r_h} \jump{\Pim \phi} }\dt, %= -h^{\viso+1}  \intn  \intTd{  r_h \Laph\Pim \phi }\dt, 
\quad 
E_4(r_h,\phi) = \intn \intTd{ r_h \vuh \cdot \left(\Grad \phi  - \Gradh (\Pim \phi)\right)} \dt.
% = \intn  \intfaces{  \jump{r_h \vuh}  \cdot \vc{n} (\phi - \avs{\Pim \phi}) }\dt,
\end{aligned}    
\end{equation*}
\end{Lemma}

\begin{proof} 
%Let $\tau \in [t_n, t_{n+1})$. 
With the definition of $e_{\vr}$, see \eqref{CS1-new}, we have
\begin{align*}
 e_\vr(\tau, \TS, h, \phi) =& \left[ \intTd{ \vrh \phi  } \right]_{t=0}^{t=\tau} -
\int_0^\tau \intTdB{   \vrh \partial_t \phi + \vrh   \vuh \cdot \Grad \phi } \dt
\br
=& \left[ \intTd{ \vrh \phi  } \right]_{t=0}^{t=\tau} -
\intn \intTd{ D_t r_h(t) \Pim \phi(t)  } \dt - \int_0^{\tau} \intTd{ r_h(t) \pd_t \phi (t) } \dt
\br
&+ \intn \intTd{ D_t r_h(t) \Pim \phi(t)  } \dt - \int_0^{\tau} \intTd{  \vrh   \vuh \cdot \Grad \phi } \dt
\br
=& E_t(\vrh,\phi)+ \intn  \intfaces{ F_h^{up}[ r_h, \vuh]  \jump{\Pim \phi} }\dt - \int_0^{\tau} \intTd{  \vrh   \vuh \cdot \Grad \phi } \dt 
\br
=& E_t(\vrh,\phi)+ E_F(\vrh,\phi) = E_t(\vrh,\phi)- \sum_{i=1}^4 E_i(r_h,\phi) + \int_{\tau}^{t_{n+1}} \intTd{ \vrh \vuh \cdot \Grad \phi }\dt,
\end{align*}
where we have used \cite[Lemma 2.5]{FLMS_FVNS} for the last equality. Analogous analysis applies to $e_{\vm}$ and completes the proof.
\end{proof}
The rest is to estimate $E_t, E_F, E_{\Grad \vu}, E_p, E_{\penl}$. For simplicity, hereafter we  shorten $L^p(0,T;L^q(\tor))$ as $L^pL^q$.

\subsection{Negative estimates of density and momentum}\label{ap_lms-1}

Firstly we introduce two negative density estimates. For completeness we present the proof of Lemma~\ref{RGS}.  Lemma~\ref{lem_ne}
states new improved results.

\begin{Lemma}[{\cite[Lemma 11.4]{FeLMMiSh}}]\label{RGS}
Let  $(\vrh,\vuh)$ be a solution of the FV scheme \eqref{VFV} with $(\TS,h,\penl) \in (0,1)^3$ and $\viso>-1$.
%If $d = 2$ then $p>2$, if $d = 3$ then $p = 6$.
Then  for $\gamma \in (1,2]$ it holds that
\begin{align}
\label{RGS1}
&
\norm{\vrh}_{L^\gamma L^{p\gamma/2}}^\gamma= \norm{\vrh^{\gamma/2}}_{L^2L^p}^2 =
\int_0^T \norm{\vrh^{\gamma/2}}_{L^{p}}^{2} \dt
\aleq  h^{-\viso-1},
\\&\label{RGS2}
\norm{\vrh}_{L^s L^{\infty}}^s =   \int_0^T \norm{\vrh}_{L^{\infty}}^s \dt \aleq h^{-s(2d+p(\viso+1))/(p\gamma) }
\end{align}
with $s\in(0,\gamma]$, $p=6$ if $d=3$ and $p\geq 1$ if $d=2$.
\end{Lemma}
\begin{proof}
First, it is easy to check the equivalence of the norms in \eqref{RGS1}, i.e.,
\begin{align*}
&\norm{\vrh}_{L^\gamma L^{p\gamma/2}}^\gamma
 = \int_0^T \left( \intTd{ \vrh^{p\gamma/2} } \right)^{ \frac{2}{p\gamma} \times \gamma} \dt
 = \int_0^T  \left(\intTd{ (\vrh^{\gamma/2})^p } \right)^{\frac1p \times 2} \dt
  =\int_0^T \norm{\vrh^{\gamma/2}}_{L^{p}}^{2} \dt
 = \norm{\vrh^{\gamma/2}}_{L^2L^p}^2.
\end{align*}
Further, %recalling the proof of \cite[Lemma 11.4]{FeLMMiSh},
we deduce from the estimate \eqref{ap3} that
\begin{equation*}
\norm{\Gradd \vrh^{\gamma/2}}_{L^2L^2}^2 \aleq  h^{-1}\int_0^T \intfaces{  \Hc''( \vr_{h,\dagger} ) \jump{ \vrh  }^2 } \dt \aleq  h^{-\viso-1} \quad \mbox{for} ~ \gamma \in (1,2].
\end{equation*}
Next, applying the Sobolev-Poincar\'{e} inequality, cf. Lemma \ref{lmSP} we have from the density estimate \eqref{ap1} that
\begin{align*}
  \int_0^T \norm{\vrh^{\gamma/2}}_{L^{p}}^{2} \ \dt
 \aleq \int_0^T \left(  \norm{\Gradd \vrh^{\gamma/2}}_{L^2}^2 + \norm{\vrh^{\gamma/2}}_{L^2}^2 \right)  \ \dt
=  \norm{\Gradd \vrh^{\gamma/2}}_{L^2L^2}^{2} + \norm{\vrh}_{L^{\gamma}L^\gamma} ^{\gamma}
\aleq  h^{-\viso-1},
\end{align*}
where $p>1$ in the case of $d=2$ and $p=6$ for the case of $d=3$.
Together with the inverse estimates we obtain
\begin{align*}
&  \int_0^T \norm{\vrh}_{L^{\infty}}^s \ \dt = \int_0^T \norm{\vrh^{\gamma/2}}_{L^{\infty}}^{2s/\gamma} \ \dt \aleq \int_0^T \left( h^{-d/p} \norm{\vrh^{\gamma/2}}_{L^p}\right) ^{2s/\gamma} \ \dt
=h^{-2sd/(p\gamma)} \int_0^T \left(  \norm{\vrh^{\gamma/2}}_{L^p}^2\right) ^{s/\gamma} \ \dt
\br
& \aleq h^{-2sd/(p\gamma)} \int_0^T \left(  \norm{\Gradd \vrh^{\gamma/2}}_{L^2}^{2s/\gamma}  + \norm{\vrh^{\gamma/2}}_{L^2}^{2s/\gamma}  \right) \ \dt
= h^{-2sd/(p\gamma)} \left(  \norm{\Gradd \vrh^{\gamma/2}}_{L^{2s/\gamma}L^2}^{2s/\gamma} + \norm{\vrh}_{L^sL^\gamma} ^s  \right)
\br
&\aleq h^{-2sd/(p\gamma)} \left(  \norm{\Gradd \vrh^{\gamma/2}}_{L^{2}L^2}^{2s/\gamma} + \norm{\vrh}_{L^{\infty}L^\gamma} ^s  \right)
\aleq h^{-2sd/(p\gamma)} \left(  (h^{-\viso-1})^{s/\gamma} + 1  \right)
\aleq h^{-s(2d+p(\viso+1))/(p\gamma) } ,
\end{align*}
which completes the proof.
\end{proof}

Further, we derive the following negative estimates of density and momentum.
\begin{Lemma}[Negative estimates of density and momentum] \label{lem_ne}
Let $(\vrh, \vuh)$ be a solution of the FV scheme \eqref{VFV}  with $(h,\penl) \in (0,1)^2,\, \viso > -1$ and $\gamma>1$.
Then the following hold:
\begin{equation}\label{beta-D}
\norm{\vrh}_{L^2((0,T)\times\tor)} \aleq h^{\beta_D},
\quad
\begin{aligned}
\left. \begin{array}{l}
\beta_D =
\begin{cases}
 \min\limits_{_{p \in \left[ 1, \infty \right)}}\left\{ \frac{p(\viso+1)+4}{2p}, 1 \right\} \cdot \frac{\gamma-2}{\gamma} \  & \mbox{if } d= 2, \gamma \in(1,2), \\
\min\left\{ \frac{\viso+2}{3} , 1 \right\} \cdot \frac{3(\gamma-2)}{2\gamma} & \mbox{if } d= 3, \gamma \in(1,2), \\
0 &\mbox{if } \gamma \geq 2,
\end{cases}
\end{array}
\right.
\end{aligned}
\end{equation}
\begin{equation}\label{beta-R}
\norm{\vrh}_{L^2(0,T;L^{6/5}(\tor))} \aleq h^{\betane},
\quad
\begin{aligned}
\left. \begin{array}{l}
\betane =
\begin{cases}
 \min\limits_{_{p \in \left[ \frac{12}{5\gamma}, \infty \right)}}\left\{   \frac{(1+\viso)p}{2(p-2)} , 1\right\} \cdot \frac{5\gamma-6}{3\gamma}  \  & \mbox{if } d = 2, \gamma \in \left(1, \frac65\right), \\
 \min\left\{   \frac{1+\viso}{2} , 1\right\} \cdot \frac{5\gamma-6}{2\gamma}  & \mbox{if } d = 3, \gamma \in \left(1, \frac65\right), \\
0 &\mbox{if } \gamma \geq \frac65,
\end{cases}
\end{array}
\right.
\end{aligned}
\end{equation}
\begin{equation}\label{beta-M}
\norm{\vrh \vuh}_{L^2((0,T)\times\tor)} \aleq h^{\beta_M},
\quad
\begin{aligned}
\beta_M =
\begin{cases}
\max\limits_{p \in \left[ \frac{2\gamma}{\gamma-1}, \infty \right)}\left\{ - \frac{p(\viso+1)+4}{2p\gamma}, \frac{p(\gamma-2)-2\gamma}{p\gamma}\right\} \
& \mbox{ if } d = 2, \gamma \leq 2, \\
0
& \mbox{ if } d = 2, \gamma > 2, \\
\max \left\{ - \frac{\viso+2}{2\gamma}, \frac{\gamma-3}{\gamma},-\frac{3}{2\gamma} \right\}
& \mbox{ if } d = 3, \gamma \leq 2, \\
 \frac{\gamma-3}{\gamma}
&  \mbox{ if } d = 3, \gamma \in (2,3),\\
0
& \mbox{ if } d = 3, \gamma \geq 3.
\end{cases}
\end{aligned}
\end{equation}
\end{Lemma}

\begin{proof}
We start with estimating the density in the $L^2L^2$-norm, i.e. \eqref{beta-D}.
For $\gamma \geq 2$ we easily check
\begin{equation*}
\norm{\vrh}_{L^2L^2} \aleq \norm{\vrh}_{L^\infty L^\gamma} \aleq 1 \quad \mbox{ meaning } \ \beta_D=0.
\end{equation*}
Now, let us focus on the case $\gamma < 2$.
On the one hand, thanks to the inverse estimate we have
\begin{align*}
\norm{\vrh}_{L^2L^2} \aleq h^{\frac{\gamma-2}{2\gamma}d} \norm{\vrh}_{L^{\infty}L^{\gamma}} \aleq h^{\frac{\gamma-2}{2\gamma}d}.
\end{align*}
On the other hand, in view of Lemma \ref{RGS} we have
\begin{align*}
\norm{\vrh}_{L^{2-\gamma}L^{\infty}}^{2-\gamma}
\aleq h^{-(2-\gamma)(2d+p(\viso+1))/(p\gamma) },
\end{align*}
which yields
\begin{align*}
\norm{ \vrh}_{L^2L^2}
& = \left( \int_0^T\intTd{ \vrh^{\gamma} \cdot  \vrh^{2-\gamma}} \dt \right)^{1/2}
\leq \left( \int_0^T \norm{\vrh}_{L^{\infty}(\tor)}^{2-\gamma} \intTd{ \vrh^{\gamma}} \dt \right)^{1/2}
\br
&= \left( \int_0^T \norm{\vrh}_{L^{\infty}(\tor)}^{2-\gamma}  \norm{\vrh}_{L^{\gamma}(\tor)}^{\gamma}  \dt  \right)^{1/2}
\leq \norm{\vrh}_{L^{\infty}L^\gamma} ^{\gamma/2} \norm{\vrh}_{L^{2-\gamma}L^{\infty}}^{(2-\gamma)/2} \aleq h^{\frac{p(\viso+1)+2d}{p}\cdot \frac{\gamma-2}{2\gamma}}.
\end{align*}
This completes the proof of \eqref{beta-D}.

Next, we prove \eqref{beta-R} for the estimate of the $L^2L^{6/5}$-norm of the density.
Considering $\gamma \geq 6/5$ it is obvious that
\[
\norm{\vrh}_{L^2L^{6/5}} \aleq \norm{\vrh}_{L^\infty L^\gamma} \aleq 1.
\]
For $\gamma < \frac65$ the proof can be done in the following two ways. In the first approach we apply the inverse estimates to get
\begin{align*}
\norm{\vrh}_{L^2L^{6/5}} \aleq h^{\frac{5\gamma-6}{6\gamma}d} \norm{\vrh}_{L^{\infty}L^{\gamma}} \aleq h^{\frac{5\gamma-6}{6\gamma}d}.
\end{align*}
In the second approach, recalling estimate \eqref{RGS1} and applying the interpolation inequality we obtain
\begin{align*}
\norm{\vrh}_{L^2L^{6/5}}
\aleq  \norm{\vrh}_{L^\infty L^\gamma}^q
\norm{\vrh}_{L^\gamma L^{p\gamma/2}} ^{1-q}
\aleq h^{-\frac{1+\viso}{\gamma}  \times (1-q) } \quad \mbox{for} ~ p > \frac{12}{5\gamma} > 2.
\end{align*}
Here $q$ satisfies
\begin{align*}
\frac12 \geq \frac{q}{\infty} + \frac{1-q}{\gamma} \mbox{ and }
\frac65 \geq \frac{q}{\gamma} + \frac{1-q}{\gamma p/2}
\Longleftrightarrow \frac{2-\gamma}{2} \leq q \leq  \frac{5\gamma p -12}{6p -12}.
\end{align*}
Hence, the optimal bound is achieved by choosing $q=\frac{5\gamma p - 12}{6(p-2)}$, i.e.
\begin{align*}
 \norm{ \vrh}_{L^2L^{6/5}}
 \aleq h^{-\frac{1+\viso}{\gamma}  \times (1-q) }
=  h^{- \frac{1+\viso}{\gamma} \cdot \frac{(6 - 5\gamma)p}{6(p-2)}}.
\end{align*}
Collecting the above estimates we obtain \eqref{beta-R}.

Finally, we are left with the estimate of $\norm{\vrh \vuh}_{L^2L^2}$. It is easy to check
\begin{equation}\label{em1}
 \begin{aligned}
   \norm{\vrh \vuh}_{L^2L^2} &= \norm{\vrh^2 \vuh^2}_{L^1L^1}^{1/2} \aleq  \left( \norm{\vrh \vuh^2}_{L^\infty L^{p\gamma/(p\gamma-2)}}   \norm{\vrh}_{L^\gamma L^{p\gamma/2}} \right)^{1/2}
 \\& \aleq \left(  h^{-\frac{2d}{p\gamma}}  h^{-\frac{\viso+1}{\gamma}}\right)^{1/2}
  =  h^{-\frac{p(\viso+1)+2d}{2 p\gamma}} \quad \quad \mbox{for }  \gamma \in (1,2]
\end{aligned}
\end{equation}
and
\begin{equation}\label{em2}
 \norm{\vrh \vuh}_{L^2L^2} \aleq h^{-\frac{d}{2\gamma}} \norm{\vrh \vuh}_{L^\infty L^{2\gamma/(\gamma+1)}} \aleq h^{-\frac{d}{2\gamma}}.
 \end{equation}
Moreover, by H\"older's inequality we have
\begin{equation*}
\norm{\vrh \vuh}_{L^2 L^{\gamma p/(\gamma+p)}} \aleq  \norm{\vrh}_{L^\infty L^{\gamma}}  \norm{\vuh}_{L^2 L^p}  \quad \mbox{for any }  p>1,
\end{equation*}
from which we obtain
\begin{align*}
\mbox{for } d = 2: \ &
 \mbox{if } \gamma > 2, \   \norm{\vrh \vuh}_{L^2L^2} \aleq  \norm{\vrh \vuh}_{L^2 L^{p\gamma /(\gamma+p)}}   \aleq 1  \quad \mbox{with} ~  p \geq 2+\frac{4}{\gamma-2},
 \br
 &\mbox{if } \gamma \leq 2, \  \norm{\vrh \vuh}_{L^2L^2} \aleq  h^{\left(\frac12 - \frac{(\gamma+p)}{p\gamma}\right)d} \norm{\vrh \vuh}_{L^2 L^{p\gamma /(\gamma+p)}} \aleq  h^{\frac{p(\gamma-2)-2\gamma}{p\gamma}}  ~ \mbox{for any }   p > 1,\br
\mbox{for } d = 3: \ &
 \mbox{if } \gamma \geq 3, \  \norm{\vrh \vuh}_{L^2L^2} \aleq  \norm{\vrh \vuh}_{L^2 L^{6\gamma /(\gamma+6)}} \aleq 1,
 \br
 &\mbox{if } \gamma < 3, \  \norm{\vrh \vuh}_{L^2L^2} \aleq  h^{\left(\frac12 - \frac{(\gamma+6)}{6\gamma}\right)d} \norm{\vrh \vuh}_{L^2 L^{6\gamma /(\gamma+6)}} \aleq  h^{\frac{(\gamma-3)}{3\gamma}d} = h^{\frac{(\gamma-3)}{\gamma}}.
 \end{align*}
Consequently, collecting \eqref{em1}, \eqref{em2} and the above estimates, we obtain
 \begin{align*}
   \mbox{for} \ d = 2: \ &
 \mbox{if} \ \gamma > 2, \   \beta_M = 0,
 \br
 &\mbox{if} \ \gamma \leq 2, \ \beta_M =  \max\limits_{p \in \left[ 1, \infty \right)}\left\{ - \frac{p(\viso+1)+4}{2p\gamma}, \frac{p(\gamma-2)-2\gamma}{p\gamma}, -\frac{1}{\gamma} \right\}
 \br
 & \hspace{2.1cm}= \max\limits_{p \in \left[ \frac{2\gamma}{\gamma-1}, \infty \right)}\left\{ - \frac{p(\viso+1)+4}{2p\gamma}, \frac{p(\gamma-2)-2\gamma}{p\gamma}\right\},
 \br
   \mbox{for} \ d = 3: \ &
 \mbox{if} \ \gamma \geq 3, \  \beta_M = 0,
 \br
 & \mbox{if} \ \gamma \in (2,3), \  \beta_M = \max \left\{  \frac{\gamma-3}{\gamma},-\frac{3}{2\gamma} \right\} =  \frac{\gamma-3}{\gamma},
 \br
  &\mbox{if} \ \gamma \leq 2, \  \beta_M = \max \left\{ - \frac{\viso+2}{2\gamma}, \frac{\gamma-3}{\gamma},-\frac{3}{2\gamma} \right\} ,
 \end{align*}
which concludes the proof.
\end{proof}

\subsection{Negative variational estimates}\label{ap_lms-2}
Having shown the negative estimates of density and momentum we can present some useful negative variational estimates that shall be used later for the consistency formulation. These proofs are analogous to Lemma 11.5 and Lemma 11.6 of \cite{FeLMMiSh}.

\begin{Lemma}\label{lem_dd}
Let $(\vrh, \vuh)$ be a solution of the FV method \eqref{VFV} with $(h,\penl) \in (0,1)^2$ and $\gamma>1$.
Then the following hold:
\begin{subequations}\label{es-ds}
\begin{align}
&\label{es-d0}
\int_0^{\tau}  \intfaces{ \frac{\jump{\vrh}^2}{ \max\{\rho_{h}^{\rm in}, \rho_{h}^{\rm out}\} } \cdot \left( h^{\viso} + \abs{ \avs{\vuh} \cdot \vc{n} }\right) } \dt \aleq h^{\beta_D},
\\
\label{es-d1}
&  \int_0^{\tau}  \intfaces{   \abs{ \jump{\vrh} }  } \dt \aleq h^{-(\viso +1)/2},
\\&  \label{es-d2}
\int_0^{\tau}  \intfaces{  \abs{ \jump{\vrh} }  \cdot  \abs{ \avs{\vuh} \cdot \vc{n} } } \dt \aleq h^{(\betacf - 1)/2},
\\ \label{es-d3}
&  \int_0^{\tau}  \intfaces{  (\abs{\jump{\vrh}}+ \avs{ \vrh })  \abs{ \jump{\vuh} \cdot \vn }  } \dt \aleq h^{\beta_D},
\\  \label{es-d4}
&  \int_0^{\tau}  \intfaces{   \abs{ \jump{\vrh \vuh} \cdot \vn }  } \dt \aleq h^{(\betacf - 1)/2} + h^{\beta_D},
\end{align}
where $\beta_D$ is given in \eqref{beta-D} and
\begin{align}\label{beta-tR}
\betacf =
\begin{cases}
0, & \mbox{if} \ d = 2,\\
 \betane, & \mbox{if} \ d = 3.
\end{cases}
\end{align}
\end{subequations}
\end{Lemma}

\begin{proof}

We start by showing \eqref{es-d0}. For any $\vr > 0$, taking $B(\vr)$ and $\phi_h$ in the renormalized continuity equation \cite[Lemma 8.2]{FeLMMiSh} as $\vr \ln \vr - \vr$ and $1$, respectively, we obtain
\begin{align}
& B'(\vr) = \ln(\vr) , \quad B''(\vr) = \frac1{\vr} > 0,
\br
& \intTd{ {\rm D}_t (\vrh  \ln \vrh) } + \intTd{  \vrh    \Divh  \vuh  }
\br
\leq  & - \sumKf \frac{|\sigma|}{|K|}  \left( h^\viso -  \left( \avs{\vuh}\cdot \vc{n} \right)^- \right)  \Big( \jump{B(\vrh )} - B'(\vrh )\jump{\vrh } \Big). \label{eq:dd}
\end{align}

Due to the convexity of $B(\vr)$, i.e.
\begin{align*}
\jump{B(\vrh)} - B'(\vrh)\jump{\vrh}  = \frac12 B''(\xi) \jump{\vrh}^2 
\geq \frac{\jump{\vrh}^2}{2 \max\{\rho_{h}^{\rm in}, \rho_{h}^{\rm out}\}}, \quad \xi \in \co{\rho_{h}^{\rm in}}{\rho_{h}^{\rm out}},
\end{align*}
we obtain
\begin{align*}
& h^\viso  \sumKf  \frac{|\sigma|}{|K|}    \Big( \jump{B(\vrh )} - B'(\vrh )\jump{\vrh } \Big) \geq  h^\viso \intfaces{ \frac{\jump{\vrh}^2}{ \max\{\rho_{h}^{\rm in}, \rho_{h}^{\rm out}\} } }.
\end{align*}
Moreover, we have
\begin{align*}
& -  \sumKf  \frac{|\sigma|}{|K|}    \left( \avs{\vuh}\cdot \vc{n} \right)^- \  \Big( \jump{B(\vrh )} - B'(\vrh )\jump{\vrh } \Big)
=  \sum_{\sigma\in\faces} |\sigma| \frac12  \bigg| \avs{\vuh}\cdot \vc{n}\bigg|  \ B''(\xi) \jump{\vrh}^2
\end{align*}
with $\xi \in \co{\rho_{h}^{\rm in}}{\rho_{h}^{\rm out}}$, which implies
\begin{align*}
&  -  \sumKf  \frac{|\sigma|}{|K|}  \left(  \left( \avs{u}\cdot \vc{n} \right)^- \right)  \Big( \jump{B(\vrh )} - B'(\vrh )\jump{\vrh } \Big) \geq  \intfaces{ \abs{ \avs{u}\cdot \vc{n} }  \frac{\jump{\vrh}^2  }{ 2\max\{\rho_{h}^{\rm in}, \rho_{h}^{\rm out}\} } }.
\end{align*}
Collecting the above estimates we derive from \eqref{eq:dd} that
\begin{align*}
& \int_0^{\tau}  \intfaces{ \frac{\jump{\vrh}^2}{ \max\{\rho_{h}^{\rm in}, \rho_{h}^{\rm out}\} }  \left( h^{\viso} + \frac{\abs{ \avs{\vuh} \cdot \vc{n} }}2 \right) } \dt
\br
\leq  &  \int_0^{\tau}   \sumKf  \frac{|\sigma|}{|K|}  \left( h^\viso -  \left( \avs{u}\cdot \vc{n} \right)^- \right)  \Big( \jump{B(\vrh )} - B'(\vrh )\jump{\vrh } \Big) \dt
\br
\leq & \intTd{ \vrh^0 \ln(\vrh^0) }  - \intTd{ \vrh  \ln(\vrh ) } - \int_0^\tau \intTdB{ \vrh \Divh \vuh }
\br
\aleq & 1 + \abs{  \int_0^\tau \intTdB{ \vrh \Divh \vuh } } \leq 1 +  \norm{\vrh}_{L^2L^2}  \norm{\Divh \vuh}_{L^2L^2} .
\end{align*}
Note that we have used here the inequality $\abs{\vr \ln(\vr)} \aleq 1 + \vr^\gamma$.
Consequently, applying Lemma \ref{lem_ne} concludes the proof of \eqref{es-d0}.

Secondly, thanks to H\"older's inequality and trace inequality, together with the density dissipation \eqref{es-d0} and uniform bounds \eqref{ap} we obtain \eqref{es-d1} in the following way:
\begin{equation}
\begin{aligned}
  \int_0^{\tau}  \intfaces{  \abs{ \jump{\vrh} }  } \dt
&\aleq  \left(  \int_0^{\tau}  \intfaces{ \frac{\jump{\vrh}^2}{ \max\{\rho_{h}^{\rm in}, \rho_{h}^{\rm out}\} }    } \dt\right)^{1/2} \left( \int_0^{\tau}  \intfaces{  \max\{\rho_{h}^{\rm in}, \rho_{h}^{\rm out}\}   } \dt\right)^{1/2}
\br
& \aleq  h^{-\viso/2}  h^{-1/2} = h^{-(1+\viso)/2} \quad \mbox{for} \  \gamma \geq 2,
\br
 \int_0^{\tau}  \intfaces{  \abs{ \jump{\vrh} }   } \dt
 & \aleq \int_0^{\tau}  \intfaces{  \abs{ \jump{\vrh} }   \ \sqrt{\Hc''( \vr_{h,\dagger} ) \; ( \vr_{h,\dagger} + 1)}  } \dt
\br
\aleq & \left(  \int_0^{\tau}  \intfaces{ \jump{\vrh}^2 \Hc''( \vr_{h,\dagger} )   } \dt\right)^{1/2} \left( \int_0^{\tau}  \intfaces{ ( \vr_{h,\dagger} + 1)   } \dt\right)^{1/2}
\br
\aleq & h^{-\viso/2}  h^{-1/2} = h^{-(1+\viso)/2} \quad \mbox{for} \  \gamma \in(1,2),
\end{aligned}
\end{equation}
where we have used the inequality
\[
\mathds{1}_{(0,\infty)}(\vr) \, \Hc''( \vr )  ( \vr+ 1) \ageq1 \mbox{ for } \gamma \in (1,2)
\]
and $\vr_{h,\dagger}$ is given in  \eqref{ap3}.

Thirdly, we can derive \eqref{es-d2} in an analogous way.
On the one hand, we have  that for $\gamma \geq 2$
\begin{align*}
&  \int_0^{\tau}  \intfaces{  \abs{ \jump{\vrh}  \avs{\vuh} \cdot \vc{n} } } \dt
\\& \aleq
\left(  \int_0^{\tau}  \intfaces{ \frac{\jump{\vrh}^2}{ \max\{\rho_{h}^{\rm in}, \rho_{h}^{\rm out}\} }  \abs{ \avs{\vuh} \cdot \vc{n} }  } \dt\right)^{1/2} \left( \int_0^{\tau}  \intfaces{  \max\{\rho_{h}^{\rm in}, \rho_{h}^{\rm out}\} \abs{ \avs{\vuh} \cdot \vc{n} }  } \dt\right)^{1/2}
\\& \aleq
 h^{-1/2} \norm{\vrh}_{L^2L^{p'}}^{1/2} \norm{\vuh}_{L^2L^p}^{1/2}
\aleq h^{-1/2} \norm{\vrh}_{L^2L^{p'}}^{1/2},
\end{align*}
where $p'= \frac{p}{p-1}$, for any $p>1$ in the case of $d=2$ and $p=6$ for the case of $d=3$.
On the other hand, for $\gamma < 2$ we have
\begin{align*}
&\int_0^{\tau}  \intfaces{  \abs{ \jump{\vrh} }  \cdot  \abs{ \avs{\vuh} \cdot \vc{n} } } \dt \aleq \int_0^{\tau}  \intfaces{  \abs{ \jump{\vrh} }  \cdot  \abs{ \avs{\vuh} \cdot \vc{n} } \cdot \sqrt{\Hc''( \vr_{h,\dagger} ) \cdot ( \vr_{h,\dagger} + 1)}  } \dt
\\& \aleq
\left(  \int_0^{\tau}  \intfaces{ \jump{\vrh}^2 \Hc''( \vr_{h,\dagger} ) \abs{ \avs{\vuh} \cdot \vc{n} }  } \dt\right)^{1/2} \left( \int_0^{\tau}  \intfaces{ ( \vr_{h,\dagger} + 1) \abs{ \avs{\vuh} \cdot \vc{n} }  } \dt\right)^{1/2}
\\& \aleq
h^{-1/2} (\norm{\vrh}_{L^2L^{p'}}^{1/2} +1)\norm{\vuh}_{L^2L^p}^{1/2}\aleq h^{-1/2} \norm{\vrh}_{L^2L^{p'}}^{1/2}.
\end{align*}
In view of the above two estimates, the proof of \eqref{es-d2} reduces to show
$\norm{\vrh}_{L^2L^{p'}} \aleq h^{\betacf}$.
If $d=2$, we take $p' \in (1,\gamma]$ (i.e. $p \geq \frac{\gamma}{\gamma-1}$) and obtain $\norm{\vrh}_{L^2L^{p'}} \aleq 1$.
If $d=3$, we choose $p=6$ and $p'=\frac{6}{5}$. Then we apply \eqref{beta-R} to get
\begin{align*}
\norm{\vrh}_{L^2L^{p'}}
=\norm{\vrh}_{L^2L^{6/5}}
\aleq h^{\betane},
\end{align*}
which completes the proof of \eqref{es-d2}.

The fourth estimate \eqref{es-d3} is straightforward:
\begin{align*}
& \int_0^{\tau}  \intfaces{ (\abs{\jump{\vrh}}+ \avs{ \vrh }) \abs{ \jump{\vuh} \cdot \vn} } \dt
\aleq
\left( \int_0^{\tau}  \intfaces{ h \avs{ \vrh^2 } } \dt \right)^{1/2}
\left( \int_0^{\tau}  \intfaces{ \frac{ \abs{ \jump{\vuh} }^2 }{h}} \dt \right)^{1/2}
 \\& \aleq \norm{\vrh}_{L^2L^2}\norm{\Gradd \vuh}_{L^2L^2} \aleq h^{\beta_D}.
\end{align*}

Finally, recalling the product rule for the equality
\[\jump{\vrh \vuh} = \jump{\vrh} \avs{\vuh} +\avs{\vrh} \jump{\vuh}
\]
we may employ \eqref{es-d1} and \eqref{es-d3} to derive \eqref{es-d4}
\begin{align*}
& \int_0^{\tau}  \intfaces{  \abs{ \jump{\vrh\vuh} \cdot \vn }  } \dt
\leq
 \int_0^{\tau}  \intfaces{  \abs{ \jump{\vrh} \avs{\vuh}\cdot \vn }  } \dt
+ \int_0^{\tau}  \intfaces{   \avs{\vrh} \abs{\jump{\vuh}\cdot \vn }  } \dt
\\& \aleq h^{(\betacf-1)/2} + h^{\beta_D},
\end{align*}
which completes the proof.
\end{proof}

\subsection{Consistency proof}\label{sec_ce}
Equipped with Lemmas \ref{thm_eeD}, \ref{lem_ne} and \ref{lem_dd}, we are now ready to prove the consistency formulation, which is similar to the one in \cite[Section 11.3]{FeLMMiSh}, \cite[Theorem 4.1]{FLMS_FVNS} or \cite[Section 2.7]{FLS_IEE}. %Hence, in the following we only give the idea and framework of the proof.

\begin{proof}[Proof of Lemma \ref{thm_CS}]

Let $\tau \in [t_n, t_{n+1})$.

\paragraph{Step 1 -- time derivative terms:}
Let $r_h$ stand for $\vrh$ or $\vrh \vuh$.
Recalling \cite[equation(2.17)]{FLS_IEE} we have
\begin{align}\label{CST}
E_t(r_h,\phi)& \aleq  \left(\norm{   \pd_{t}^2\phi  }_{L^\infty L^\infty}  + \norm{ \pd_{t}\phi   }_{L^\infty L^\infty} \right) \TS
\aleq  \TS.
\end{align}
\paragraph{Step 2 -- convective terms:}
Applying H\"older's inequality with  \eqref{proj} and \eqref{es-ds} we obtain for $r_h = \vrh$ that
\begin{equation*}
\sum_{i = 1}^4\abs{ E_i(\vrh, \phi)} \aleq  h^{(1+\betacf)/2} + h^{(1+\viso)/2} + h^{1+\beta_D} \quad \quad \mbox{for} \ \phi \in L^\infty (0,T; W^{1,\infty}(\tor)).
\end{equation*}
Directly following \cite[Theorem 4.1]{FLMS_FVNS} we obtain for $r_h = \vrh \vuh$  that
\begin{equation*}
\sum_{i = 1}^4 \abs{  E_i(\vrh \vuh, \bfvarphi)} \aleq  h +  h^{1+\viso} + h^{1+\beta_M} \quad \quad \mbox{for} \ \bfvarphi \in L^\infty(0,T; W^{2,\infty}(\tor; \R^d)).
\end{equation*}
Moreover, it is obvious that
\begin{equation*}
 \Bigabs{\int_{\tau}^{t_{n+1}} \intTd{ r_h \cdot \Grad \phi} \dt } \aleq  \TS \|\Grad\phi\|_{L^\infty L^\infty} \|r_h\|_{L^1 L^1} \aleq  \TS, \quad  r_h = \vrh \ \mbox{or} \ \vrh \vuh.
\end{equation*}
Consequently, we obtain
\begin{equation}\label{Eis_1}
% \Bigabs{ \int_0^{\tau} \intTd{ r_h \cdot \Grad \phi} \dt  -\intn  \intfaces{ F_h^{up}[ r_h, \vuh]  \jump{\Pim \phi} }\dt} 
E_F(\vrh,\phi)\aleq  h^{(1+\betacf)/2} + h^{(1+\viso)/2} + h^{1+\beta_D}+ \TS; \quad 
E_F(\vmh,\bfvarphi) \aleq h +  h^{1+\viso} + h^{1+\beta_M} + \TS.
\end{equation}

 \paragraph{Step 3 -- viscosity terms:}
Applying the interpolation error estimates we have 
 \begin{align}\label{e_vis_1}
 \abs{E_{\Grad \vu}(\bfvarphi)}
 & \aleq  h \left( \norm{\Gradd \vuh }_{L^2 L^2} + \norm{\Divh \vuh }_{L^2 L^2} \right)  \norm{\Grad^2\bfvarphi}_{L^\infty L^\infty}
 \br
& +   (\TS)^{1/2}  \left( \norm{\Gradd \vuh }_{L^2 L^2} + \norm{\Divh \vuh }_{L^2 L^2} \right)  \|\Grad\bfvarphi\|_{L^\infty L^\infty}.
\end{align}

 \paragraph{Step 4 -- pressure term:}
Applying the interpolation error estimates we have 
\begin{align}\label{e_p_1}
 \abs{E_{p}(\bfvarphi)}
 \aleq  h \norm{p_h}_{L^\infty L^1}  \norm{\Grad^2\bfvarphi}_{L^\infty L^\infty}
 +   \TS  \norm{p_h}_{L^\infty L^1}   \|\Grad\bfvarphi\|_{L^\infty L^\infty}.
\end{align}

 \paragraph{Step 5 -- penalization term:}
With
$$\TS \norm{\vuh(t_{n})}_{L^2(\Os)}^2 \leq \TS \sum_{i=0}^{ N_T-1}\norm{\vuh(t_{i})}_{L^2(\Os)}^2 =  \norm{\vuh}_{L^2((0,T)\times \Os)}^2 \aleq \penl $$
we have 
\begin{align*}
\Bigabs{\frac{1}{\penl} \int_\tau^{t_{n+1}} \int_{\Os} \vuh \cdot \bfvarphi  \dxdt}  \aleq \frac{\norm{\bfvarphi}_{L^\infty L^\infty}  }{\penl} \TS\norm{\vuh(t_{n})}_{L^1(\Os)} \aleq \frac{\TS}{\penl} \norm{\vuh(t_{n})}_{L^2(\Os)} \aleq (\TS / \penl)^{1/2},
\end{align*}
resulting the following estimate of $E_{\penl}$
\begin{align}\label{e_pen_1}
\abs{E_{\penl}(\bfvarphi)}
 & \aleq  (\TS / \penl)^{1/2} + \frac{\norm{\vuh}_{L^2((0,T)\times\Osh)}}{\penl}  \norm{\bfvarphi}_{L^{\infty}L^{\infty}}  \left( |\Osh \setminus \Os|\right)^{1/2}
 \aleq  (\TS / \penl)^{1/2} +(h/\penl)^{1/2}.
\end{align}

In summary, combining \eqref{CST} -- \eqref{e_pen_1} we have
\begin{align*}
& \abs{e_{\vr}} \aleq \TS  + h^{(1+\viso)/2} + h^{(1+\betacf )/2} +  h^{1+\beta_D},
\br
& \abs{e_{\vm}}  \aleq  (\TS)^{1/2} + h + h^{1+\viso}  + h^{1+\beta_M}   + (\TS / \penl)^{1/2} + (h/\penl)^{1/2} ,
\end{align*}
which concludes the proof.
\end{proof}

\section{Some useful estimates}%\label{sec_pee}

In this section we derive some useful inequalities that help us to obtain the optimal error estimates in Theorem~\ref{THM:ES}. 
To begin, let us recall the regularity of the extended strong solution $(\tvr,\tvu)$, cf. \eqref{EXT1}:
\begin{align*}
&(\tvr,\tvu)|_{\Os} \in C^2(\Os;\R^{d+1}), \quad (\tvr,\tvu)|_{\Of} \in C^1((0,T)\times \Of;\R^{d+1}) \cap C(0,T;W^{4,\infty}(\Of;\R^{d+1})),\\
&\tvr \in L^{\infty}((0,T)\times \tor), \quad \tvu \in W^{1,\infty}((0,T)\times \tor;\R^{d}).
\end{align*}
Note that $(\tvr,\tvu)$ is piecewise smooth with possible discontinuities of $\tvr$ and $\Grad \tvu$ on the boundary.  
To this end, we introduce an artificial splitting of the domain and derive suitable estimates related to the boundary part.
\begin{Definition}[Splitting of the mesh]
\label{def_ES2}
We split the mesh $\grid$ into three pieces, see Figure \ref{fig:domain}. $\Oc_h$ denotes the area containing the neighbourhood of the fluid boundary $\pd \Of$
\begin{align*}
\Oc_h = \left\{  K \mid   \cup_{L\cap K\neq \emptyset} L  \cap \partial \Of \neq  \emptyset \right\}.
\end{align*}
The inner domain $\OI_h$ and the outer domain $\OO_h$ are given as
\[\OI_h \colon = \Of  \setminus  \Oc_h \quad \mbox{ and }\quad \OO_h \colon =\Os \setminus  \Oc_h .\]
\begin{figure}[!h]
	\centering
	\includegraphics[width=0.7\textwidth]{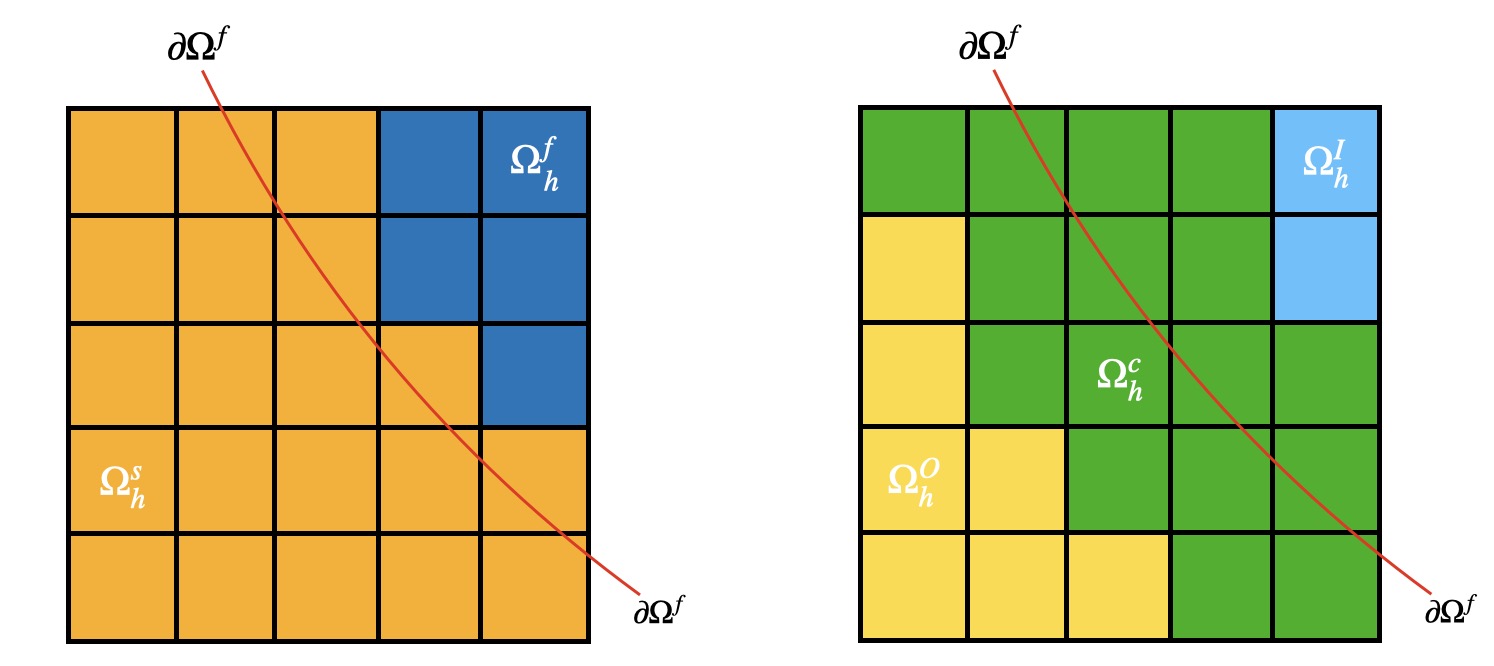}
	\caption{\small{ Zoom-in of domain splitting. }}\label{fig:domain}
\end{figure}
\end{Definition}
By the above definition, for any cell $K \in \Oc_h$,  we know that either $K$ or one of its neighbours intersects with the fluid boundary $\partial\Of$. Moreover, we have
\begin{subequations}\label{EXTE}
%\begin{equation}\label{EXTE0}
% \OI_h \subset \Ofh \subset \Of,\quad
%\OO_h \subset \Os \subset \Os_h , \quad
%\end{equation}
\begin{equation}\label{EXTE1}
\abs{\Oc_h} \aleq h, \quad \abs{\tvu} \aleq  \norm{\tvu}_{W^{1,\infty}(\tor)} h \quad \mbox{if} \ x \in \Oc_h,
\end{equation}
\begin{equation}\label{EXTE2}
 \abs{ \Gradh( \Pim \tvu)} + \abs{\Divh(\Pim \tvu)} + \abs{ \Gradd( \Pim \tvu)}  \aleq \begin{cases}
\norm{\tvu}_{W^{1,\infty}(\tor)} \  & \mbox{if} \ x \in  \OI_h \cup \Oc_h, \\
 0 & \mbox{if} \ x \in  \OO_h,
\end{cases}
\end{equation}
\begin{equation}\label{EXTE4}
\abs{ \Laph \Pim \tvu }  \aleq \begin{cases}
 \norm{\tvu}_{W^{2,\infty}(\OI_h)}  & \mbox{if} \ x \in  \OI_h, \\
\norm{\tvu}_{W^{1,\infty}(\tor)} h^{-1} \ & \mbox{if} \ x \in  \Oc_h, \\
 0 & \mbox{if} \ x \in  \OO_h.
\end{cases}
\end{equation}
\end{subequations}

\begin{Lemma}\label{lm_m}
Let $\delta \in (0,1)$ be an arbitrary constant,
$(\vrh, \vuh)$ be a solution of the FV method \eqref{VFV} with $(h,\penl) \in (0,1)^2.$
%Let $(\vr,\vu)$ be the strong solution in the sense of Definition~\ref{def_ST},
%and $(\tvr,\tvu)$ be given by Definition \ref{def_ES}.
Let  $(\tvr, \tvu)$ be the extended strong solution in the sense of Definition~\ref{def_ES}.  
Then
\begin{subequations}\label{lm_ms}
\begin{equation}\label{lm_m3}
\intTauOch{ \vrh  \abs{\vuh}^2    }
\aleq h^3+ \int_0^\tau \RE(\vrh, \vuh| \tvr, \tvu) \dt,
\end{equation}
\begin{equation}\label{lm_u0}
\int_0^\tau \int_{\Of \setminus \Ofh} {\abs{\vuh}}\dxdt 
\aleq  \penl+ \delta h  \frac{\norm{\vuh}^2_{L^2((0,\tau)\times\Osh)}}{\penl},  
\end{equation}
\begin{equation}\label{lm_u1}
\intTauOs{\abs{\vuh}} \leq \intTauOsh{ \abs{ \vuh}  }
\aleq  
\penl+ \delta \frac{\norm{\vuh}_{L^2((0,\tau)\times\Osh)}^2}{\penl},
\end{equation}

\begin{equation}\label{lm_gradu4}
   \intTauOch{\abs{ \Gradd \vuh} }
\aleq   
h+ \delta \mu \norm{ \Gradd  \vuh - \Grad \tvu}_{L^2((0,\tau)\times\tor)}^2, %\intTauOch{ |\Gradd \vuh - \Grad \tvu|^2 }, 
\end{equation}
\begin{equation}\label{lm_div4}
   \intTauOch{\abs{ \Divh \vuh} } \aleq 
h + \delta  \nu \norm{ \Divh  \vuh - \Div \tvu}_{L^2((0,\tau)\times\tor)}^2, %\intTauOch{ |\Divh \vuh - \Div \tvu|^2 },
\end{equation}

\begin{equation}\label{lm_u2}
\intTauOch{\abs{\vuh}} 
\aleq h^2+ \penl + h\left( \delta \mu \norm{ \Gradd  \vuh - \Grad \tvu}_{L^2((0,\tau)\times\tor)}^2%\intTauTd{|\Gradd \vuh - \Grad \tvu|^2}
 +  \delta  \frac{\norm{\vuh}_{L^2((0,\tau)\times \Osh)}^2}{\penl} \right).
 \end{equation}
\end{subequations}
\end{Lemma}
\begin{Remark}
We point out that the above estimates \eqref{lm_gradu4}--\eqref{lm_u2} on $ \Och$  
also hold on any general domain $\Ochp$ given by 
\[
\Ochp := \{ x\in \tor | \dist(x,\pd \Of) \aleq h\} ,
\] 
which can be seen as an extension of  $\Och$. 
\end{Remark}
\begin{proof}
Firstly, we use the triangular inequality, the boundedness of $\vrh$, and \eqref{EXTE1} to get \eqref{lm_m3}
\begin{align*}
&  \intTauOch{ \vrh \abs{\vuh}^2  }  \aleq \intTauOch{ \vrh\abs{ \vuh - \tvu}^2  }
+ \intTauOch{ \vrh \abs{ \tvu}^2  }
\aleq \int_0^\tau \RE(\vrh, \vuh| \tvr, \tvu) \dt + h^3.
\end{align*}
Secondly, by Young's inequality we get \eqref{lm_u0}
 \begin{align*}
& \intTau \int_{\Of \setminus \Ofh} {\abs{\vuh}} \dxdt
\leq \intTau \int_{\Of \setminus \Ofh}{\left(\frac{\delta h}{2\penl} \abs{\vuh}^2 + \frac{\penl}{2\delta h}\right) }\dxdt
 \aleq \delta h  \frac{\norm{\vuh}^2_{L^2((0,\tau)\times\Osh)}}{\penl} +\penl. 
\end{align*}
Further, denoting 
$\Ochm = \{x \in \Osh| \dist(x, \pd \Of) \aleq h  \}
$ we can exactly follow the above estimate to get 
 \begin{equation}\label{Ochm}
 \intTauOchm{\abs{\vuh}}
 \aleq \delta h  \frac{\norm{\vuh}^2_{L^2((0,\tau)\times\Osh)}}{\penl} +\penl. 
\end{equation}
Analogously, we have the third estimate \eqref{lm_u1}
\begin{align*}
\intTauOs{\abs{\vuh}} \leq \intTauOsh{ \abs{ \vuh}  }
\aleq \penl^{1/2} \frac{\norm{\vuh}_{L^2((0,\tau)\times\Osh)}}{\penl^{1/2}}
\aleq \frac{\penl}{\delta} + \delta \frac{\norm{\vuh}_{L^2((0,\tau)\times\Osh)}^2}{\penl}.
\end{align*}
Fourthly, we use triangular inequality, Young's inequality, and \eqref{EXTE1} to get \eqref{lm_gradu4}
\begin{align*}
&\intTauOch{ \abs{\Gradd \vuh} }
\leq \intTauOch{ |\Gradd \vuh - \Grad \tvu| } + \intTauOch{ \abs{\Grad \tvu} }
\\ &
\aleq \intTauOch{\left( \delta \mu |\Gradd \vuh - \Grad \tvu|^2 + \frac1{\delta \mu} \right) }
 + h \norm{ \Grad \tvu }_{L^\infty(\tor;\R^{d\times d})} 
\\& \aleq h + \delta \mu \norm{ \Gradd  \vuh - \Grad \tvu}_{L^2((0,\tau)\times\tor)}^2.%\intTauTd{ |\Gradd \vuh - \Grad \tvu|^2 }.
\end{align*}
The same process applies to $\Divh \vuh$ and yields \eqref{lm_div4}.

To show the sixth estimate \eqref{lm_u2}, we apply \eqref{lm_gradu4} and \eqref{Ochm} and find
\begin{align*}
& \intTauOch{\abs{\vuh}} 
\leq   h\intTauOch{|\Gradd \vuh|} + \intTauOchm{\abs{\vuh}} 
 \\& \aleq h^2 +\penl + h\left(\delta \mu \norm{ \Gradd  \vuh - \Grad \tvu}_{L^2((0,\tau)\times\tor)}^2%\intTauTd{|\Gradd \vuh - \Grad \tvu|^2}
 +  \delta  \frac{\norm{\vuh}_{L^2((0,\tau)\times \Osh)}^2}{\penl}\right).
\end{align*}
\end{proof}

%%%%%%%%%%%%%%%%%%%%%%%%%%%%%%%%%%%%%%%%%%%%%%%%%%%%%%%%%%%%%%%%%%%%%%%%
%%%%%%%%%%%%%%%%%%%%%%%%%%%%%%%%%%%%%%%%%%%%%%%%%%%%%%%%%%%%%%%%%%%%%%%%
\begin{Lemma}
Let $v \in L^1(\tor)$, $\vu \in W^{2,\infty}(\Of;\R^d)$, $\tvu$ be given by Definition \ref{def_ES}. Let $(\vrh, \vuh)$ be a solution of the FV method \eqref{VFV} with  $(h,\penl) \in (0,1)^2$. Then
\begin{subequations}
\begin{equation}\label{grad_u2}
\abs{\intTd{ v (\Grad \tvu - D_1 \tvu)}} +\abs{\intTd{ v (\Div \tvu - D_2 \tvu)}} 
\aleq h +\intOch{ \abs{v} }, 
\end{equation}
% \begin{equation}\label{grad_u3}
% \abs{\intTd{ \Gradd \vuh: (\Grad \tvu - \Gradd \Pim \tvu)}}
% \aleq h + \frac{h}{\delta} + \delta \intOch{ |\Gradd \vuh -  \Grad \tvu|^2 } ,
% \end{equation}
% \begin{equation}\label{div_u2}
% \abs{\intTd{ \Divh \vuh\ (\Div \tvu - \Divh \Pim \tvu)}}
% \aleq h + \frac{h}{\delta} + \delta \intOch{ |\Divh \vuh -  \Div \tvu|^2 }
% \end{equation}
%for any constant $\delta>0$, 
where $D_1 = \Gradd \Pim , \Gradh \Pim, \Pie\Grad  $ and $D_2 = \Divh \Pim, \Pid \Div$ with 
\[
\Pie \bfphi = \left( \Pie^{(1)} \phi_1, \cdots, \Pie^{(d)} \phi_d \right), 
\Piei \phi = \sum_{\sigma \in \facesi}  \frac{\mathrm{1}_{D_\sigma}}{|\sigma|} \intSh{\phi}; \quad  \Pid \phi|_{K} := \frac{1_K}{|\ep|} \int_{\ep} \phi \ds, \ep \in \pd D_\sigma \cap K.
\]
Further, it holds 
\begin{equation}\label{grad_u}
\norm{ \Grad \tvu - \Gradd \Pim \tvu}_{L^2(\tor)}^2 + \norm{ \Grad \tvu - \Pie\Grad \tvu}_{L^2(\tor)}^2
+\norm{ \Div \tvu  - \Pid \Div \tvu}_{L^2(\tor;\R^d)}^2
\aleq h.
\end{equation}
\end{subequations}
\end{Lemma}
\begin{proof}
Thanks to the regularity of $\tvu$, the estimate \eqref{grad_u} directly follows from \eqref{grad_u2}  by taking $v = \Grad \tvu - D_1 \tvu$, $\Div \tvu - D_2 \tvu$, respectively. 
To see \eqref{grad_u2} it is enough to show the proof for one term, e.g., $\abs{\intTd{ v (\Grad \tvu - \Gradd \Pim \tvu)}}$, as the rest can be done exactly in the same way. Splitting $\tor$ into three regions $\OI_h, \Oc_h, \OO_h$ we obtain 
\begin{align*}
&\abs{\intTd{ v (\Grad \tvu - \Gradd \Pim \tvu)}}
\\ & \leq
\abs{\intOIh{ v (\Grad \tvu - \Gradd \Pim \tvu)}}
+\abs{\intOch{ v (\Grad \tvu - \Gradd \Pim \tvu)}}
+\abs{\intOOh{ v (\Grad \tvu - \Gradd \Pim \tvu)}}
\\& \aleq
 \norm{v}_{L^1(\OI_h)} \norm{ \Grad \tvu - \Gradd \Pim \tvu}_{L^\infty(\OI_h;\R^{d\times d})}
+\norm{ v }_{L^1(\Oc_h)} \norm{ \Grad \tvu - \Gradd \Pim \tvu}_{L^\infty(\Oc_h;\R^{d\times d})}
+0
\\& \aleq
h \norm{ \vu }_{W^{2,\infty}(\Of;\R^d)}\norm{v}_{L^1(\OI_h)}
+\norm{ \Grad \tvu }_{L^\infty(\tor;\R^{d\times d})}
\norm{v}_{L^1(\Oc_h)}
\aleq h +\norm{ v }_{L^1(\Oc_h)},
\end{align*}
which completes the proof.
\end{proof}

\section{Proof of error estimates}\label{sec_pee}
In this section we show the details of the proof of the error estimates stated in Theorem~\ref{THM:ES}. 

\subsection{Relative energy balance}\label{sec_REB}
We start with deriving the relative energy balance by means of the consistency formulation. 
Noticing the lower regularity of $\tvr$ on $\tor$, i.e. $\tvr \in L^{\infty}((0,T)\times \tor)$, let us modify the density consistency formulation \eqref{CS1-new}: 
\begin{align} \label{CS1-new-1}
e_\vr^{new}(\tau, \TS, h, \phi) := &  \left[ \intTd{ \vrh \phi  } \right]_{t=0}^{t=\tau} -
\int_0^\tau \intOfB{   \vrh \partial_t \phi + \vrh   \vuh \cdot \Grad \phi } \dt,
\end{align}
where $\phi|_{\Os} \in L^{\infty}(\Os), \phi|_{\Of} \in W^{1,\infty}((0,T)\times \Of)$.
Following the decomposition of $e_{\vr}$ in Lemma~\ref{thm_eeD} we have 
\begin{equation}\label{eeD-n}
\begin{aligned}
&e_\vr^{new} =  -\sum_{i=1}^3 E_{i}(\vrh, \phi) - E_4^{new}(\vrh,\phi) + \int_{\tau}^{t_{n+1}} \intOf{ \vrh \vuh \cdot \Grad \phi }, 
\\  
&E_4^{new}(\vrh,\phi) = \intn \intOf{ \vrh \vuh \cdot \Grad \phi } \dt - \intn \intTd{ \vrh \vuh \cdot \Gradh (\Pim \phi)} \dt.
\end{aligned}
\end{equation}

Hence, collecting the energy estimate \eqref{ST},  
the density consistency formulation \eqref{CS1-new-1} with the test function $\phi_1 = \frac12  \abs{\tvu}^2-\Hc'(\tvr)$ and the momentum consistency formulation \eqref{CS2-new} with the test function $\bfvarphi=- \tvu$, we obtain
\begin{align*}%\label{REb}
 &\left[ \RE(\vrh, \vuh| \tvr, \tvu)  \right]_{t=0}^{t=\tau}  +  \intTauOB{\mu \abs{\Gradd \vuh- \Grad \tvu}^2 +  \nu \abs{\Divh \vuh- \Div \tvu}^2 } 
 \br
& \hspace{2cm} + \frac{1}{\penl} \intTauOsh{|\vuh|^2} + \int_0^{\tau} D_{num}\, \dt 
\br 
 & = \intTauOf{\Big(\vrh \pd_t \frac{\abs{\tvu}^2}{2} + \vrh \vuh \cdot \Grad \frac{\abs{\tvu}^2}{2} \Big)} 
\br
&
- \intTauOf{\Big( \vrh \pd_t \Hc'(\tvr) + \vrh \vuh \cdot \Grad \Hc'(\tvr) \Big)} + e_{\vr}^{new} \left(\tau, h, \phi_1 \right)
\br
& - \intTauTdB{\vrh \vuh \pd_t \tvu + \vrh \vuh \otimes \vuh : \Grad \tvu + p_h \Div \tvu } 
\br
& + \intTauTdB{\mu \Gradd \vuh : \Grad \tvu + \nu \Divh \vuh \Div \tvu} + \frac{1}{\penl} \intTauOs{ \vuh \cdot \tvu} + e_{\vm} (\tau,  h, -\tvu)
\br
&+\intTauTd{ \pd_t \Big( \tvr \Hc'(\tvr) - \Hc(\tvr) \Big) }.
\end{align*}

Further, let us split the integrals in the right hand side into two parts: $A=\intOs{\cdot}$ and $B=\intOf{\cdot}$.
Applying \eqref{EXT1} 
% \begin{align*}
% \pdt \tvr=0 \quad \text{and} \quad \tvu = {\bf 0}\quad \mbox{ on } \Os,
% \end{align*}
we have $A = 0$. On the other hand, thanks to
\begin{align*}
&\pd_t \tvr + \Div(\tvr \tvu)=0  \ \mbox{ and } \
\tvr (\pd_t \tvu + \tvu \cdot \Grad \tvu)  + \Grad p(\tvr)  - \mu \Lap \tvu - \nu \Grad \Div \tvu =0 \ \mbox{ on } \ \Of, 
\br
& \intOf{\Big( \tvu \cdot \Grad p(\tvr) + p(\tvr) \Div \tvu \Big)} = 0, \quad \intOf{ \Big( \Div \bS(\Grad\tvu) \cdot \tvu + \bS(\Grad\tvu) : \Grad \tvu \Big) } = 0,
\end{align*}
 we can follow calculations in \cite[Section 3]{FLS_IEE} and obtain $B = \sum_{i=1}^5 R_i$ with
\begin{equation*}%\label{REi}
\begin{aligned}
R_1 & =
\intTauOf{(\vrh-\tvr)  (\tvu - \vuh) \cdot \left(\pd_t \tvu + \tvu \cdot \Grad \tvu  + \frac{\Grad p(\tvr)}{\tvr} \right) },
\\R_2 & =
- \intTauOf{ \vrh (\vuh-\tvu) \otimes (\vuh-\tvu) : \Grad \tvu},
\\ R_3 & =
 -  \mu \intTauOf{ \big( \Gradd \vuh: \Grad \tvu  +  \vuh \cdot \Lap \tvu   \big)},
\\ R_4 & =
- \nu \intTauOf{   \big(   \Divh \vuh \Div \tvu + \vuh \cdot \Grad \Div \tvu  \big) },
\\  R_5 & =
- \intTauOf{\Big(p_h - p'(\tvr) (\vrh -\tvr) -p(\tvr) \Big) \Div \tvu}.
\end{aligned}
\end{equation*}
Altogether, we obtain the following relative energy balance  
\begin{align*}%\label{REb}
 \left[ \RE(\vrh, \vuh| \tvr, \tvu)  \right]_{t=0}^{t=\tau} & +  \intTauOB{\mu \abs{\Gradd \vuh- \Grad \tvu}^2 +  \nu \abs{\Divh \vuh- \Div \tvu}^2 }
\br &
 + \frac{1}{\penl} \intTauOsh{|\vuh|^2}
  =
- \int_0^{\tau} D_{num}\, \dt + e_S  + e_R,
\end{align*}
where $D_{num} \geq 0$ is given in Lemma~\ref{thm_ST}  and 
\begin{equation}\label{ee-SR}
e_S
 =  %e_{\vr} \left(\tau, h, \abs{\tvu}^2/2 \right) +  
 e_{\vr}^{new} \left(\tau, h, \phi_1 \right) +  e_{\vm} (\tau,  h, -\tvu), \ \phi_1 = \frac12  \abs{\tvu}^2-\Hc'(\tvr); \quad 
e_R =  \sum_{i=1}^5 R_i.  
\end{equation}

\subsection{Estimates on $e_S$}
%\subsection{Consistency error $e_S$ for the semi-discrete scheme with bounded density}
\label{sec_ceb}
In this section we estimate new consistency error $e_S$ given in \eqref{ee-SR} for the semi-discrete version of the FV scheme with bounded density. It means that time evolution is not approximated $D_t = \pdt$ and $\vrh \aleq 1$.  
\begin{Lemma}\label{lm_es}
Let $\gamma > 1$ and $(\vrh, \vuh)$ be a solution of the FV method \eqref{VFV} with $(h,\penl) \in (0,1)^2$ and $\alpha > 0$. 
Let  $(\tvr, \tvu)$ be the extended strong solution in the sense of Definition~\ref{def_ES}. 
Suppose that $\vrh$ is uniformly bounded from above by a positive constant, cf.\ \eqref{ass}.  
Then
\begin{equation*}
    \begin{aligned}
    \abs{e_S} \aleq &  h + h^{(1+\viso)/2} + h^{\viso}%+h^{1+\viso}  
  +\frac{\penl}{h} + \frac{h^3}{\penl} %+\penl h^\viso
+ \int_0^\tau \RE(\vrh, \vuh| \tvr, \tvu) \dt + \delta \frac{\norm{\vuh}_{L^2(0,\tau)\times\Osh)}^2}{\penl}  
\\&  + \delta \mu\norm{ \Gradd  \vuh - \Grad \tvu}_{L^2((0,\tau)\times\tor)}^2 + \delta \nu\norm{ \Divh  \vuh - \Div \tvu}_{L^2((0,\tau)\times\tor)}^2.
\end{aligned}
\end{equation*}
\end{Lemma}

\begin{proof}
Let us recall the error terms in $e_S$ given in \eqref{eeD-n} and \eqref{eeD-1}. We study each term separately. 
\begin{itemize}
\item  $E_1(r_h,\phi)$:
To begin, we recall the integration by parts formula for any $f_h \in Q_h$
\begin{equation}\label{ibp}
     \intTd{\pdedgei{f_h}v_h } = -  \intTd{f_h \pdmeshi v_h },
\end{equation}
where $v_h$ is piecewise constant on the dual grid $\{D_\sigma| \sigma \in \facesi\}$, $\pdedgei := \Gradd$ on $D_\sigma$ when $\sigma \in \facesi$ and
$\pdmeshi v_h|_K := \frac{v_h|_{\sigma_{K+}} - v_h|_{\sigma_{K-}}}{h}$. Here, $\sigma_{K\pm} \in \facesi$ are the left and right faces of $K \in \grid$ in the $i^{\rm th}$ direction. %whose barycenters sit on $x_{\sigma_{K\pm}} = x_K \pm \ve_i h/2$. 
Denoting $\Laphi f_h= \pdmeshi \pdedgei f_h$ and noticing $\Laph f_h =\sumi \pdmeshi \pdedgei f_h$ we have
 \begin{align*}
&  \abs{ E_1(r_h,\phi) }
 = \abs{\intTau  \intfaces{  \jump{r_h } \abs{\avs{\vuh}\cdot\vc{n}}\jump{ \Pim \phi} }\dt }
 =
\abs{h \sumi \intTauTd{ r_h  \pdmeshi \Big( \abs{\avs{\uih} } \pdedgei{\Pim \phi}\Big) }  }
\\& =
\abs{h \sumi \intTauTd{ r_h \left( \Laphi{\Pim \phi} \Pim \abs{\avs{\uih} } +  (\Pim \pdedgei{\Pim \phi})  \pdmeshi \abs{\avs{\uih} }   \right)} }
\\& \aleq 
h  \norm{ r_h}_{L^2((0,\tau)\times{\tor\setminus \Och})} \norm{ \phi}_{L^\infty(0,\tau;W^{2,\infty}({\tor\setminus \Och}))}  \left( \norm{ \vuh}_{L^2((0,\tau)\times{\tor\setminus \Och})}  +\norm{ \Gradd \vuh}_{L^2((0,\tau)\times{\tor\setminus \Och})}  \right)
\\& \quad 
+\abs{h \sumi \intTauOch{ r_h \left( \Laphi{\Pim \phi} \Pim \abs{\avs{\uih} } +  (\Pim \pdedgei{\Pim \phi})  \pdmeshi \abs{\avs{\uih} }   \right)} }
\\& \aleq  h \norm{ \phi}_{L^\infty(0,\tau;W^{2,\infty}({\tor\setminus \Och}))} +
h  \intTauOchp{ \abs{r_h \vuh} }  \norm{\Laph{\Pim \phi}}_{L^\infty ((0,\tau) \times\Och)}   \\& \quad +h \intTauOchp{\abs{ r_h \pdmeshi \abs{\avs{\uih} }} } \norm{\Gradd{\Pim \phi}}_{L^\infty ((0,\tau) \times\Och)}
\end{align*}
which implies that for $(r_h,\phi) = (\vrh, \phi_1)$ and $(\vrh \vuh, \tvu)$ we have 
\begin{align*}
&|E_1(\vrh,\phi_1) | \aleq h+ h^{-1} \intTauOchp{\abs{\vuh}} + \intTauOchp{\abs{\Gradd \vuh}},
\\&|E_1(\vrh \vuh,\tvu) | \aleq h+  \intTauOchp{\vrh\abs{\vuh}^2}.
\end{align*}

\item $E_2(r_h,\phi)$: Let us definite by $\facesC$ the set of all faces inside $\Och$. Then applying triangular inequality,  H\"older's inequality and the boundedness of $\vrh$ we have 
\begin{align*}
& \abs{E_2(\vrh, \phi_1)} %= \abs{ \intTau  \intfaces{  \jump{r_h}  \jump{\vuh} \cdot \vc{n} \jump{\Pim \phi} }  \dt} \\&
\aleq  h   \intTau  \intfacesIO{ \abs{\jump{\vuh}}} \dt 
+\int_0^{\tau}  \intfacesC{\abs{\jump{\vuh} \cdot \vn}}\dt +
  \aleq h +  \intTauOchp{|\Gradd \vuh|}
\end{align*}
and 
\begin{align*}
& \abs{E_2(\vrh \vuh,\tvu)}
\aleq   h \norm{\tvu}_{L^{\infty}(0,\tau;W^{1,\infty}(\tor))}  \left( \int_0^{\tau}  \intfaces{ \abs{\jump{\vuh}}^2 }\dt\right)^{1/2}   \left( \int_0^{\tau}  \intfaces{ \abs{\jump{\vrh\vuh}}^2 }\dt\right)^{1/2}
\\&
\aleq   h^{3/2}  \left( \int_0^{\tau}  \intfaces{ \avs{\abs{\vrh\vuh}^2} }\dt\right)^{1/2} 
\aleq h  \left( \intTauTd{ \vrh\abs{\vuh}^2} \right)^{1/2}  
\aleq  h.
\end{align*}

\item $E_3(r_h,\phi)$: 
Using triangular inequality and \eqref{es-d1} we have
\begin{align*}
&\abs{ E_3(\vrh,\phi_1)} = h^\viso \abs{ \intTau \intfaces{\jump{\vrh} \jump{\Pim \phi} }\dt }
\\& \aleq
h^\viso  \intTau \intfacesC{ 1 }\dt 
+ h^{1+\viso} \intTau \intfacesIO{|\jump{\vrh}|} \dt  
\aleq  h^{\viso}+  h^{(1+\viso)/2} .
\end{align*}
Moreover, using \eqref{ibp}, \eqref{EXTE4}, and \eqref{lm_u2} we have
\begin{align*}
&  \abs{E_3(\vrh \vuh,\tvu)}
= h^{\viso+1} \abs{ \intTauTd{\vrh \vuh \cdot \Laph \Pim \tvu } }
\aleq h^{\viso+1} 
+  h^\viso \intTauOch{| \vuh | }.
\end{align*}

\item $E_4(\vrh \vuh,\tvu)$: By H\"older's inequality and \eqref{grad_u2} we have 
\begin{align*}
 \abs{E_4(\vrh \vuh,\tvu)}  
 \aleq h  + \intTauOch{\vrh|\vuh|^2} . 
\end{align*}

\item $E_4^{new}(\vrh, \phi_1)$: Analogously, we have
\begin{align*}
&
\abs{E_4^{new}(\vrh, \phi_1)} 
= \Big| \intTauOIh{ \vrh \vuh \cdot (\Gradh (\Pim \phi_1) -  \Grad \phi_1) } 
+\intTauOOh{ \vrh \vuh \cdot \Gradh (\Pim \phi_1)}  
\\& 
+ \intTauOch{ \vrh \vuh \cdot \Gradh (\Pim \phi_1)}   
- \intTau \int_{\Of \setminus \OI_h}{ \vrh \vuh \cdot \Grad \phi_1 } \dxdt
  \Big|
\\&
\aleq h + \intTauOsh{ \abs{\vuh} }   + (h^{-1}+1) \intTauOch{\abs{\vuh}} ,
\end{align*}
where we have used the fact $\OO_h \subset \Osh$ and $\Of \setminus \OI_h \subset \Och$.  
\item Penalty term $E_{\penl}$:
\begin{align*}
\abs{E_{\penl}(\tvu)}  = \Bigabs{ \frac{1}{\penl} \int_{0}^{\tau} \int_{\Osh \setminus \Os} \vuh \cdot \tvu  \dxdt}
\aleq  \abs{ \frac{1}{\penl} \int_0^{\tau}\int_{\Osh\setminus \Os }{ \left(\delta \vuh^2 + \frac1{\delta} \tvu^2\right) } \dxdt}
\aleq \delta \frac{\norm{\vuh}_{L^2(0,\tau)\times\Osh)}^2}{\penl} + \frac{h^3}{\penl} .
\end{align*}

\item Viscosity terms $E_{\Grad \vu}$ and pressure term $ E_{p}$:
Recalling \eqref{grad_u2} 
and boundedness of the density \eqref{ass}, we have
\begin{align*}
\abs{E_{\Grad \vu}(\tvu)} \aleq h+ \intTauOch{\big( |\Gradd\vuh| + |\Divh \vuh| \big)}, \quad 
\abs{E_{p}(\tvu)} \aleq h + \intTauOch{p_h} \aleq h.
\end{align*}
 \end{itemize}

Consequently, collecting all the above estimates we derive
\begin{align*}
  &  \abs{e_S} \aleq  h+ h^{(1+\viso)/2} + h^{\viso} + \intTauOsh{ \abs{\vuh} } 
  \\& + (h^\alpha + 1 +h^{-1}) \intTauOchp{|\vuh|} + \intTauOch{ (|\Gradd\vuh|+ \vrh|\vuh|^2  + |\Divh \vuh| )}.
\end{align*}
Finally, applying Lemma~\ref{lm_m} finishes the proof.
\end{proof}

\subsection{Estimates on $e_{R}$}\label{sec_er}
%%%%%%%%%%%%%%%%%%%%%%%%%%%%%%%%%%%%%%%%%%%%%%%%%%%%%%%%%%%%%%%%%%%%%%%%
 In this section we estimate $e_R$ defined in \eqref{ee-SR}. We begin with the following lemma. 
\begin{Lemma}\label{lmSP2}
Let $\gamma >1$ and $(\vrh, \vuh)$ be a solution obtained by the FV method \eqref{VFV} with $(h,\penl) \in(0,1)^2$. Let  $(\tvr, \tvu)$ be the extended strong solution in the sense of Definition~\ref{def_ES}. Then the following holds 
\begin{align*}
\norm{ \vuh - \tvu}_{L^2(\tor;\R^d)}^2   \aleq   h+\norm{ \Gradd  \vuh - \Grad \tvu}_{L^2(\tor)}^2  + \RE(\vrh, \vuh| \tvr, \tvu).
\end{align*}
\end{Lemma}
\begin{proof}
Firstly, by setting $f_h =\vuh -\Pim \tvu$ in Lemma \ref{lmSP} we know that
\[
\norm{\vuh -\Pim \tvu}_{L^2(\tor;\R^d)}^2  \aleq  \left( \norm{  \Gradd (\vuh -\Pim \tvu)}_{L^2(\tor;\R^d)}^2  + \intTd{\vrh |\vuh -\Pim \tvu|^2 } \right).
\]
%where the constant $C_1$ depends on the total mass $M$, initial energy $E_0$ and $\gamma$.
Then by the triangular inequality, projection error estimate and \eqref{grad_u} we derive
\begin{align*}
& \norm{ \vuh - \tvu}_{L^2(\tor;\R^d)} ^2
\leq
 \norm{ \vuh - \Pim \tvu}_{L^2(\tor;\R^d)} ^2 + \norm{ \Pim \tvu - \tvu}_{L^2(\tor;\R^d)}^2
\\& \aleq
 \norm{  \Gradd  (\vuh -\Pim \tvu)}_{L^2(\tor;\R^d)}^2  + \intTd{\vrh |\vuh -\Pim \tvu|^2 }    + \left(h   \norm{\Grad \tvu }_{ L^\infty(\tor;\R^{d \times d})}\right)^2
\\& \aleq
  \norm{ \Gradd  \vuh - \Grad \tvu}_{L^2(\tor;\R^{d \times d})}^2  +  \intTd{\vrh  |\vuh -\tvu|^2} 
  \\& \quad
  + \norm{ \Grad \tvu - \Gradd  \Pim \tvu}_{L^2(\tor;\R^{d \times d})}^2  + \intTd{\vrh |\Pim \tvu -\tvu|^2 }     + h^2   
%\\&\aleq
%	   \norm{ \Gradd  \vuh - \Grad \tvu }_{L^2(\tor;\R^{d \times d})}^2  + \RE(\vrh, \vuh| \tvr, \tvu)
% \\& \quad +    \left( h^2 \norm{\vu}_{ W^{2,\infty}(\Of;\R^d)}^2 + h \norm{\tvu }_{W^{1,\infty}(\Och;\R^d)}+  h^2 \norm{\Grad \tvu}_{L^\infty(\tor; \R^{d\times d})}^2  \intTd{\vrh}  \right)
% + h^2    
 %  \\&
 % =    \left( \norm{ \Gradd \vuh - \Grad \tvu }_{L^2(\tor;\R^{d \times d})}^2  + \intTd{\vrh  |\vuh -\tvu|^2 }    \right)
 % +   (h^2 + h) 
  \\&
 \aleq    \norm{ \Gradd \vuh - \Grad \tvu }_{L^2(\tor;\R^{d \times d})}^2  + \RE(\vrh, \vuh| \tvr, \tvu)   + h  ,
\end{align*}
which completes the proof. 
\end{proof}

We are now ready to show the estimate of $e_R$ given in \eqref{ee-SR}. 
\begin{Lemma}\label{lm_er}
Let $\gamma > 1$ and $(\vrh, \vuh)$ be a solution of the FV method \eqref{VFV} with $(h,\penl) \in (0,1)^2$.  
Let  $(\tvr, \tvu)$ be the extended strong solution in the sense of Definition~\ref{def_ES}. 
Suppose that $\vrh$ is uniformly bounded from above by a positive constant. 
Then
\begin{multline*}
    \abs{e_R} \aleq   h 
  +\frac{\penl}{h} 
+ \int_0^\tau \RE(\vrh, \vuh| \tvr, \tvu) \dt + \delta \frac{\norm{\vuh}_{L^2(0,\tau)\times\Osh)}^2}{\penl}  
\\  + \delta \mu\norm{ \Gradd  \vuh - \Grad \tvu}_{L^2((0,\tau)\times\tor)}^2 + \delta \nu\norm{ \Divh  \vuh - \Div \tvu}_{L^2((0,\tau)\times\tor)}^2.
\end{multline*}
\end{Lemma}
%%%%%%%%%%%%%%%%%%%%%%%%%%%%%%%%%%%%%%%%%%%%%%%%%%%%%%%%%%%%%%%%%%%%%%%%
\begin{proof}
Firstly, thanks to the regularity of the strong solution we can control $R_2+R_5$ by means of relative energy
\begin{align*}
\abs{R_2+R_5} \aleq  \int_0^\tau  \RE(\vrh, \vuh| \tvr, \tvu)  \dt.
\end{align*}
Secondly, thanks to Lemmas~\ref{difru} and \ref{lmSP2} we directly obtain the estimate of $R_1$ 
\begin{equation*}
\abs{R_1}
\aleq
h +\int_0^{\tau} \RE(\vrh, \vuh| \tvr, \tvu) \dt
+\delta \norm{ \Gradd  \vuh - \Grad \tvu}_{L^2((0,\tau)\times\tor)}^2.
\end{equation*}
%%%%%%%%%%%%%%%%%%%%%%%%%%%%%%%%%%%%%%%%%%%%%%%%%%%%%%%%%%%%%%%%%%%%%%%%
%%%%%%%%%%%%%%%%%%%%%%%%%%%%%%%%%%%%%%%%%%%%%%%%%%%%%%%%%%%%%%%%%%%%%%%%
Thirdly, by the Gauss theorem we know that
\begin{align*}
    \int_K{ \Div \vU} \dx = \int_K{ \Divmesh \Pie \vU}\dx  \mbox{ for any } K \in \grid,
\end{align*}
where $\Divmesh =\sumi \pdmeshi$ and $\pdmeshi$ is defined in the previous subsection. 
Thus we have
\[\intOfh{\vuh \cdot \Lap \tvu } = \intOfh{\vuh \cdot \Divmesh \Pie \Grad \tvu}.\]
Using this identity and $\tvu|_{\Os} = 0$ we observe
\begin{align*}
&  R_3% =  -  \mu \intTauOf{ \big( \Gradd \vuh: \Grad \tvu  +  \vuh \cdot \Lap \tvu   \big)} 
= -  \mu \intTauTd{  \Gradd \vuh: \Grad \tvu } -  \mu \intTauOf{  \vuh \cdot \Lap \tvu  }
\\& = -  \mu \intTauTd{  \Gradd \vuh: \Grad \tvu} 
-\mu \intTauOfh{ \vuh \cdot \Divmesh \Pie \Grad \tvu }
    - \mu \int_0^\tau \int_{\Of \setminus \Ofh}{ \vuh \cdot \Lap \tvu }\dxdt
\\& =-  \mu \intTauTd{  \big( \Gradd \vuh: \Grad \tvu  +  \vuh \cdot \Divmesh \Pie \Grad \tvu   \big)} +I_1 = I_1+I_2,
\end{align*}
where 
\begin{align*}
I_1 &=  \mu \intTauOsh{ \vuh \cdot \Divmesh \Pie \Grad \tvu }
    -\mu \int_0^\tau \int_{\Of \setminus \Ofh}{ \vuh \cdot \Lap \tvu }\dxdt,
\\ I_2 &=  \mu \intTauTd{\Gradd \vuh: \big(  \Pie \Grad \tvu    - \Grad \tvu   \big)} = 0.
\end{align*}
Noticing that $\Divmesh \Pie \Grad \tvu|_{\OO_h} =0$, $|\Divmesh \Pie \Grad \tvu|_{\Och} \aleq h^{-1}$ and $|\Lap \tvu|_{\Of \setminus \Ofh}\aleq 1$, we can control $R_3 = I_1$ by  \eqref{lm_u2} as
\begin{align*}
  &   \abs{R_3} \aleq (1+h^{-1}) \intTauOch{\abs{\vuh}} 
     \aleq  h+\frac{\penl}{h} + \delta \mu \norm{ \Gradd  \vuh - \Grad \tvu}_{L^2((0,\tau)\times\tor)}^2   +   \delta \frac{\norm{\vuh}_{L^2((0,\tau)\times\Osh)}^2}{\penl} .
\end{align*}
%Moreover, the term $I_2$ can be controlled by \eqref{grad_u2} and \eqref{lm_gradu4} as
%\begin{align*}
%    &|I_2| 
% \aleq h +    \intTauOch{\abs{ \Gradd \vuh} } 
% \aleq  h +\delta \mu \intTauOch{\abs{\Gradd \vuh - \Grad \tvu}^2}.
%\end{align*}
%Consequently, we have 
%\[
%\abs{R_3} \aleq h+\frac{\penl}{h} + \delta \intTauOch{\abs{\Gradd \vuh - \Grad \tvu}^2}  +   \delta \frac{\norm{\vuh}_{L^2((0,\tau)\times\Osh)}^2}{\penl}.
%\]
To estimate $R_4$ we use the Gauss theorem again for $D_{\sigma}, \sigma \in \facesi$
\begin{align*}
    \intDh{ \pd_i  \Div \tvu } = \intDh{ \pdedgei  \Pid \Div \tvu  }, %\quad \forall \sigma \in \facesi.
\end{align*}
where $\Pid \phi|_{K} := \frac{1_K}{|\ep|} \int_{\ep} \phi \ds$
for $\ep \in \pd D_\sigma \cap K$. Thanks to this equality and the integration by parts \eqref{ibp} and the equality $\Divh \vuh= \sumi \pdmeshi \avsi{\uih}$ we have
\begin{align*}
&    \intOf{ \vuh \cdot \Grad \Div \tvu   }
=   \intOfh{ \vuh \cdot \Grad \Div \tvu  } +    \int_{\Of \setminus \Ofh}{ \vuh \cdot \Grad \Div \tvu  }
\\&=
 \intOfh{ \avs{\vuh} \cdot \Grad \Div \tvu  }
+\underbrace{ \intOfh{ (\vuh-\avs{\vuh}) \cdot \Grad \Div \tvu  }
+   \int_{\Of \setminus \Ofh}{ \vuh \cdot \Grad \Div \tvu  }\dx }_{:= J_1}
 \\&=
\sumi \intTd{ \avsi{\uih}  \pdedgei  \Pid  \Div \tvu  }  +J_1 \underbrace{-
\sumi \intOsh{ \avsi{\uih}  \pdedgei  \Pid  \Div \tvu  } }_{:=J_2}
 \\&=
    - \sumi \intTd{ \pdmeshi \avsi{\uih}  \Pid    \Div \tvu  }  +J_1 +J_2
=
   -  \intTd{ \Divh \vuh  \Pid    \Div \tvu  }  +J_1 +J_2.
\end{align*}
Thus, together with \eqref{grad_u2}, \eqref{lm_u0}, \eqref{lm_div4}, and \eqref{lm_u2} we have 
\begin{align*}
&    |R_4| = \nu \abs{ \intTauTd{\Divh \vuh  (\Pid \Div \tvu -\Div \tvu)} -\intTau(J_1 +J_2)\dt}
\\&
\aleq h+ \intOch{|\Divh \vuh|} + h\intTauOfh{|\Gradd \vuh|} + \intTau \int_{\Of \setminus \Ofh}|\vuh| \dxdt + h^{-1}\intTauOch{|\vuh|}
\\& \aleq \penl + h + \frac{\penl}{h}  + \delta \frac{\norm{\vuh}_{L^2((0,\tau)\times\Osh)}^2}{\penl} 
 +\delta \mu\norm{ \Gradd  \vuh - \Grad \tvu}_{L^2((0,\tau)\times\tor)}^2 +  \delta\nu\norm{ \Divh  \vuh - \Div \tvu}_{L^2((0,\tau)\times\tor)}^2.
\end{align*}
Consequently, collecting the estimates of $R_i$-terms completes the proof. 
\end{proof}

\end{document}